\documentclass[11pt]{amsart}

\usepackage[utf8]{inputenc}
\usepackage[T1]{fontenc}
\usepackage[english]{babel}

\usepackage{amsaddr}

\usepackage{amsthm}
\usepackage{amsmath}
\usepackage{amssymb}
\usepackage{mathrsfs}
\usepackage{bbold}
\usepackage{graphicx}
\usepackage{listings}
\usepackage{mathrsfs}
\usepackage{enumerate}
\usepackage{pict2e}
\usepackage{stmaryrd}
\usepackage{mathtools}

\usepackage[all,cmtip]{xy}

\usepackage{multirow}

\usepackage{color}
\definecolor{marin}{rgb}   {0.,   0.1,   0.5} 
\definecolor{rouge}{rgb}   {0.8,   0.,   0.} 
\definecolor{sepia}{rgb}   {0.4,   0.25,   0.} 
\definecolor{mag}{rgb}   {0.3,   0,   0.3} 

\usepackage[colorlinks,citecolor=sepia,linkcolor=marin,urlcolor=mag ,bookmarksopen,bookmarksnumbered]{hyperref}

\newtheorem{theorem}{Theorem}[section]
\newtheorem{corollary}[theorem]{Corollary}
\newtheorem{lemma}[theorem]{Lemma}
\newtheorem{proposition}[theorem]{Proposition}
\newtheorem{definition}[theorem]{Definition}
\newtheorem{remark}[theorem]{Remark}
\newtheorem{example}[theorem]{Example}

\begin{document}

\title[Birkhoff normal forms for Hamiltonian PDEs in their energy space]{Birkhoff normal forms for Hamiltonian PDEs in their energy space}
\author{Joackim Bernier and Beno\^it Gr\'ebert}

 \address{\small{Laboratoire de Math\'ematiques Jean Leray, Universit\'e de Nantes, UMR CNRS 6629\\
2 rue de la Houssini\`ere \\
44322 Nantes Cedex 03, France}}



\email{joackim.bernier@univ-nantes.fr}
\email{benoit.grebert@univ-nantes.fr}

\begin{abstract} We study the long time behavior of small solutions of semi-linear dispersive Hamiltonian partial differential equations on confined domains. Provided that the system enjoys a new non-resonance condition and a strong enough energy estimate, we prove that its low super-actions are almost preserved for very long times. Roughly speaking, it means that, to exchange energy, modes have to oscillate at the same frequency. Contrary to the  previous existing results, we do not require the solutions to be especially smooth. They only have to live in the energy space. We apply our result to nonlinear Klein-Gordon equations in dimension $d=1$ and nonlinear Schr\"odinger equations in dimension $d\leq 2$.
\end{abstract} 
\maketitle

\setcounter{tocdepth}{1} 
\tableofcontents

\section{Introduction}

The theory of normal forms for Hamiltonian PDEs has been very popular over the last twenty years, with great success both in non-resonant cases (stability over long periods of small and regular solutions \cite{Bou96,Bam99,Bou00,Bam03,BG06,BDGS07,GIP,Del12,FGL13,YZ14,BD17,FI19,KillBill,BMP20,FI20,KtA}) and in resonant cases (weak turbulence phenomena  \cite{CKSTT10,CF12,GG12,GK15}, beatings phenomena \cite{GV,GT12,HP17} or chaotic phenomena \cite{GGMP21}). However, this theory was only applied for the moment in very regular function spaces, essentially the Sobolev spaces $H^s$ for $s$ very large.

An emblematic result of this technique, demonstrated in \cite{BG06}, states that given a non-resonant semi-linear Hamiltonian PDE, with a non-linear part with a tame property, given an integer $r$, there exists $s_0(r)$ such that, any solution in $H^s$ with $s$ larger than $s_0(r)$ of sufficiently small initial size $\varepsilon$, is stable in $H^s$, for very long times of the order $\varepsilon^{-r}$, in the following sense: the amplitudes of its modes (or its super-actions) are almost preserved and thus the solution remains small in $H^s$. The main flaw of this result lies in the constraint $s\geq s_0(r)$ which is far from being negligible since, in the best cases, the nonlinear Schr\"odinger equations (called in the following NLS) on the torus for example, $s_0(r)\sim r$ (see \cite{BMP20}). This restriction is all the more problematical that numerical experiments strongly suggest that it is irrelevant (see for example the numerical experiments in \cite{CHL08a,CHL08b} dealing with non smooth solutions of nonlinear Klein--Gordon equations). In the meantime a constant effort has been developed to lower the degree of regularity at which the equation is well posed (see e.g. \cite{Vla84,Bou93,Bou99,Caz03,BGT05}) and to compute accurately its non-smooth solutions (see e.g. \cite{HS17,ORS20}).  
  
In this paper we develop a Birkhoff normal form technique in low regularity. Considering small solutions in the energy space, it is clear that the energy norm remains small as long as the solution exists. However a relevant question consists in estimating the exchanges of energy between modes. Actually, our result is a little bit weaker than the classical Birkhoff normal form result since 
 it concerns essentially only the low modes of the solution: schematically, given $r$ and $N$ if the initial data is small enough in the energy space then we prove that the amplitudes of the first $N$ modes of the solution remain almost unchanged over times of the order $\varepsilon^{-r}$. Nevertheless, if the initial datum is a little bit smoother then the high modes  are also almost preserved (i.e. $N=+\infty$; see Corollary \ref{cor_whaou} for a concrete example). This result is obtained by separating the dynamics of the low modes from those of the high modes which, themselves, are controlled only because of energy conservation.  This separation is obtained thanks to a new non-resonance condition which is really the key to this work and  to which we will return later.

\subsection{Main result}
In this section we give a heuristic version of our main result which is stated precisely and rigorously in Theorem \ref{thm_main_dyn}. 
We consider a Hamiltonian PDE that can be written
$$\partial_t u=J\nabla H(u)$$
where $J$ denotes a reasonable Poisson structure, and $H$ is a smooth Hamiltonian defined on the energy space $\mathcal E$. We assume that $\mathcal E$ is a Hilbert space, admitting a Hilbertian basis $(e_n)_{n\in \mathbf{N}_d}$ where $\mathbf{N}_d \subset \mathbb{Z}^d$, in such a way the decomposition $u=\sum_{n\in \mathbf{N}_d}u_ne_n $ allows an identification between $\mathcal E$ and a discrete Sobolev space  $h^s(\mathbf{N}_d;\mathbb{C})$ (defined in \eqref{eq_def_hs}) for some $s>0$  and the Hamiltonian PDE reads 
$$
\partial_t u_n =-i \partial_{\bar u_n}H(u),\quad n\in \mathbf{N}_d.
$$
We  assume that  $H=Z_2+P$ with
$$
Z_2=\frac12\sum_{n\in\mathbf{N}_d}\omega_n |u_n|^2,\quad \text{with }\omega\equiv(\omega_n)_{n\in\mathbf{N}_d }\in \mathbb{R}_{+}^{\mathbf{N}_d}
$$
and $P\in C^1(\mathcal E;\mathbb{R})$ a regular Hamiltonian, i.e. $\nabla P$ maps continuously $\mathcal E$ into itself (this condition can be slightly relaxed, for instance in the case of NLS in 2-d). We further assume that $P$ is of order $p$ at the origin, i.e. there exists $C$ such that $|P(u)|\leq C\| u\|_{\mathcal{E}}^p$ for $\|u\|_{\mathcal{E}}$ small enough. Finally we assume that $H$ (or another constant of the motion) controls and is well controlled by the energy norm: there exists $\Lambda \geq 1$ such that
\begin{equation}\label{Lambda}\Lambda^{-1}\| u\|_{\mathcal{E}} \leq H(u)\leq \Lambda \|u\|_{\mathcal{E}},\quad \text{for all }u\in{\mathcal{E}} \text{ small enough.}\end{equation}
Concerning the frequency vector $\omega$ we assume that it is strongly non resonant in the following sense.
\begin{definition}[Strong non-resonance]
\label{def_SNR} Let $d\geq 1$, $\mathbf{N}_d \subset \mathbb{Z}^d$ and $\omega \in \mathbb{R}^{\mathbf{N}_d}$.

 The frequencies $\omega$ are strongly non resonant, if for all $r>0$ there exists $\gamma_r>0$, $\alpha_r>0$ such that for all $r_\star \leq r$, all $\ell_1,\dots,\ell_{r_\star} \in \mathbb{Z}^*$, all $n \in \mathbf{N}_d^{r_\star}$ injective, provided that $|\ell_1|+\cdots+|\ell_{r_\star}| \leq r$ and $\langle n_1 \rangle \leq \cdots \leq \langle n_{r_\star} \rangle$ we have
\begin{equation}
\label{eq_est_NR}
|\ell_1 \omega_{n_1} + \cdots + \ell_{r_\star} \omega_{n_{r_\star}} | \geq \gamma_{r} \langle n_1 \rangle^{-\alpha_r}.
\end{equation}
\end{definition}
Note that this definition is not well suited to deal with multiplicities (here $\omega$ has to be injective). Therefore it is extended in Definition \ref{def_nonres} (which is heavier).

\begin{theorem}[Heuristic]\label{sketch}
Fix $r\geq p$, there exist $\beta_r>0$ and $\varepsilon_0>0$  such that if $u^{(0)} \in{\mathcal{E}}$ satisfies $\| u^{(0)} \|_{{\mathcal{E}}}=\varepsilon<\varepsilon_0$ then the Cauchy problem $\left\{\begin{array}{ll}\partial_t u=J\nabla H(u)\\u(0)=u^{(0)} \end{array}\right.$ admits an unique global solution in $\mathcal E$ and there exits $C_r$ such that
\begin{equation}\label{mainest}
|t|\leq \varepsilon^{-r} \quad \implies\quad \forall n\in \mathbf{N}_d,\ \Big ||u_n(t)|^2-|u_n(0)|^2\Big|\leq C_r\langle n\rangle^{\beta_r}\varepsilon^p.
\end{equation}
\end{theorem}
The rigorous statement is given in Theorem \ref{thm_main_dyn} and we provide a scheme of prove in the subsection \ref{subsec_scheme}. Concrete examples of applications of this theorem are given in Theorem \ref{thm_main_KG}, \ref{thm_Dir_1d}, \ref{thm_per_1d}, \ref{thm_per_2d}.
\subsection{Comments}

\noindent $\bullet$  {\it Partially resonant case}: in Theorem \ref{thm_main_dyn} we consider a more general setting including  the (partially) resonant case, i.e. when it can happen that $\omega_n=\omega_m$ for $n\neq m$. In that case we can only control the super action $J_n=\sum_{\omega_m=\omega_n}|u_m|^2$. We could also consider clusters of close frequencies in the spirit of \cite{BDGS07} or \cite{GIP}.

\noindent $\bullet$ {\it Low regularity}: The main novelty of this theorem is that it applies to solutions of low regularity, namely solution in the energy space. So we can consider non smooth initial data but, also, we can consider PDEs with coefficients that are not smooth. Typically in nonlinear Schr\"odinger equations we can consider nonlinearities of the type $g(x)|u|^2u$ with $g$ only in the $C^1$ class or non-smooth multiplicative potentials. This allow us to provide (in Theorem \ref{thm_per_1d}), to the best of our knowledge, the first Birkhoff normal form Theorem for nonlinear Schr\"odinger equations both with a multiplicative potential and periodic boundary conditions. On the other hand, this allows to consider Dirichlet boundary conditions without parity restriction on the equation. For instance we consider the Klein Gordon equation with Dirichlet boundary conditions with quadratic, or more generally even, nonlinearity.

\noindent $\bullet$ {\it Unbounded solutions}: we stress that in the case of NLS in dimension $2$ with periodic boundary conditions (see \eqref{eq_NLS_2}) the energy space $H^1$ is not included in $L^\infty$. Thus the solutions whose Fourier modes we control are not bounded, which is quite amazing.

\noindent $\bullet$ {\it Free behavior near the boundary}: our technique let a lot of freedom to the solutions near the boundary. For example, to deal with PDEs with homogeneous Dirichlet boundary conditions, we only have to assume that the solution vanishes on the boundary whereas, with the classical technique, the solution almost have to be odd (i.e. even derivatives have to vanish on the boundary up to order $s\geq s_0(r)$; see e.g. the compatibility condition (2.4) in \cite{Bam03}).

\noindent $\bullet$ {\it Nekhoroshev in finite regularity}: optimizing $r$ with respect to $\varepsilon$, we could get\footnote{unfortunately, it would be quite technical and would require to optimize all the estimate with respect $r$, thus we have not do it.} a stability result in the energy space for super-polynomial times with respect to $\varepsilon^{-1}$. On the contrary, in the usual setting, due to the constraint $s\geq s_0(r)$, the only way to reach stability for super-polynomial times was to consider analytic or Gevrey solutions (see \cite{FG13,BMP20}) or, at least, to optimize $s$ with respect to $\varepsilon$ as in \cite{BMP20}.

\noindent $\bullet$ {\it Finite number of modes}: because of the term $\langle n\rangle^{\beta_r}$ in \eqref{mainest}, our result essentially says that we can control finite number of modes during very long times. Namely instead of \eqref{mainest} we could says: given $r\geq p$ and $N\geq 1$ there exists $C_{r,N}>0$ such that
$$
|t|\leq \varepsilon^{-r} \quad \mathrm{and} \quad \langle n\rangle\leq N\quad \implies\quad  \Big ||u_n(t)|^2-|u_n(0)|^2\Big|\leq C_{r,N}\varepsilon^p.
$$ 
Nevertheless, we notice that in our result the number $N$ of modes we can control depends on the size of the in initial datum $\varepsilon$: using \eqref{mainest} we have $N\equiv N_\varepsilon\sim \varepsilon^{\frac{-p+2}{\beta_r}}$. This optimization is especially useful to describe solutions whose initial datum is little bit smoother (see e.g. Corollary \ref{cor_whaou}).

\noindent $\bullet$ In our {\it new non resonant condition} given by \eqref{eq_est_NR}, we ask for a control of the small divisors with respect to the smallest index involved. Clearly it is a much stronger condition than the one usually used where we ask for a control of the same small divisors with respect to the third largest index involved (see \cite{BG06}). Surprisingly, this stronger condition is  often verified. Indeed a control of the small divisor with respect to the largest index involved implies \eqref{eq_est_NR} provided that the high frequencies are close to integer values (see Proposition \ref{prop_weak_become_strong}).

\noindent $\bullet$ The {\it ellipticity condition} \eqref{Lambda} is used to control the energy norm and thus to ensure the global well posedness in the energy space. Indeed, our method allows to control only a finite number of modes and therefore it is necessary to control, by another argument, the norm of the solution. On the contrary, the standard Birkhoff normal form method provides a control of the norm of the solution, but in high regularity. This is a constraining restriction if we want to consider multidimensional PDEs.

\noindent $\bullet$ {\it Admissibility of the PDE}: The main restrictions imposed to be able to apply our result are the ellipticity condition \eqref{Lambda}, the non-resonance condition \eqref{eq_est_NR} and an extra property even more constraining than ellipticity: we ask for a certain norm, $\|\cdot\|_{\tilde{\mathcal{E}}}$, built from the energy norm to be a norm of algebra:
$\|uv\|_{\tilde{\mathcal{E}}}\lesssim \|u\|_{\tilde{\mathcal{E}}} \|v\|_{\tilde{\mathcal{E}}}$. Roughly speaking, $\|\cdot\|_{\tilde{\mathcal{E}}}$ takes into account the regularization property of the linear part of the PDE. For instance for NLS, in 1-d, $\| u\|_{\tilde{\mathcal{E}}}=\|u\|_{{\mathcal{E}}}$, while for the Klein Gordon equation, $\|\cdot\|_{\tilde{\mathcal{E}}}$ integrates the fact that the equation is $1/2$ regularizing. In our technical statements, this condition writes $s>d/2-q$ where $s$ denotes the Sobolev exponent of the energy space and $q$ quantifies the regularizing effect (as illustrated below on the example the NLS in dimension $2$, there is a trick to deal with the limit case $s=d/2-q$). As said before, the condition \eqref{eq_est_NR} is often satisfied in standard examples of Hamiltonian PDEs and paradoxically it is the condition \eqref{Lambda} and the condition on the energy norm that most restrict the field of application.

\subsection{Applications} 
 In this paper, as representative examples of what our result can achieve,  we consider the Klein Gordon equation in 1d with Dirichlet boundary conditions and the nonlinear Schr\"odinger equations in 1d and 2d  with both periodic and Dirichlet boundary conditions. The proofs are given in Section \ref{sec_proofs} (excepted the probabilistic results which are all proven in Section \ref{sec_NR}).\\
Clearly the result also apply to other equations, for instance the beam equation, and other manifolds, for instance  a  sphere or a Zoll manifold. Nevertheless, our purpose is not to exhaust all the possible applications but rather to choose a few to illustrate our method. Note that, in particular, even if in all our applications\footnote{excepted in Remark \ref{rem_KG_per}} the eigenvalues of the linearized vector fields are simple, such a limitation is not necessary at all and our abstract results allow to deal with multiplicities.

\subsubsection{\bf Klein-Gordon equations in dimension $d=1$}
\label{sub_sec_KG}
We consider the Cauchy problem for the nonlinear Klein-Gordon equation on $[0,\pi]$ with homogeneous Dirichlet boundary conditions
\begin{equation}
\label{eq_KG}
\tag{KG}
\left\{ \begin{array}{rllll} \partial_t^2 \Phi(t,x) &=& \partial_x^2 \Phi(t,x) - m \Phi(t,x) + g(x,\Phi(t,x)) & (t,x)\in \mathbb{R}\times (0,\pi), \\
										\Phi(t,0) &=& \Phi(t,\pi) \ = \ 0 & t\in \mathbb{R}, \\
										\Phi(0,x) &=& \Phi^{(0)}(x) & x\in [0,\pi], \\
										\partial_t \Phi(0,x) &=& \dot{\Phi}^{(0)}(x) & x\in [0,\pi] .
\end{array}
\right.
\end{equation}
where the unknown $\Phi(t,x)\in \mathbb{R}$ is real valued, $\Phi^{(0)} \in H^1_0([0,\pi];\mathbb{R})$, $\dot{\Phi}^{(0)} \in L^2(0,\pi;\mathbb{R})$, the mass $m>-1$ is a parameter and $(y \mapsto g(\cdot,y)) \in  C^\infty(\mathbb{R}; H^{1}([0,\pi];\mathbb{R}))$ is a smooth nonlinearity of order $p-1\geq 2$ at the origin\footnote{ i.e. $\displaystyle g(\cdot,y) \mathop{=}_{y=0} \mathcal{O}(y^2) $}.

It is a well know Hamiltonian system. Indeed, it rewrites formally
$$
\partial_t \begin{pmatrix} \Phi \\ \partial_t \Phi \end{pmatrix} = \begin{pmatrix} 0 & 1 \\ -1 & 0 \end{pmatrix} \nabla H (\Phi,\partial_t \Phi)
$$
where, denoting $G(x,y) := \int_0^y g(x,\mathrm{y}) \mathrm{d}\mathrm{y}$, the Hamiltonian $H$ is defined by
$$
H(\Phi,\partial_t \Phi) = \int_0^\pi \frac12(\partial_t \Phi(x))^2 + \frac12(\partial_x \Phi(x))^2+\frac{m}2 (\Phi(x))^2 - G(x,\Phi(x)) \ \mathrm{d}x.
$$
It is relevant, as stated in the following lemma, to note that this Hamiltonian is strongly convex in a neighborhood of the origin.
\begin{lemma}\label{lemma_Ham_KG_coer} For all $m>-1$, there exists $\varepsilon_m>0$ and $\Lambda_m>1$ such that for all $\Phi\in H^1_0([0,\pi];\mathbb{R})$ and all $\Psi \in L^2(0,\pi;\mathbb{R})$, if $\| \Phi\|_{H^1}+\|\Psi\|_{L^2}\leq \Lambda_m \varepsilon_m$ then
$$
\Lambda_m^{-1}(\| \Phi\|_{H^1}+\|\Psi\|_{L^2})^2 \leq H(\Phi,\Psi)\leq \Lambda_m (\| \Phi\|_{H^1}+\|\Psi\|_{L^2})^2.
$$
\end{lemma}

As a consequence of Lemma \ref{lemma_Ham_KG_coer}, using standard methods for semi-linear Hamiltonian systems in their energy space, the global well-posedness of \eqref{eq_KG} for small solutions in $H^1_0 \times L^2$ can be easily obtained (see e.g. \cite{Caz03} for the methods to prove it). It is summarized in the following theorem.
\begin{theorem}[Global well-posedness]
\label{thm_GWP_KG}
Let $m>-1$ and $\varepsilon_m$ be given by Lemma \ref{lemma_Ham_KG_coer}. Provided that $\|\Phi^{(0)}\|_{H^1} +\|\dot{\Phi}^{(0)}\|_{ L^2} \leq \varepsilon_m$ there exists an unique global solution to \eqref{eq_KG} 
$$
(\Phi,\partial_t \Phi) \in C^0_b(\mathbb{R}; H^1_0 \times L^2 ) \cap C^1(\mathbb{R}; L^2 \times H^{-1} ).
$$ 
Moreover, this solution preserves the energy, i.e.
$$
H(\Phi(t),\partial_t \Phi(t)) = H(\Phi^{(0)},\dot{\Phi}^{(0)}), \quad \forall t\in \mathbb{R}.
$$
\end{theorem}

As a consequence of this global well-posedness result and the abstract corollary of our Birkhoff normal form result (i.e. Theorem \ref{thm_main_dyn}), we deduce the almost global preservation of the low harmonic energies
$$
E_n(\Phi,\Psi) := \sqrt{n^2+m} \left( \int_0^\pi \sin( n x) \Phi(x) \ \mathrm{d}x \right)^2+  \frac1{\sqrt{n^2+m}} \left( \int_0^\pi \sin(n x) \Psi(x) \ \mathrm{d}x \right)^2.
$$

\begin{theorem}
\label{thm_main_KG}
For almost all $m>-1$ and all $r\geq 1$, there exist $\beta_r>0$ (depending only on $r$) and $C_{m,r}>0$ such that, for all
 $\Phi^{(0)} \in H^1_0([0,\pi];\mathbb{R})$ and all $\dot{\Phi}^{(0)} \in L^2(0,\pi;\mathbb{R})$, provided that 
$$
\varepsilon:=\|\Phi^{(0)}\|_{H^1} +\|\dot{\Phi}^{(0)}\|_{ L^2} \leq \varepsilon_m
$$ (where $\varepsilon_m$ is given by Lemma \ref{lemma_Ham_KG_coer}), the global solution solution of \eqref{eq_KG} given by Theorem \ref{thm_GWP_KG} satisfies
$$
|t|< \varepsilon^{-r} \quad \Longrightarrow \quad \forall n\geq 1, \ |E_n(\Phi(t),\partial_t \Phi(t)) - E_n(\Phi^{(0)},\dot{\Phi}^{(0)})| \leq C_{m,r} \langle  n \rangle^{\beta_r} \varepsilon^{p}.
$$
\end{theorem}
\begin{remark}\label{rem_KG_per} The proof of this result could be easily adapted to deal with the nonlinear Klein-Gordon equations with periodic boundary condition on $[0,2\pi]$. In this context, we would require the mass $m$ to be positive and we would define the harmonic energies\footnote{which, in this case, are no more the actions of the linear Klein Gordon equation (but super-actions).} by
$$
E_n^{\mathrm{per}}(\Phi,\Psi) := \sqrt{n^2+m} \left| \int_0^{2\pi} e^{inx}  \Phi(x) \ \mathrm{d}x \right|^2+  \frac1{\sqrt{n^2+m}} \left| \int_0^{2\pi} e^{inx}  \Psi(x) \ \mathrm{d}x \right|^2, \quad n\in \mathbb{N}.
$$
Actually we chose to present the equation with homogeneous Dirichlet boundary conditions because, in this case, there did not exist any normal form result to deal with the even nonlinear terms (because it requires to work with low regularity solutions).
\end{remark}

\subsubsection{\bf Nonlinear Schr\"odinger equations in dimension $d=1$}
\label{sub_sec_NLS}
We consider nonlinear Schr\"odinger equations of the form
\begin{equation}
\label{eq_NLS}
\tag{NLS}
i\partial_t u(t,x) = -\partial_x^2 u(t,x) +V(x) u(t,x) + g(x,|u(t,x)|^2)\, u(t,x), \quad t\in \mathbb{R}
\end{equation}
on a domain $\Omega \in \{ \Omega^{\mathrm{Dir}}_1, \Omega^{\mathrm{per}}_1  \}$, with
$$
\Omega^{\mathrm{Dir}}_1= (0,\pi) \quad \mathrm{and} \quad \Omega^{\mathrm{per}}_1= \mathbb{T} = \mathbb{R}/2\pi\mathbb{Z},
$$
and equipped with Dirichlet homogeneous boundary conditions (note that $\partial \mathbb{T} = \emptyset$)
\begin{align}
	u(t,x)&=  0, \quad  x\in \partial \Omega.
\end{align}
In any cases, we denote by $u^{(0)}$ the initial datum
$$
u(0,x) = u^{(0)}(x), \quad x\in \overline{\Omega}.
$$
The unknown $u(t,x) \in \mathbb{C}$ is complex valued, $(y \mapsto g(\cdot,y)) \in  C^\infty(\mathbb{R}; H^{2}(\Omega;\mathbb{R}))$ is a smooth function of order $(p-2)/2 \geq 1$ at the origin\footnote{for example, for the cubic nonlinearity $|u|^2 u$, we have $g(\cdot,y)=y$ and $p=4$.}, $V\in L^\infty(\Omega;\mathbb{R})$ is a real valued potential and $u^{(0)} \in H^1_0(\Omega;\mathbb{C})$ (note that $H^1_0(\mathbb{T};\mathbb{C}) = H^1(\mathbb{T};\mathbb{C})$). We choose this framework because it is physically relevant and quite simple to expose. Actually, we could also consider more general non-linearities (e.g. Hartree, quadratic...).

\medskip

These Schr\"odinger equations are Hamiltonian. Indeed, \eqref{eq_NLS} rewrites formally
$$
i\partial_t u =  \nabla \mathcal{H}(u) \quad \mathrm{where} \quad \mathcal{H}(u) = \frac12 \int_{\Omega} |\partial_x u(x) |^2 + V(x) |u(x)|^2 + G(x,|u(x)|^2) \ \mathrm{d}x
$$
and $G(\cdot,y) := \int_0^y g(\cdot,\mathrm{y}) \mathrm{d}\mathrm{y}$. Moreover they are gauge invariant, which implies, by Noether's Theorem, the preservation of the mass $
\mathcal{M}(u) := \| u \|_{L^2}^2.$ As stated in the following lemma, \eqref{eq_NLS} have natural constants of motion which provide an a priori bound on the $H^1$ norms of the solutions.

\begin{lemma} \label{lemma_Lag_NLS_coer} For all $\rho>0$, there exists $\varepsilon_\rho >0$ and $\Lambda_\rho>0$ such that provided that $\| V \|_{L^\infty} \leq \rho$ and $\|u\|_{H^1} \leq \Lambda_\rho \varepsilon_\rho$, we have
$$
\Lambda_{\rho}^{-1} \| u\|_{H^1}^2 \leq \mathcal{H}(u) +  (\rho+1) \mathcal{M}(u) \leq \Lambda_{\rho} \| u\|_{H^1}^2.
$$
\end{lemma}

As a consequence of Lemma \ref{lemma_Lag_NLS_coer}, the Schr\"odinger equations \eqref{eq_NLS} are globally well posed for small solutions in $H^1_0$.
The proof relies on standard methods for semi-linear Hamiltonian systems in their energy spaces and the Sobolev embedding $H^1 \hookrightarrow L^{\infty}$.
\begin{theorem}[Global well-posedness, $d=1$, Corollary $3.5.3$ of \cite{Caz03} page $77$]
\label{thm_GWP_NLS_1d} $\empty$ \\
Let $\Omega \in \{ \Omega^{\mathrm{Dir}}_1, \Omega^{\mathrm{per}}_1 \}$, $\rho >0$ and $\varepsilon_{\rho}>0$ be given by Lemma \ref{lemma_Lag_NLS_coer}. Provided that $\|u^{(0)} \|_{H^1_0} \leq \varepsilon_{\rho}$ and $\| V\|_{L^\infty} \leq \rho$, there exists an unique global solution to \eqref{eq_NLS} 
$$
u \in C^0_b(\mathbb{R}; H^1_0  ) \cap C^1(\mathbb{R};H^{-1} ).
$$ 
Moreover, this solution preserves the energy and the mass
$$
\forall t\in \mathbb{R}, \quad \mathcal{H}(u(t)) = \mathcal{H}(u^{(0)}) \quad \mathrm{and} \quad \mathcal{M}(u(t)) = \mathcal{M}(u^{(0)}).
$$
\end{theorem}

As a consequence of our abstract normal form result (i.e. Theorem \ref{thm_NF} and Theorem \ref{thm_main_dyn}), we can specify the dynamics of these global solutions for very long times. Nevertheless, we need to introduce some notations about the spectra of Sturm--Liouville operators (much more details are provided in Section \ref{sec_NR} to establish small divisors estimates). The spectral theory of these operators is very classical and the associated literature is huge. We chose as reference the nice book of P\"oschel and Trubowitz \cite{PT} \footnote{actually this book only deals with the Dirichlet spectrum, but the result can be easily adapted for the Neuman spectrum.}.  To state our results for \eqref{eq_NLS}, we only need the objects introduced in the following proposition.

\begin{proposition}[Thm 7 page 43 of \cite{PT}]\label{prop_def_SL}  $\empty$

\noindent $\bullet$ (Dirichlet spectrum) For all $V\in L^2(0,\pi;\mathbb{R})$, there exist an increasing sequence of real numbers $(\lambda_n)_{n\geq 1}$ and a Hilbertian basis $(f_n)_{n\geq 1}$ of $L^2(0,\pi;\mathbb{R})$, composed of functions $f_n \in H^2 \cap H^1_0$, such that for all $n\geq 1$ we have $f_n(0) = f_n(\pi) = 0$ and 
\begin{equation}
\label{eq_SL_dir}
-\partial_x^2 f_n(x) + V(x) f_n(x) = \lambda_n f_n(x), \quad  \forall x\in (0,\pi).
\end{equation}

\medskip

\noindent $\bullet$ (Neuman spectrum) For all $V\in L^2(0,\pi;\mathbb{R})$, there exist a decreasing sequence of real numbers $(\lambda_n)_{n \leq 0}$ and a Hilbertian basis $(f_n)_{n\leq 0}$ of $L^2(0,\pi;\mathbb{R})$, composed of functions $f_n \in H^2$, such that for all $n\leq 0$ we have $\partial_x f_n(0) = \partial_x f_n(\pi) = 0$ and \eqref{eq_SL_dir}.

\medskip

\noindent $\bullet$ (A periodic spectrum\footnote{Note that it is not the usual definition of periodic spectrum for which we usually assume that $V$ is $\pi$ periodic on $\mathbb{T}=\mathbb{R}/2\pi \mathbb{Z}$ (see for example Appendix B of \cite{KP}).}) For all even potential $V\in L^2(\mathbb{T};\mathbb{R})$, let $(\lambda_n)_{n\in \mathbb{Z}}$ (resp. $(f_n)_{n\in \mathbb{Z}}$) be the eigenvalues (resp. eigenfunctions) of the Dirichlet and Neuman Sturm--Liouville operators associated with the restriction of $V$ on $(0,\pi)$. When $n$ is positive (resp. nonnegative), we extend $f_n$ as an odd (resp. even) function on $\mathbb{T}$. Therefore, $(f_n/\sqrt2)_{n\in \mathbb{Z}}$ is a Hilbertian basis of $L^2(\mathbb{T};\mathbb{R})$ and for all $n\in \mathbb{Z}$ we have\footnote{this justifies a posteriori the name of \emph{periodic spectrum}.}
\begin{equation}
\label{eq_SL_per}
-\partial_x^2 f_n(x) + V(x) f_n(x) = \lambda_n f_n(x), \quad  \forall x\in \mathbb{T}.
\end{equation}
\end{proposition}
Note that, in the periodic case, the assumption that $V$ is even is especially useful to ensure that the eigenvalues of the Strum--Liouville operator depend smoothly on $V\in L^2$ (see Proposition \ref{prop_smooth_spec})\footnote{In the general case, i.e. $V$ not even, crossing between eigenvalues may occur which leads to the loss of differentiability of the later.}. 

\medskip

 The following theorem deals with the dynamics of \eqref{eq_NLS} in dimension $d=1$ with homogeneous Dirichlet boundary conditions.
\begin{theorem}[Case $\Omega = \Omega^{\mathrm{Dir}}_1$]
\label{thm_Dir_1d}
Let $V \in L^\infty(0,\pi; \mathbb{R})$ be a bounded real valued potential such that the Dirichlet spectrum of the Sturm-Liouville operator $-\partial_x^2 +V$ is strongly non-resonant according to Definition \ref{def_SNR}. There exists $\epsilon_0>0$ and for all $r\geq 1$, there exist $\beta_r>0$ and $C_r>0$ such that, provided that $\varepsilon := \| u^{(0)} \|_{H^1_0} \leq \epsilon_{0}$,
the global solution of \eqref{eq_NLS} given by Theorem \ref{thm_GWP_NLS_1d} satisfies
$$
|t|< \varepsilon^{-r} \quad \Longrightarrow \quad \forall n\geq 1, \quad \big|\, |u_n(t)|^2 - |u_n(0)|^2 \big| \leq C_r \langle n \rangle^{\beta_r} \ \varepsilon^p
$$ 
where $u_n(t) =  \int_0^\pi u(t,x) f_n(x) \mathrm{d}x$.
\end{theorem}
\begin{remark} 
\label{rem_dep}$\epsilon_0$ depends on $V$ only through its $L^\infty$ norm and $\beta$  depends on $V$ only through the sequence $\alpha$ of Definition \ref{def_SNR}. 
\end{remark}

To check that this result is non-empty, we have to prove that there exist potentials satisfying the assumptions of this theorem. Fortunately, there are many ways to draw $V$ randomly to ensure that, almost surely, the Dirichlet spectrum of $-\partial_x^2 +V$ is strongly non-resonant. However, we do not know if there is a natural way to draw $V$. Usually, in the literature (see e.g. \cite{Bou00,BG06,YZ14,KillBill,KtA}), its Fourier coefficients are drawn independently and uniformly. We could do the same here\footnote{and actually the proof would be a little bit simpler.} but, in order to avoid a too rigid asymptotic behavior for high modes, we draw them  independently with Gaussian laws.

\begin{proposition}
\label{prop_proba_dir_1d}
Let $s>3/2$, $V$ be a real random function on $\mathbb{T}$ of the form
\begin{equation}
\label{eq_V_dir}
V(x) =  \sum_{n\leq -1} V_n \langle n \rangle^{-s} \sin(n \, x) +  \sum_{n \geq 0 } V_n \langle n \rangle^{-s} \cos(n \, x) 
\end{equation}
where $V_n \sim \mathcal{N}(0,1)$ are some independent real centered Gaussian variables of variance $1$. There exists $\rho>0$, such that, almost surely, provided that $\|V\|_{H^1(\mathbb{T})}< \rho$, the Dirichlet spectrum of the Sturm--Liouville operator $-\partial_x^2 +V_{| (0,\pi)}$ is strongly non-resonant according to Definition \ref{def_SNR}.
\end{proposition}
\begin{remark} This result make sense because almost surely $V\in H^1(\mathbb{T})$ \footnote{actually $V\in H^{s-\frac12-\delta}$ for any $\delta>0$.} and $\mathbb{P}(\|V\|_{H^1(\mathbb{T})} < \eta)>0$ for any $\eta>0$. The sequence $\alpha$  of Definition \ref{def_SNR} is deterministic but it depends on $s$. 
\end{remark}

As mentioned in the comments above, our results not only provide a control of low modes for very long times but also an orbital stability result describing\footnote{actually, as in Theorem 1 of \cite{KtA}, we could estimate precisely the variations of the angles $\theta_n$.} the leading part of the dynamics. Concretely, we have the following corollary of Theorem \ref{thm_Dir_1d}.
\begin{corollary}[Case $\Omega = \Omega^{\mathrm{Dir}}_1$] \label{cor_whaou}
Let $V$ and $\epsilon_0$ be as in Theorem \ref{thm_Dir_1d} and $s> 1$ be a real number.

For all $r\geq 1$, there exists $K>0$ and $\delta >0$ such that, provided that $u^{(0)}\in H^s(0,\pi;\mathbb{C}) \cap H^1_0(0,\pi;\mathbb{C})$ satisfies $\varepsilon := \| u^{(0)} \|_{H^s} \leq \epsilon_0$, the global solution of \eqref{eq_NLS} given by Theorem \ref{thm_GWP_NLS_1d} satisfies
$$
|t|< \varepsilon^{-r} \quad \Longrightarrow \quad  \big\| u(t) - \sum_{n\geq 1} e^{i\theta_n(t)} u_n(0) f_n   \big\|_{H^1} \leq K \varepsilon^{1+\delta}
$$ 
where $u_n(0) =  \int_0^\pi u(0,x) f_n(x) \mathrm{d}x$ and $\theta_n : \mathbb{R} \to \mathbb{R}$ depends only on $n$ and $u^{(0)}$.
\end{corollary}
 Somehow, this result is similar to \cite{Bou96,BFG20b}: to control the solution, we require an extra smoothness to the initial datum. Nevertheless, contrary to these previous results, here the loss of smoothness is arbitrarily small (i.e. $s$ can be chosen arbitrarily close to $1$).

\medskip

In the periodic setting (i.e. when $\Omega= \Omega_1^{\mathrm{per}}$), we could prove the same result as in Theorem \ref{thm_Dir_1d}. Unfortunately, we do not success to prove that the set of the admissible potentials is non-empty. Therefore, we have to introduce the following weaker non-resonance condition.  

\begin{definition}[Limited strong non-resonance] \label{def_SNR_lim} Let $d\geq 1$, $\mathbf{N}_d \subset \mathbb{Z}^d$ and $\omega \in \mathbb{R}^{\mathbf{N}_d}$, $r\geq 1$, $N \geq 1$.

  The frequencies $\omega$ are strongly non resonant up to order $r$, for small divisors involving at least one mode smaller than $N$, if there exists $\gamma_r>0$ such that for all $r_\star \leq r$, all $\ell_1,\dots,\ell_{r_\star} \in \mathbb{Z}^*$, all $n\in \mathbf{N}_d^{r_\star}$ injective, provided that $|\ell_1|+\cdots+|\ell_{r_\star}| \leq r$, $\langle n_1 \rangle \leq \cdots \leq \langle n_{r_\star} \rangle$ and $\langle n_1 \rangle\leq N$ we have 
  $$
 |\ell_1 \omega_{n_1} + \cdots + \ell_{r_\star} \omega_{n_{r_\star}} | \geq \gamma.
  $$
\end{definition}
\begin{remark}
Note that, if $\mathbf{N}_d$ is infinite, $n_{r_\star}$ is unbounded and this uniform lower bound has to hold for infinitely many small divisors.
\end{remark}

Using this weaker non-resonance condition, we have the following theorem which deals with the dynamics of \eqref{eq_NLS} on $\mathbb{T}$. 
\begin{theorem}[Case $\Omega = \Omega^{\mathrm{per}}_1$]
\label{thm_per_1d}
Let $N\geq 1$, $r\geq p$ be an even number and $V \in L^\infty(\mathbb{T}; \mathbb{R})$ be a bounded real valued even potential such that the periodic spectrum of the Sturm-Liouville operator $-\partial_x^2 +V$ is strongly non-resonant, up to order $r$, for small divisors involving at least one mode smaller than $N$, according to Definition \ref{def_SNR_lim}. There exists $\epsilon_0>0$, there exists $C>0$ such that, provided that $\varepsilon := \| u^{(0)} \|_{H^1} \leq \epsilon_{0}$,
the global solution of \eqref{eq_NLS} given by Theorem \ref{thm_GWP_NLS_1d} satisfies
$$
|t|< \varepsilon^{-r'} \quad \Longrightarrow \quad \forall n \in \llbracket-N,N \rrbracket, \quad \big|\, |u_n(t)|^2 - |u_n(0)|^2 \big| \leq C  \varepsilon^p
$$ 
where $r'=r-p$ and $u_n(t) =  \int_\mathbb{T} u(t,x) f_n(x) \mathrm{d}x$.
\end{theorem}
This result is weaker than the one we have for Dirichlet boundary conditions. Being given a potential, the number of modes we control does not grow as the norm of the solution decreases. Such a result is not strong enough to deduce a stability result as in Corollary \ref{cor_whaou}. However, we prove, in the following proposition, that the non-resonance condition is typically fulfilled. 

\begin{proposition}
\label{prop_proba_per_1d}
Let $s>3/2$ and $V$ be a real even random function on $\mathbb{T}$ of the form
$$
V(x) =  \sum_{n \geq 0 } V_n \langle n \rangle^{-s} \cos(n \, x) 
$$
where  $V_n \sim \mathcal{N}(0,1)$ are some independent real centered Gaussian variables of variance $1$. For all  $N\geq 1$ and $r\geq 1$ there exists $\rho_{r,N}>0$, such that provided $\|V\|_{H^1}<\rho_{r,N}$, almost surely, the periodic spectrum of the Sturm-Liouville operator $-\partial_x^2 +V$ is strongly non-resonant, up to order $r$, for small divisors involving at least one mode smaller than $N$, according to Definition \ref{def_SNR_lim}.
\end{proposition}
Therefore, in the periodic setting, the larger the number of modes we control is and the longer the time scale on which we control them is, the smaller the potential has to be. We do not know if such a limitation is physical or just technical.

\subsubsection{\bf Nonlinear Schr\"odinger equations in dimension $d=2$}
\label{sub_sec_NLS2}
In dimension $2$, the behavior of the Sturm--Liouville spectra is much more intricate. Indeed, due to the multiplicities of the eigenvalues, they do not necessarily depends smoothly on the potential and the eigenfunctions are not especially well localized (see e.g. \cite{BB13}). Therefore, as usual (see e.g. \cite{BG06,FI19,BMP20}), we consider a toy model where the multiplicative potential is replaced by a convolutional potential. Moreover, to simplify as much as possible, we only consider  nonlinear Schr\"odinger equations on $\mathbb{T}^2 = \mathbb{R}^2 / (2\pi \mathbb{Z}^2)$ with a homogeneous cubic nonlinearity. More precisely, they are of the form
\begin{equation}
\label{eq_NLS_2}
\tag{$\mathrm{NLS}^2$}
\left\{ \begin{array}{lll}
i\partial_t u(t,x) = -\Delta u(t,x) +(V \star u)(t,x) + |u(t,x)|^2 u(t,x),  \\
u(0,x) = u^{(0)}(x), 
\end{array}
\right.
\end{equation}
where $(t,x)\in \mathbb{R}\times\mathbb{T}^2$, $u^{(0)}\in H^1(\mathbb{T}^2;\mathbb{C})$ and $V \in H^1(\mathbb{T}^2 ; \mathbb{C})$ is a potential with \emph{real Fourier coefficients}, the Fourier transform on $\mathbb{T}^2$ being defined by
$$
\forall v\in L^2(\mathbb{T}^2),\forall n \in \mathbb{Z}^2, \ \widehat{v}_n := \frac1{2\pi} \int_{\mathbb{T}^2} v(x) e^{-inx}\mathrm{d}x. 
$$
 Of course, these Schr\"odinger equations are also Hamiltonian. Indeed, \eqref{eq_NLS_2} rewrites formally
$$
i\partial_t u =  \nabla \mathcal{H}(u) \quad \mathrm{where} \quad \mathcal{H}(u) =\frac12 \int_{\mathbb{T}}  |\partial_x u(x) |^2 + \Re (\bar{u}(x)(V \star u)(x)) + \frac12 |u(x)|^4 \ \mathrm{d}x
$$
This equation being gauge invariant, the mass $\mathcal{M}(u) = \|u\|_{L^2}^2$ is a constant of the motion. Therefore, as in dimension one, the constant of the motions provide an a priori control on the $H^1$ norm.
\begin{lemma} \label{lemma_Lag_NLS_coer_2d} For all $\rho>0$, there exists $\varepsilon_\rho >0$ and $\Lambda_\rho>0$ such that provided that $\| V \|_{L^2} \leq \rho$ and $\|u\|_{H^1} < \Lambda_\rho \varepsilon_\rho$, we have
$$
\Lambda_{\rho}^{-1} \| u\|_{H^1}^2 \leq \mathcal{H}(u) +  (\rho+1) \mathcal{M}(u) \leq \Lambda_{\rho} \| u\|_{H^1}^2.
$$
\end{lemma}
Since, in dimension $d=2$, $H^1$ functions are not bounded, the global well-posedness of small solutions of \eqref{eq_NLS_2} in $H^1$ is not trivial (especially the uniqueness). Fortunately, it has been proven by Vladimirov in \cite{Vla84} and is presented by Cazenave in Theorem $3.6.1$ page $78$ of \cite{Caz03}. It is summarized in the following theorem. 
\begin{theorem}[Global well-posedness, $d=2$, \cite{Vla84}]
\label{thm_GWP_NLS_2d} Let $\rho >0$ and $\varepsilon_{\rho}>0$ be given by Lemma \ref{lemma_Lag_NLS_coer_2d}. If  $\|u^{(0)} \|_{H^1} \leq \varepsilon_{\rho}$ and $\| V\|_{L^2} \leq \rho$, there exists an unique global solution to \eqref{eq_NLS_2} 
$$
u \in L^\infty(\mathbb{R}; H^1  ) \cap W^{1,\infty}(\mathbb{R};H^{-1} ).
$$ 
Moreover, this solution preserves the energy and the mass
$$
\forall t\in \mathbb{R}, \quad \mathcal{H}(u(t)) = \mathcal{H}(u^{(0)}) \quad \mathrm{and} \quad \mathcal{M}(u(t)) = \mathcal{M}(u^{(0)}).
$$
\end{theorem}

For typical values of the potential, we get the following description of the small solutions of \eqref{eq_NLS_2}.
\begin{theorem}
\label{thm_per_2d}
Let $V\in H^1(0,\pi;\mathbb{C})$ be a potential whose Fourier coefficients are real and such that the frequencies $(|n|^2+\widehat{V}_n)_{n\in \mathbb{Z}^2}$ are strongly non-resonant according to Definition \ref{def_SNR}. There exists $\epsilon_0>0$, such that for all $r\geq 1$, there exists $\beta_r>0$ and for all $\delta \in (0,1/2)$, there exists $C_{r,\delta}>0$ such that, provided that $\varepsilon := \| u^{(0)} \|_{H^1} \leq \epsilon_{0}$,
the global solution of \eqref{eq_NLS_2} given by Theorem \ref{thm_GWP_NLS_2d} satisfies,
$$
|t|< \varepsilon^{-r} \quad \Longrightarrow \quad \forall n \in \mathbb{Z}^2, \quad \big|\, |\widehat{u}_{n}(t)|^2 - |\widehat{u}_{n}(0)|^2 \big| \leq C_{r,\delta} \langle n \rangle^{\beta_r} \ \varepsilon^{4-\delta}.
$$ 
\end{theorem}
This result is similar to the one we have in dimension $1$ (i.e. Theorem \ref{thm_Dir_1d}) excepted that we have an arbitrarily small loss (the exponent $\delta$) in the control of the variation of the actions. Roughly speaking it is due to the fact that, in dimension $2$, $H^1$ is not an algebra but almost! The statement of Remark \eqref{rem_dep} about the dependencies in Theorem \ref{thm_Dir_1d} also holds here. On the probabilistic side, the following result proves that Theorem \ref{thm_per_2d} makes sense (since \eqref{eq_NLS_2} is a toy model we draw $V$ as simply as possible).
\begin{proposition}
\label{prop_proba_dir_2d}
Let $s>3/2$ and $V \in H^1(0,\pi;\mathbb{C})$ be a random potential whose Fourier coefficients, $\widehat{V}_n$, are real, independent and uniformly distributed in $(-\langle n \rangle^{-s},\langle n \rangle^{-s})$. Almost surely, the frequencies $(|n|^2+\widehat{V}_n)_{n\in \mathbb{Z}}$ are strongly non-resonant according to Definition \ref{def_SNR}.
\end{proposition}

\subsection{Scheme of the proof}
\label{subsec_scheme} As usual the proof is based on a normal form process to eliminate as many terms as possible from the Hamiltonian. Thanks to our non-resonance condition \eqref{eq_est_NR}, we can separate the dynamics of the low modes ($\langle n\rangle\leq N$) from those of the high modes ($\langle n\rangle> N$)
by eliminating from the Hamiltonian, $H=Z_2+P$, all the monomials  that influence the dynamics of the  low modes. To be more precise, but not too technical, let us assume that $P$ is a homogeneous polynomial of degree $p$ and write formally
$$
P=\sum_{\sigma\in\{-1,1\}^p}\sum_{n\in \mathbf{N}^p_d } P^\sigma_n u_{n_1}^{\sigma_1}\cdots u_{n_p}^{\sigma_p}
$$
where $u_{n_j}^{1}:=u_{n_j}$ while $u_{n_j}^{-1}:=\bar u_{n_j}$.
To eliminate the monomials $u_{n_1}^{\sigma_1}\cdots u_{n_p}^{\sigma_p}$ we have to control the associated small divisor
$\sigma_1 \omega_{n_1}+\cdots+\sigma_p \omega_{n_p}.$
Using our non resonance condition \eqref{eq_est_NR} we have
$$
\big| \sigma_1 \omega_{n_1}+\cdots+\sigma_p \omega_{n_p}\big|\geq \gamma_p\kappa(\sigma,n)^{-\beta_p}
$$
where 
\begin{equation}
\label{eq_def_kappa}
\kappa_\omega(\sigma,n) := \min \big\{ \ \langle n_j \rangle \quad \mathrm{such \ that} \quad j \in \llbracket 1,p \rrbracket \quad \mathrm{and} \quad \sum_{k \ s.t.\ \omega_{n_k} = \omega_{n_j}} \sigma_k \neq 0\big\}
\end{equation}
is the \emph{effective} lower index.
 So, paying a factor $N^{\beta_p}$ on the coefficients of the transformed Hamiltonian,  we can eliminate all the monomials of $P$ for which $\kappa_\omega(\sigma,n)\leq N$. Then iterating this procedure up to degree $r$, we construct a symplectic transformation $\tau$ such that, on a neighborhood of the origin in ${\mathcal{E}}$,
$$
H\circ \tau =Z_2+Q_r+R_r
$$
where $R_r$ is a remainder term satisfying $\|\nabla R_r(u)\|_{\mathcal{E}}\lesssim N^{\alpha_r}\|u\|^{r-1}_{\mathcal{E}}$ for some $\alpha_r>0$ and $Q_r$ is a polynomials of degree $r$ containing only monomials $ u_{n_1}^{\sigma_1}\cdots u_{n_\ell}^{\sigma_\ell}$ such that $\kappa(\sigma,n)> N$, and thus satisfying 
$$\{|u_n|^2,Q_r(u)\}=0
\quad \mathrm{for} \quad
\langle n \rangle \leq N.$$
 This \emph{algebraic} result is formalized and quantified in Theorem \ref{thm_NF}.

As a dynamical consequence, denoting $v=\tau^{-1}(u)$ the new variable, thank to the a priori estimate provided by the coercivity estimate \eqref{Lambda}, we have for the low modes, $\langle n \rangle \leq N$,
$$
\frac{d}{dt}|v_n|^2=\{|v_n|^2,H\circ \tau\}=\{|v_n|^2,R_r\}= O_{N,r}(\|v\|_{\mathcal{E}}^r) = O_{N,r}(\varepsilon^r) .
$$
On the other hand the high modes, $|v_n|^2$, $\langle n\rangle > N$, are controlled by using the a priori bound on the energy norm. Formally, these estimates on the variation of $|v_n|^2$ leads naturally to Theorem \ref{sketch}.

 However, due to fact that we work in low regularity, new technical difficulties appear. For instance, in low regularity, it is not so trivial to prove  the time derivability of the solution expressed in the new variables (see section \ref{sec_dyn_co}).

\subsection{Outline of the work} Section \ref{sec_NR} is devoted to the small divisor estimates. In particular, we provide tools to prove that many systems (like \eqref{eq_NLS} or \eqref{eq_KG}) enjoy the non-resonance condition \eqref{eq_est_NR}. In Section \ref{sec_class_Ham}, we introduce the Hamiltonian formalism we need to state and to prove our main results. Then, in Section \ref{sec_BNF_LR}, we prove our Birkhoff normal form theorem and, in Section \ref{sec_dyn_co}, we prove its main dynamical corollary. Finally, Section \ref{sec_proofs} is devoted to the proofs of the theorems associated with the applications. We stress that Section \ref{sec_NR} is almost independent of the other sections.

\subsection{Notations and conventions} $\empty$

\noindent $\bullet$ If $E$ is a real normed vector space then $\mathscr{L}(E)$ denotes the space of the bounded operators from $E$ into $E$.

\noindent $\bullet$ As usual, the Japanese bracket is defined by $\langle x\rangle := \sqrt{1+|x|^2}$.

\noindent $\bullet$ If $x\in \mathbb{R}^d$ for some $d\geq 1$ then $|x|_1:=|x_1|+\cdots+|x_d|$.

\noindent $\bullet$ Smooth always means $C^\infty$.
 
\noindent $\bullet$ When it is not specified, functions or sequences are always implicitly complex valued.

\noindent $\bullet$ If $P$ is a property then $\mathbb{1}_P = 1$ if $P$ is true while $\mathbb{1}_P = 0$ if $P$ is false. 

\noindent $\bullet$ If $p$ is a parameter or a list of parameters and $x,y\in \mathbb{R}$ then we denote $x\lesssim_p y$ if there exists a constant $c(p)$, depending continuously on $p$, such that $x\lesssim c(p) \, y$. Similarly, we denote $x\gtrsim_p y$ if $y\lesssim_p x$ and $x\approx_p y$ if $x\lesssim_p y\lesssim_p x$.

\noindent $\bullet$ If $(x_j)_{j\in S} \in (\mathbb{R}_{+}^*)^S$ is a family of positive numbers indexed by a finite set $S$ then its harmonic mean is defined by
\begin{equation}
\label{eq_def_hmean}
\mathop{\mathrm{hmean}}_{j\in S} (x_j):=\big( \frac1{\# S} \sum_{j\in S} \frac1{x_j} \big)^{-1}.
\end{equation}

\subsection{Acknowledgments}

During the preparation of this work the authors benefited from the support of the Centre Henri Lebesgue ANR-11-LABX-0020-0 and by ANR-15-CE40-0001-01  "BEKAM". 

\section{A new non-resonance condition}
\label{sec_NR}

In this section, first we establish useful results to prove strong non-resonance. Then we apply them to \eqref{eq_KG}, \eqref{eq_NLS} and \eqref{eq_NLS_2}.

\subsection{Abstract results}
The following proposition proves that if some frequencies are non-resonant in a classical (weak) sense and are well localized then they are strongly non-resonant (according to Definition \ref{def_SNR} or Definition \ref{def_SNR_lim}). The case $\mu=0$ and $N_{\max}=+\infty$ is the the easiest to understand.

\begin{proposition} \label{prop_weak_become_strong} Let $d\geq 1$, $\mathbf{N}_d \subset \mathbb{Z}^d$, $N_{\max} \in [1,+\infty]$, $r\geq 1$, $\mu \in \mathbb{R}$ and $\omega \in  \mathbb{R}^{\mathbf{N}_d}$.

If there exists $\alpha,\gamma>0$, such that for all $r_\star \leq r$, all $\ell \in (\mathbb{Z}^*)^{r_\star}$, all $n \in \mathbf{N}_d^{r_\star}$ injective satisfying $\langle n_1 \rangle \leq \dots \leq \langle n_{r_\star} \rangle$, $|\ell|_1 \leq r$ and $\langle n_1 \rangle \leq N_{\max}$, we have
\begin{equation}
\label{eq_def_weak_NR}
\forall k \in \mathbb{Z},\forall h\in \llbracket -r,r \rrbracket, \quad 
\big| k + h\mu+\ell_1 \omega_{n_1}+\cdots +\ell_{r_\star} \omega_{n_{r_\star}} \big| \geq \gamma \langle n_{r_\star} \rangle^{-\alpha},
\end{equation}
and if there exist $C>0$ and $\nu >0$ such that
\begin{equation}
\label{eq_def_accumulate}
\forall n \in  \mathbf{N}_d, \exists k \in \mathbb{Z}, \ |\omega_n - k - \mu | \leq C \langle n \rangle^{-\nu},
\end{equation}
then there exist $\beta>0$ (depending only on $(\alpha,\nu,r)$) and $\eta >0$ (depending only on $(\alpha,\nu,C,\gamma,r)$) such that for all $r_\star \leq r$, all $\ell \in (\mathbb{Z}^*)^{r_\star}$, all $n \in \mathbf{N}_d^{r_\star}$ injective satisfying $|\ell|_1 \leq r$, $\langle n_1 \rangle \leq \dots \leq \langle n_{r_\star} \rangle$  and $\langle n_1 \rangle \leq N_{\max}$, we have
\begin{equation}
\label{eq_def_strong_NR}
\big| \ell_1 \omega_{n_1}+\cdots +\ell_{r_\star} \omega_{n_{r_\star}} \big| \geq \eta \langle n_{1} \rangle^{-\beta}
\end{equation}
\end{proposition}
\begin{remark} Most of the non-resonant systems enjoy the non resonant estimate \eqref{eq_def_weak_NR}. The second assumption \eqref{eq_def_accumulate} (which means that the frequencies accumulate polynomially fast on $ \mathbb{Z}$) is much more restrictive. 
\end{remark}
\begin{proof}[Proof of Proposition \ref{prop_weak_become_strong}] We fix $r,r_\star$ such that $r_\star \leq r$ and $\ell \in (\mathbb{Z}^*)^{r_\star}$, $n \in \mathbf{N}_d^{r_\star}$ injective such that $\langle n_1 \rangle \leq \dots \leq \langle n_{r_\star} \rangle$, $|\ell|_1 \leq r$ and $\langle n_1 \rangle \leq N_{\max}$.

We are going to prove by induction on $r_{\flat}\leq r_\star$ that there exists $\beta_{r_\flat}>0$ (depending only on $(\alpha,\nu,r)$) and $\eta_{r_\flat} >0$ (depending only on $(\alpha,\nu,C,\gamma,r)$) such that
\begin{equation}
\label{eq_def_strong_strong_NR}
\forall k \in \mathbb{Z}, \quad  \big| k + \sum_{r_{\flat} <j \leq r_\star}  \ell_j \mu + \sum_{1\leq j \leq r_{\flat} }\ell_j \omega_{n_j} \big| \geq \eta_{r_\flat} \langle n_{1} \rangle^{-\beta_{r_\flat}}.
\end{equation}
Note that, when $r_{\flat}= r_\star$, this property is stronger than \eqref{eq_def_strong_NR}. Furthermore, the initialization of the induction is obvious because, when $r_{\flat}=1$, applying \eqref{eq_def_weak_NR} with $r_\star \leftarrow 1$ and recalling that $|\ell|_1 \leq r$, we get \eqref{eq_def_strong_strong_NR} with $\eta_{1} = \gamma$ and $\beta_{r_\flat} = \alpha$.

We assume that $r_\flat < r_\star$ and that \eqref{eq_def_strong_strong_NR} holds and we fix $k\in \mathbb{Z}$.
Since the frequencies accumulate polynomially fast on $\mathbb{Z} + \mu$ (see \eqref{eq_def_accumulate}), there exists $k_\flat \in \mathbb{Z}$ such that
$$
|\omega_{n_{r_\flat +1} } - k_\flat-\mu|\leq C \langle n_{r_\flat +1} \rangle^{-\nu}.
$$
Therefore, applying the triangle inequality and the induction hypothesis \eqref{eq_def_strong_strong_NR}, we have
\begin{multline*}
   \big| k + \sum_{{r_{\flat}+1} <j \leq r_\star}  \ell_j \mu + \sum_{1\leq j \leq {r_{\flat}+1} }\ell_j \omega_{n_j} \big|
 \geq  \big| k + \ell_{r_\flat+1}  k_\flat  + \sum_{r_{\flat} <j \leq r_\star}  \ell_j \mu  + \sum_{1\leq j \leq r_{\flat} }\ell_j \omega_{n_j}  \big| - |\ell_{r_\flat+1}|  |\omega_{n_{r_\flat +1} } - k_\flat -\mu| \\ 
 \geq  \eta_{r_\flat} \langle n_{1} \rangle^{-\beta_{r_\flat}} -C r \langle n_{r_\flat +1} \rangle^{-\nu}.
\end{multline*}
Hence we have to distinguish two cases. 

\noindent $\bullet$ If $2 Cr \langle n_{r_\flat +1} \rangle^{-\nu} \leq \eta_{r_\flat} \langle n_{1} \rangle^{-\beta_{r_\flat}}$  we have
$$
\big| k + \sum_{{r_{\flat}+1} <j \leq r_\star}  \ell_j \mu + \sum_{1\leq j \leq {r_{\flat}+1} }\ell_j \omega_{n_j} \big|\geq \frac12 \eta_{r_\flat} \langle n_{1} \rangle^{-\beta_{r_\flat}}.
$$
\noindent $\bullet$ Else, we have $\langle n_{r_\flat +1} \rangle \leq (2Cr\eta_{r_\flat}^{-1})^{1/\nu} \langle n_{1} \rangle^{\beta_{r_\flat}/\nu}$ and so, by \eqref{eq_def_weak_NR} (applied with $r_\star \leftarrow r_\flat+1$ and $h\leftarrow \ell_{r_{\flat}+2} + \cdots + \ell_{r_\star}$) we get
$$
\big| k + \sum_{{r_{\flat}+1} <j \leq r_\star}  \ell_j \mu + \sum_{1\leq j \leq {r_{\flat}+1} }\ell_j \omega_{n_j} \big| \geq \gamma \big(\frac{\eta_{r_\flat}}{2Cr}\big)^{\frac{\alpha}{\nu}} \langle n_{1} \rangle^{-\frac{\alpha\beta_{r_\flat}}{\nu}}.
$$
\end{proof}

In practice, the eigenvalues of the linearized vector field may have some multiplicities. Therefore, the sequence of frequencies $\omega$ may be non injective and thus strongly non-resonant according to Definition \ref{def_SNR} or Definition \ref{def_SNR_lim}. So we extend these definitions in order to deal with these multiplicities (moreover, we choose a formalism well suited to the Birkhoff normal form process).
\begin{definition}[Generalized strong non-resonance] \label{def_nonres} A family of frequencies $\omega\in \mathbb{R}^{\mathbf{Z}_d}$, where $\mathbf{Z}_d \subset \mathbb{Z}^d$, is strongly non-resonant, up to order $r_{\max}$, for small divisors involving at least one mode of index smaller than $N_{\max}\in [1,+\infty]$, if for all $r \in \llbracket 1,r_{\max} \rrbracket $, there exists $\gamma_r>0$ and $\beta_r>0$ such that for all $n \in \mathbf{Z}_d^{r}, \sigma \in \{ -1,1\}^{r}$, provided that $\kappa_\omega(\sigma,n)\leq N_{\max}$ ($\kappa_\omega(\sigma,n)$ is defined by \eqref{eq_def_kappa})
we have either
$
\displaystyle \big|\sum_{j=1}^{r} \sigma_j \, \omega_{n_j}\big| \geq \gamma_r \, \kappa_\omega(\sigma,n)^{-\beta_r} 
$
or $r$ is even and there exists $\rho \in \mathfrak{S}_{r}$ such that
$$
\forall j \in \llbracket 1,r/2\rrbracket, \ \sigma_{\rho_{2j-1}}= -\sigma_{\rho_{2j}} \quad \mathrm{and} \quad \omega_{n_{\rho_{2j-1}}}=\omega_{n_{\rho_{2j}}}.
$$ 
We denote by $\mathcal{D}_{\gamma,N_{\max}}^{\beta,r_{\max}}(\mathbf{Z}_d)$ the set of these strongly non-resonant frequencies. 
\end{definition}
In the following lemma, we specify how this definition is an extension of Definition \ref{def_SNR} and Definition \ref{def_SNR_lim}.
\begin{lemma} \label{lem_defs_equiv}  Let $d\geq 1$, $\mathbf{N}_d \subset \mathbb{Z}^d$ and $\omega \in  \mathbb{R}^{\mathbf{N}_d}$ an injective sequence.

\noindent $\bullet$ If $\omega$ is strongly non resonant according to Definition \ref{def_SNR} then it is strongly non resonant according to Definition \ref{def_nonres} up to any order and for all modes (i.e. $N_{\max} = +\infty$).

\noindent $\bullet$ Provided that $N_{\max}\equiv N<\infty$, Definition \ref{def_SNR_lim} and Definition \ref{def_nonres} are equivalent.
\end{lemma}
\begin{proof} It is enough to rearrange the small divisors to have
$$
\sigma_1 \, \omega_{n_1} + \cdots + \sigma_r \, \omega_{n_r} =  \ell_1 \, \omega_{m_1} + \cdots +  \ell_{r_\star} \, \omega_{m_{r_\star}}
$$ 
where $r_\star\leq r$, $\kappa_\omega(\sigma,n) = \langle m_1 \rangle \leq \cdots \leq \langle m_{r_{\star}} \rangle$ and $\ell_j =\sum_{\omega_{n_k} = \omega_{m_j}} \sigma_k \neq 0$.
\end{proof}

\subsection{Strong non-resonance of \eqref{eq_KG}}
As a first direct application we deduce that for generic values of the mass, the frequencies of the Klein--Gordon equation are strongly non-resonant.
\begin{lemma} \label{lem_nr_KG} For almost all $m>-1$, the frequencies $\omega_n = \sqrt{n^2+m},$ with $n\geq 1$, are strongly non-resonant according to Definition \ref{def_SNR}.
\end{lemma}
\begin{proof} 
These frequencies have been widely studied in the literature. In particular, it is well known (see e.g. \cite{Bam03,Del09}) that they are non-resonant in the following sense: for almost all $m>-1$, there exists $\gamma_r$ (depending only on $m$ and $r$) and $\alpha_r>0$ (depending only on $r$) such that, if $r_\star \leq r$, $\ell \in (\mathbb{Z}^*)^{r_\star}$, $n \in (\mathbb{N}^*)^{r_\star}$ satisfy  $ n_1 < \dots < n_{r_\star} $ and $|\ell|_1 \leq r$ then
$$
\forall k\in \mathbb{Z}, \ |k+\ell_1 \omega_{n_1}+\cdots+ \ell_{r_\star} \omega_{n_{r_\star}} | \geq \gamma_r  \langle n_{r_\star} \rangle^{-\alpha_r}.
$$
Moreover, a Taylor expansion proves that the frequencies accumulate on $\mathbb{Z}$ :
$$
|\omega_n - n| = \sqrt{n^2+m} - n \leq \frac{m}{2n}.
$$
Therefore, applying Proposition \ref{prop_weak_become_strong}, we deduce that these frequencies are strongly non-resonant.
\end{proof}

\subsection{Strong non-resonance of \eqref{eq_NLS_2}} As a second direct corollary of Proposition \ref{prop_weak_become_strong}, we prove that for generic convolutional potentials, the frequencies of \eqref{eq_NLS_2} are strongly non-resonant (i.e. we prove Proposition \ref{prop_proba_dir_2d}).

\begin{proof}[\bf Proof of Proposition \ref{prop_proba_dir_2d}] Let $\omega_n = |n|^2 + \widehat{V}_n$, $n\in\mathbb{Z}^2$. It is well known (see e.g. Thm 3.22 of \cite{BG06}) that, almost surely, for all $r\geq 1$, there exists $\gamma_r$ and $\alpha_r>0$ such that, if $r_\star \leq r$,  $n \in (\mathbb{Z}^2)^{r_\star}$ injective, $\ell \in (\mathbb{Z}^*)^{r_\star}$ satisfy  $ \langle n_1 \rangle \leq \dots \leq \langle n_{r_\star} \rangle $ and $|\ell|_1 \leq r$ then
$$
\forall k\in \mathbb{Z}, \ |k+\ell_1 \omega_{n_1}+\cdots+ \ell_{r_\star} \omega_{n_{r_\star}} | \geq \gamma_r  \langle n_{r_\star} \rangle^{-\alpha_r}.
$$
 Furthermore, by definition, it is clear that $
|\omega_n - |n|^2| \leq \langle n \rangle^{-3/2}.$ Therefore, applying Proposition \ref{prop_weak_become_strong}, we deduce that, almost surely, these frequencies are strongly non-resonant according to Definition \ref{def_SNR}.

\end{proof}

\subsection{Strong non-resonance of \eqref{eq_NLS}} In this subsection we aim at proving Propositions \ref{prop_proba_dir_1d} and \ref{prop_proba_per_1d}. The frequencies of \eqref{eq_NLS} being eigenvalues of Sturm--Liouville operators, all the results of this subsection deal with the objects introduced in Proposition \ref{prop_def_SL}. In this subsection, we consider the eigenfunctions $f_n$ and eigenvalues $\lambda_n$ of the Sturm-Liouville operator $-\partial_x^2 +V$ as functions of $V \in L^2(0,\pi;\mathbb{R})$. However, when there is no possible ambiguity, we do not specify this dependency.

First, we collect some useful results about the Sturm--Liouville spectra (all of them are proven in \cite{PT} for the Dirichlet spectrum but can be easily extended to the Neuman spectrum).
\begin{proposition}[Thm 3  page 31 of \cite{PT}] \label{prop_smooth_spec} For all $n\in \mathbb{Z}$ both $\lambda_n$ and $f_n \in H^2(0,\pi;\mathbb{R})$ depend analytically on $V$ in $L^2(0,\pi;\mathbb{R})$. Moreover, for all $V,W \in L^2(0,\pi;\mathbb{R})$, we have
\begin{equation}
\label{eq_diff_lambda}
\mathrm{d} \lambda_n (V)(W) = \int_0^\pi W(x) f_n^2(x) \, \mathrm{d}x.
\end{equation}
\end{proposition}
\begin{proposition}[Thm 4  page 35 of \cite{PT}]
\label{prop_it_acumulates} Locally uniformly with respect to $V$ in $H^1(0,\pi;\mathbb{R})$, we have
$$
\lambda_n = n^2 + \frac1\pi \int_{0}^{\pi} V(x)\, \mathrm{d}x + \mathcal{O}(\frac1n).
$$
\end{proposition}

\begin{proposition}[Variation of Thm 4  page 16 of \cite{PT}] \label{prop_desc_fn}For all $\rho>0$, there exists $C>0$ such that if $\|V\|_{H^1}\leq \rho$, $n>0$  and $x\in (0,\pi)$, we have
$$
|f_n(x)-\sqrt{\frac2\pi}\Big(\sin nx- \frac{\mathcal{V}(x)}{2n}\cos nx \Big)|\leq \frac C{n^2}\|V\|_{H^1}, 
$$
$$
|f_{-n}(x)-\sqrt{\frac2\pi}\Big(\cos nx+\frac{\mathcal{V}(x)}{2n}\sin nx \Big)|\leq \frac C{n^2}\|V\|_{H^1}\ , 
$$
where $\mathcal{V}(x) := \int_0^x V(y)  - \pi^{-1}\int_0^\pi V(z) \,\mathrm{d}z\,  \mathrm{d}y$.
\end{proposition}
\begin{proof}
The proof is a variant of the proof of Theorem 4  page 16 of \cite{PT}, for the sake of completeness we include it here. Let $x\mapsto y(x,\lambda)$ be the solution of $(-\partial_{xx}+V)y=\lambda y$ with the initial conditions $y(0,\lambda)=0$ and $y'(0,\lambda)=1$. 
Using the Duhamel rule we easily get (see Theorem 1 page 7 in \cite{PT})
$$ y(x,\lambda)=s_\lambda(x)+\sum_{n\geq1}S_n(x,\lambda)$$
where $s_\lambda(x)=\frac{\sin \sqrt\lambda x}{\sqrt \lambda}$ and $S_n(x,\lambda)=\int_{0\leq t_1\leq\cdots\leq t_{n+1}=x}s_\lambda(t_1)\prod_{i=1}^n\big(s_\lambda(t_{i+1}-t_i)V(t_i)\big)dt_1\cdots dt_n$.
Then we note that $ |S_n(x,\lambda)|\leq \frac{\|V\|^n_{L^2}}{n!|\lambda|^{\frac{n+1}{2}}}$ for $n\geq 1$ and we compute
$$|S_1(x,\lambda)+\frac{\mathcal V(x)}{2\lambda}\cos \sqrt\lambda x|\leq \frac{\|V\|_{H^1}}{|\lambda|^{\frac32}}.$$
Therefore we have
$$|y(x,\lambda)+\frac{\mathcal V(x)}{2\lambda}\cos \sqrt\lambda x|\leq \frac{\|V\|_{H^1}e^{\|V\|_{L^2}}}{|\lambda|^{\frac32}}.$$
In the other hand the Dirichlet spectrum $\{\lambda_j,\ j\geq1\}$ can be characterized as the roots of the equation $y(1,\lambda)=0$ and thus
the corresponding eigenfunctions are given by $f_j(x)=\frac{y(x,\lambda_j)}{\|y(\cdot,\lambda_j)\|_{L^2}}$, $j\geq1$ up to an inessential sign. Then a computation leads to the first formula of the proposition. The second one is proved similarly.
\end{proof}

Now we deduce some useful corollaries of these results.
\begin{lemma} \label{lem_est_der} For all $\rho>0$, there exists $C>0$ such that if $\|V\|_{H^1}\leq \rho$ and $k,n>0$ we have
$$
\left| \mathrm{d}\lambda_n(V)(\cos(2k\,\cdot \,)) + \frac{\mathbb{1}_{n=k}}2  \right|\leq \frac C{n } \left(\frac{1}n + \frac1{\langle n-k \rangle} \right)\|V\|_{H^1}, 
$$
$$
\left| \mathrm{d}\lambda_{-n}(V)(\cos(2k\,\cdot \,)) - \frac{\mathbb{1}_{n=k}}2 \right|\leq \frac C{n } \left(\frac{1}n + \frac1{\langle n-k \rangle} \right) \|V\|_{H^1}.
$$
\end{lemma}
\begin{proof} We focus on the Dirichlet spectrum. The Neuman case is similar.
First, we recall that by Proposition \ref{prop_smooth_spec}, we have
$$
\mathrm{d}\lambda_n(V)(\cos(2k\,\cdot \,)) = \int_0^\pi \cos(2kx) f_n^2(x) \, \mathrm{d}x.
$$
Now, applying Proposition \ref{prop_desc_fn}, we have
\begin{equation*}
\begin{split}
\left|f_n^2(x) - \frac1\pi - \frac1\pi \cos(2nx)  +  \frac1{\pi n} \mathcal{V}(x) \sin(2n x)  \right|
&=\left|f_n^2(x) - \frac2\pi \sin^2(nx) +  \frac2{\pi n} \mathcal{V}(x) \sin(n x) \cos(n x) \right|\\ 
&\lesssim_{\|V\|_{H^1}} \frac{\|V\|_{H^1}}{n^2}.
\end{split}
\end{equation*}
Therefore, to conclude, it is enough to note that
$\sin(2nx) \cos(2kx) = \frac12(\sin(2(k+n)x) - \sin(2(k-n)x))$
and if $\ell \neq 0$
$$
\int_0^\pi \sin(2\ell x) \mathcal{V}(x) \mathrm{d}x =\frac1{2\ell} \int_0^\pi \cos(2\ell x) V(x) \mathrm{d}x \lesssim \frac{\|V\|_{H^1}}{\ell}.
$$

\end{proof}

\begin{lemma} \label{lem_14}There exists $\rho_0>0$ such that provided $\|V\|_{H^1}\leq \rho_0$, for all $r>0$, all $0<n_1<\cdots<n_{r}$ and all $\ell\in (\mathbb{Z}^*)^{r}$, there exists $j\in  \llbracket 1,r \rrbracket$ and
$$
|\mathrm{d}\left( \ell_1 \lambda_{n_1}+\cdots+\ell_r \lambda_{n_r} \right)(V)(\cos(2n_jx))| \geq \frac14.
$$
\end{lemma}
\begin{proof}
Let $j \in \llbracket 1,r \rrbracket$ be such that $|\ell_j|=|\ell|_\infty$. Applying Lemma \ref{lem_est_der} with $\rho=1$, if $\|V\|_{H^1}\leq \rho$, by the second triangular inequality we have
\begin{equation*}
\begin{split}
|\mathrm{d}\left( \ell_1 \lambda_{n_1}+\cdots+\ell_r \lambda_{n_r} \right)(V)(\cos(2n_jx))| &\geq \frac{|\ell_j|}2 -C \|V\|_{H^1} \sum_{k=1}^{r} \frac{|\ell_k|}{n_k} \left(\frac{1}{n_k} + \frac1{\langle n_k-n_j \rangle} \right) \\
&\geq |\ell|_\infty \left( \frac12 - \|V\|_{H^1} C \sum_{k=1}^{r} \frac{1}{n_k^2} + \frac1{n_k\langle n_k-n_j \rangle}   \right)
\end{split}
\end{equation*}
Noting that by Cauchy-Schwarz, since $n$ is injective, both $\sum_{k=1}^{r} \frac{1}{n_k^2}$ and $\sum_{k=1}^{r} \frac1{n_k\langle n_k-n_j \rangle}$ are bounded uniformly with respect to $n$ and $j$, provided $\|V\|_{H^1}$ is small enough, we deduce that 
$$|\mathrm{d}\left( \ell_1 \lambda_{n_1}+\cdots+\ell_r \lambda_{n_r} \right)(V)(\cos(2n_jx))| \geq \frac{|\ell|_\infty}4 \geq \frac14 .$$ 
\end{proof}
As a first corollary, we are in position to prove Proposition \ref{prop_proba_dir_1d}.
\begin{proof}[\bf Proof of Proposition \ref{prop_proba_dir_1d}] Observing that $\| V_{|(0,\pi)}\|_{H^1} \leq \| V\|_{H^1(\mathbb{T})}$, we set $\rho := \rho_0$ (which is given by Lemma \ref{lem_14}). We are going to prove that, almost surely, provided that $\| V\|_{H^1(\mathbb{T})} \leq \rho$, for all $r>0$, there exists $\gamma,\alpha>0$ such that for all $r_\star \leq r$, $\ell \in (\mathbb{Z}^*)^{r_\star}$, all $n \in (\mathbb{N}^*)^{r_\star}$ satisfying $0<n_1<\cdots<n_{r_\star}$, $|\ell|_1 \leq r$, we have
\begin{equation}
\label{eq_weak_NLS}
\forall k \in \mathbb{Z},\forall h\in \llbracket -r,r \rrbracket, \ | \Omega_{k,h,\ell,n}(V) | \geq \gamma n_{r_\star}^{-\alpha}.
\end{equation}
where 
$\displaystyle \Omega_{k,h,\ell,n}(V) := k+\frac{h}\pi \int_0^\pi V(x) \mathrm{d}x + \ell_1 \lambda_{n_1}(V)+\cdots+\ell_r \lambda_{n_{r_\star}}(V).$
Indeed, then, the result follows directly of Proposition \ref{prop_it_acumulates} (to justify the accumulation on $\mathbb{Z}+ \frac1\pi \int_0^\pi V(x) \mathrm{d}x$ of the eigenvalues) and Proposition \ref{prop_weak_become_strong} (to replace $n_{r_\star}$ by $n_{1}$).

All the parameters of \eqref{eq_weak_NLS} being fixed, we aim at estimating
$\mathbb{P}\big(|\Omega_{k,h,\ell,n}(V)| < \varepsilon \ \mathrm{and} \ \| V\|_{H^1(\mathbb{T})} \leq \rho \big)$ where $\varepsilon>0$. Let $j \in \llbracket 1,r_{\star} \rrbracket$ be the index given by Lemma \ref{lem_14}. Setting $V^{(-2n_j)} = V- V_{2n_j} \langle 2n_j \rangle^{-s} \cos(2n_jx)$, $V^{(-2n_j)}$ and $V_{2n_j}$ are independent. Therefore, we have
\begin{equation}
\label{eq_J_reviendrai}
\begin{split}
&\mathbb{P}\big(|\Omega_{k,h,\ell,n}(V)| < \varepsilon \ \mathrm{and} \ \| V\|_{H^1(\mathbb{T})} \leq \rho \big)\\
 =& \frac1{\sqrt{2\pi}}\mathbb{E} \Big[\int_{v_{2n_j} \in \mathcal{I}} \mathbb{1}_{|\Omega_{k,h,\ell,n}(V^{(-2n_j)}+ v_{2n_j} \langle 2n_j \rangle^{-s} \cos(2n_jx) )| < \varepsilon } \ e^{-\frac{v_{2n_j}^2}2}\mathrm{d}v_{2n_j} \Big] \\
\leq& \mathbb{E} \Big[\int_{v_{2n_j} \in \mathcal{I}} \mathbb{1}_{|\Omega_{k,h,\ell,n}(V^{(-2n_j)}+ v_{2n_j} \langle 2n_j \rangle^{-s} \cos(2n_jx) )| < \varepsilon } \ \mathrm{d}v_{2n_j} \Big]
\end{split}
\end{equation}
where 
$$
\mathcal{I} := \{ v_{2n_j} \in \mathbb{R} \ | \ \| V^{(-2n_j)}+ v_{2n_j} \langle 2n_j \rangle^{-s} \cos(2n_jx) \|_{H^1(\mathbb{T})}^2 \equiv \pi v_{2n_j}^2 \langle 2n_j \rangle^{-2s+2} + \| V^{(-2n_j)}\|_{H^1(\mathbb{T})}^2 \leq \rho^2 \}.$$
 By definition of $\mathcal{I}$ and $\rho$, it follows from Lemma \ref{lem_14} that, almost surely, we have
$$
\forall v_{2n_j} \in \mathcal{I}, \ |\partial_{v_{2n_j}} \Omega_{k,h,\ell,n}(V^{(-2n_j)}+ v_{2n_j} \langle 2n_j \rangle^{-s} \cos(2jx) )| \geq \frac14 \langle 2n_j \rangle^{-s}.
$$
Since $\mathcal{I}$ is a (random) interval, $v_{2n_j} \mapsto \Omega_{k,h,\ell,n}(V^{(-2n_j)}+ v_{2n_j} \langle 2n_j \rangle^{-s} \cos(2n_jx) )$ is a diffeomorphism onto its image (denoted $\mathcal{J}$). Therefore, by the change of variables formula, we have (almost surely)
$$
\int_{v_{2n_j} \in \mathcal{I}} \mathbb{1}_{|\Omega_{k,h,\ell,n}(V^{(-2n_j)}+ v_{2n_j} \langle 2n_j \rangle^{-s} \cos(2n_jx) )| < \varepsilon } \ \mathrm{d}v_{2n_j} \leq 4 \langle 2n_j \rangle^{s} \int_{v_{2n_j} \in \mathcal{J}} \mathbb{1}_{|v_{2n_j}| < \varepsilon } \ \mathrm{d}v_{2n_j} \leq 8 \langle 2n_j \rangle^{s} \varepsilon.
$$
As a consequence of \eqref{eq_J_reviendrai}, we have
$$
\mathbb{P}\big(|\Omega_{k,h,\ell,n}(V)| < \varepsilon \ \mathrm{and} \ \| V\|_{H^1(\mathbb{T})} \leq \rho \big)  \leq 8 \langle 2n_{n_{r_\star}} \rangle^{s}  \varepsilon.
$$
Hence, requiring implicitly that $(r,r_\star,k,h,\ell,n)$ satisfy the constraints described previously, we deduce that
\begin{equation*}
\begin{split}
&\mathbb{P}\big(\exists (r,r_\star,k,h,\ell,n) ,\ |\Omega_{k,h,\ell,n}(V)| < \langle k \rangle^{-2} r^{-2r} \langle n_{r_\star} \rangle^{-(s+2r)} \varepsilon \quad \mathrm{and} \quad \| V\|_{H^1(\mathbb{T})} \leq \rho \big) \\
\leq& \sum_{(r,r_\star,k,h,\ell,n)} \mathbb{P}\big( |\Omega_{k,h,\ell,n}(V)| < \langle k \rangle^{-2} r^{-2r} \langle n_{r_\star} \rangle^{-(s+2r)} \varepsilon \quad \mathrm{and} \quad \| V\|_{H^1(\mathbb{T})} \leq \rho \big) \\
\lesssim_s &   \big( \sum_{(r,r_\star,k,h,\ell,n)} \langle k \rangle^{-2} r^{-2r} \langle n_{r_\star} \rangle^{-2r} \big) \ \varepsilon  \ \approx_s \ \varepsilon \ \mathop{\longrightarrow}_{\varepsilon \to 0} \ 0.
\end{split}
\end{equation*}
We omit the details for the convergence of this series but the proof is quite straightforward (actually the factor $r^{-2r}$ is far to be sharp). Since this probability vanishes as $\varepsilon$ goes to $0$, we deduce that, almost surely, there exists $\varepsilon>0$ such that
\begin{equation}
\label{eq_quasi_bon}
\forall (r,r_\star,k,h,\ell,n), \ |\Omega_{k,h,\ell,n}(V)| \geq \varepsilon \langle k \rangle^{-2} r^{-2r} \langle n_{r_\star} \rangle^{-(s+2r)} .
\end{equation}
Therefore, to get \eqref{eq_weak_NLS}, we would like this estimate to be uniform with respect to $k$. Fortunately, by the triangle inequality and Proposition \ref{prop_it_acumulates} we have
$$
 |\Omega_{k,h,\ell,n}(V)| \geq |k| - r C \langle n_{r_\star} \rangle^{2}.
$$
where $C\geq 1$ depends only on $\rho$. Thus, if $|k|\geq 2 r C \langle n_{r_\star} \rangle^{2}$, we know that $|\Omega_{k,h,\ell,n}(V)| \geq 1$. Therefore, we can replace $k$ by $ 2 r C \langle n_{r_\star} \rangle^{2}$ in the right hand side term of \eqref{eq_quasi_bon} which provides \eqref{eq_weak_NLS} and concludes this proof.
\end{proof}

Now we focus on the periodic spectrum. First, we focus on the mode $n=0$.
\begin{lemma}
\label{lem_mode0}
For all $\rho>0$, there exists $C>0$ such that if $\|V\|_{L^2(0,\pi)}\leq \rho$, $k>0$ and $x\in (0,\pi)$ we have
$$
\big|f_0(x) - \frac1{\sqrt\pi}\big|\leq C \| V\|_{L^2} \quad \mathrm{and} \quad |\mathrm{d}\lambda_0(V)(\cos(2k \, \cdot \,))| \leq  C \| V\|_{L^2}. 
$$
\end{lemma}
\begin{proof} The first estimate follows of the smoothness of $V \mapsto f_0$ (see Proposition \ref{prop_smooth_spec}) while the second estimate is a direct consequence of the first estimate and the expression of $\mathrm{d}\lambda_0$ given by the formula \eqref{eq_diff_lambda} of Proposition \ref{prop_smooth_spec}.
\end{proof}

Then we adapt Lemma \ref{lem_14} to the periodic case (it is easier to change the formalism to present it).
\begin{lemma}
\label{lem_14_per}
For all $r>0$, there exists $\rho_r>0$ such that for all $\ell \in \mathbb{Z}^{\mathbb{Z}} $, if $\|\ell\|_{\ell^1}\leq r$ and $\|V\|_{H^1(0,\pi)}\leq \rho_r$, either $\ell$ is even\footnote{i.e. $\ell_j=\ell_{-j}$, for all $j\in \mathbb{Z}$.} or there exists $j\in\mathbb{Z}$ such that $\ell_j \neq \ell_{-j}$ and
\begin{equation}
\label{eq_14_retour}
|\mathrm{d}( \sum_{n \in \mathbb{Z}} \ell_n \lambda_n)(V)(\cos(2j \, x\,))  | \geq \frac14.
\end{equation}
\end{lemma}
\begin{proof} Assume that $\|V\|_{H^1}\leq 1$ and let $C$ be the constant of Lemma \ref{lem_mode0} and Lemma \ref{lem_est_der} associated with $\rho =1$.
Let $\ell^{\mathrm{odd}}$ (resp. $\ell^{\mathrm{even}}$) be the odd (resp. even) part of $\ell$. 

On the one hand, by Lemma \ref{lem_mode0} and Lemma \ref{lem_est_der}, for all $j\in \mathbb{Z}$, we have
$$
|\mathrm{d}( \sum_{n \in \mathbb{Z}} \ell_n^{\mathrm{even}} \lambda_n)(V)(\cos(2j \, x\,))  | =\frac12 |\mathrm{d}( \sum_{n \in \mathbb{Z}} \ell_n^{\mathrm{even}} (\lambda_n+\lambda_{-n}))(V)(\cos(2j \, x\,)) | \leq  C r \|V\|_{H^1}. 
$$
On the other hand, if $j>0$ is such that $\ell^{\mathrm{odd}}_j \neq 0$ then  $|\ell^{\mathrm{odd}}_j|\geq 1/2$ and by Lemma \ref{lem_est_der}, we have
\begin{equation*}
\begin{split}
|\mathrm{d}( \sum_{n \in \mathbb{Z}} \ell_n^{\mathrm{odd}} \lambda_n)(V)(\cos(2j \, x\,))  | = |\mathrm{d}( \sum_{n >0} \ell_n^{\mathrm{odd}} (\lambda_n-\lambda_{-n}))(V)(\cos(2j \, x\,)) | &\geq |\ell^{\mathrm{odd}}_j | - Cr \| V\|_{H^1}\\ &\geq \frac12 - Cr \| V\|_{H^1}. 
\end{split}
\end{equation*}
As a consequence, provided that $r \| V\|_{H^1}$ is small enough, we have \eqref{eq_14_retour}.
\end{proof}

As a consequence, we deduce that most of the monomials are weakly non resonant.
\begin{lemma}
\label{lem_per_nonevenok} Let $V$ be the even random potential defined in Proposition \ref{prop_proba_per_1d}. For all $r>0$, there exists $\rho,\beta>0$, such that, almost surely, there exists $\gamma>0$ such that for all $r_\star \leq r$, $n\in \mathbb{Z}^{r_\star}$ injective and $\ell \in (\mathbb{Z}^*)^{r_\star}$ satisfying $|\ell|_1 \leq r$ and $\langle n_1 \rangle \leq \cdots \leq \langle n_{r_\star} \rangle$, provided that $\| V\|_{H^1} < \rho$, either $ k \mapsto \sum_{j=1}^{r_\star} \ell_j \mathbb{1}_{ k = n_j } $ is even or
$$
\forall k\in \mathbb{Z}, \forall h\in \llbracket -r,r \rrbracket, \ |k+\frac{h}\pi \int_0^\pi V(x) \, \mathrm{d}x + \sum_{j=1}^{r_\star} \ell_j \lambda_{n_j}|\geq \gamma  \langle n_{r_\star} \rangle^{-\beta}.
$$
\end{lemma}
\begin{proof} It is very similar to the proof of Proposition \ref{prop_proba_dir_1d} excepted that we use Lemma \ref{lem_14_per} instead of Lemma \ref{lem_14} (it explains why $\rho$ depends on $r$ and we have to exclude some special multi-indices).
\end{proof}

To deal with the multi-indices we have excluded in Lemma \ref{lem_per_nonevenok}, we have to estimate the second derivative of $\lambda_n$. 
\begin{lemma} 
\label{lem_der2_lamb_abs}For all $V,W\in L^2(0,\pi;\mathbb{R})$, we have
\begin{equation}
\label{eq_der2_lambda_dir}
\mathrm{d}^2\lambda_n(V)(W,W) = 2\sum_{\substack{k \geq 1\\ k\neq n}} \frac1{\lambda_n - \lambda_k} \left( \int_0^\pi W(x) f_n(x)  f_k(x) \, \mathrm{d}x \right)^2 , \quad \mathrm{if} \quad n\geq 1,
\end{equation}
\begin{equation}
\label{eq_der2_lambda_neu}
\mathrm{d}^2\lambda_n(V)(W,W) = 2\sum_{\substack{k \leq 0\\ k\neq n}} \frac1{\lambda_n - \lambda_k} \left( \int_0^\pi W(x) f_n(x)  f_k(x) \, \mathrm{d}x \right)^2 , \quad \mathrm{if} \quad n\leq 0.
\end{equation}
\end{lemma}
\begin{proof}
We focus on the calculus of the second derivative of the Dirichlet spectrum. The calculus for the Neuman spectrum is similar. We follow the strategy of \cite{BG06} section 5.3. We recall that by Proposition \ref{prop_smooth_spec} ,$\lambda_n$ and $f_n$ depend smoothly on $V\in L^2$. For compactness, we denote $f_n',\lambda_n'$ (resp. $f_n'',\lambda_n''$) the first (resp. second) derivative in the direction $W$.

First, we note that, since $\|f_n\|_{L^2}^2 = 1$, $f_n'$ and $f_n$ are orthogonal in $L^2$ : $(f_n,f_n')_{L^2}=0$. Then differentiating the relation $(-\partial_x^2 +V )f_n = \lambda_n f_n $, we get
\begin{equation}
\label{eq_first_der}
(-\partial_x^2 +V )f_n' + W f_n= \lambda_n f_n' + \lambda_n' f_n.
\end{equation}
Since $f_n'$ and $f_n$  are orthogonal, we deduce that
\begin{equation}
\label{eq_fnp}
f_n'= -(-\partial_x^2 +V - \lambda_n)^{-1} (\mathrm{Id}_{L^2} - \Pi)(W f_n)
\end{equation}
where $\Pi$ is the orthogonal projector on $\mathrm{Span}(f_n)$. Differentiating once again \eqref{eq_first_der}, we get
$$
(-\partial_x^2 +V )f_n'' +2 W f_n'=\lambda_n'' f_n + 2\lambda_n' f_n' +\lambda_n f_n''.
$$
The scalar product of this relation against $f_n$ provides
$$
2 (f_n,W f_n')_{L^2} = \lambda_n''.
$$
Therefore, using the formula \eqref{eq_fnp} and decomposing $W f_n$ in the Hilbertian basis $(f_k)_{k\geq 1}$, we get the relation \eqref{eq_der2_lambda_dir} we wanted to prove.
\end{proof}

\begin{lemma}
\label{lem_est_d2} For all $\rho >0$, there exists $C>0$ such that for all $n\in \mathbb{Z}$, $j\in \mathbb{N}^* \setminus \{ |n|,2|n|\}$ and $\|V\|_{H^1(0,\pi)} \leq \rho$, we have
$$
\left| \partial^2_{\cos(j \,\cdot\,)}\lambda_n -  \frac1{4n^2 -j^2} \right| \leq C \|V\|_{H^1}.
$$

\end{lemma}
\begin{proof} We focus on the calculus of the second derivative of the Dirichlet spectrum (i.e. $n,k>0$). The calculus for the Neuman spectrum is similar. First, by Proposition \ref{prop_it_acumulates}, we note that : 
$$
\left| \frac1{\lambda_n - \lambda_k} - \frac1{n^2 - k^2} \right| \lesssim_{\|V\|_{H^1}}  \frac{\|V\|_{H^1}}{|n^2-k^2|}.
$$
Then by Proposition \ref{prop_desc_fn},  we have
\begin{multline*}
\int_0^\pi \cos(j x) f_n(x)  f_k(x) \, \mathrm{d}x + \mathcal{O}(\| V\|_{H^1})  = \frac2\pi \int_0^\pi \cos(j x) \sin( n x)  \sin(k x) \, \mathrm{d}x  \\
= \frac1\pi \int_0^\pi \sin(kx) (\sin((n+j)x)+\sin((n-j)x)) \, \mathrm{d}x =  \frac{\mathbb{1}_{k=n+j} + \mathrm{sg}(n-j)  \mathbb{1}_{k=\pm (n-j)} }2
\end{multline*}
where $\mathcal{O}(\| V\|_{H^1})$ is uniform with respect to $(n,k,j)$ and $\mathrm{sg}$ is the sign function. Therefore, since $n\neq 0\neq j$, we have
$$
\left( \int_0^\pi \cos(j x) f_n(x)  f_k(x) \, \mathrm{d}x \right)^2 = \frac{\mathbb{1}_{k=n+j} + \mathbb{1}_{k=\pm (n-j)}  }4 + \mathcal{O}(\| V\|_{H^1}).
$$
As a consequence of Lemma \ref{lem_der2_lamb_abs}, applying the Cauchy Schwarz inequality, we have
$$
\Big| \partial^2_{\cos(j \,\cdot\,)}\lambda_n - \frac12\sum_{\substack{k \geq 1\\ k\neq n}} \frac{\mathbb{1}_{k=n+j} +   \mathbb{1}_{k=\pm (n-j)}}{ n^2-k^2} \Big|\lesssim_{\|V\|_{H^1}}  \sum_{\substack{k \geq 1\\ k\neq n}} \frac{\|V\|_{H^1}}{|n^2-k^2|} \lesssim_{\|V\|_{H^1}} \frac{\pi^2}3 \|V\|_{H^1}.
$$
Finally, since $j\neq 2n$, we have
$$
\sum_{\substack{k \geq 1\\ k\neq n}} \frac{\mathbb{1}_{k=n+j} +   \mathbb{1}_{k=\pm (n-j)}}{ n^2-k^2} = \frac{1}{n^2 - (n+j)^2} + \frac{1}{n^2 - (n-j)^2} = -\frac1{j(2n+j)} + \frac1{j(2n-j)} = \frac2{4n^2-j^2}.
$$
\end{proof}

The following lemma deals with the leading part of the second derivative of the small divisors. Its proof is similar to the proof of Proposition \ref{prop_weak_become_strong}.
\begin{lemma}\label{lem_SD_frac}
For all $r>0$ and $r_\star \leq r$, there exists $\gamma_{r,r_\star}>0$ and $\beta_{r_\star}>0$ such that for all $\ell \in (\mathbb{Z}^*)^{r_\star}$, for all $n \in \mathbb{N}^{r_\star}$ satisfying $0\leq n_1 < \dots < n_{r_\star}$ and $|\ell|_1 \leq r$, there exists $j\in \llbracket 1,5r_\star \rrbracket \setminus \bigcup_{k=1}^{r_\star}\{ n_k,2n_k\}$ such that
\begin{equation}
\label{eq_SD_rat}
\left| \sum_{k=1}^{r_\star} \frac{\ell_k}{4n_k^2-j^2}  \right| \geq \gamma_{r,r_\star} \langle n_1 \rangle^{-\beta_{r_\star}}.
\end{equation}
\end{lemma}
\begin{proof} Let fix $r>0$ and let us prove this result by induction on $r_\star$. Indeed, if $r_\star = 1$ then for all $n_1\in \mathbb{N}$ there exists $j\in \llbracket 1,3 \rrbracket \setminus \{n_1,2n_1\}$ such that $|4n_1^2-j^2 |\leq 4 n_1^2 + 9 \leq 9 \langle n_1 \rangle^2$. Consequently \eqref{eq_SD_rat} holds with $\beta_1=2$ and $\gamma_{r,1} = 1/9$.

Now assume that $r_\star < r$, $\ell \in (\mathbb{Z}^*)^{r_\star+1}$, $n \in \mathbb{N}^{r_\star+1}$ satisfy $0\leq n_1 < \dots < n_{r_\star}$, $|\ell|_1\leq r$ and are such that there exists $j\in \llbracket 1,5r_\star \rrbracket \setminus \bigcup_{k=1}^{r_\star}\{ n_k,2n_k\}$ such that \eqref{eq_SD_rat} holds.

If $n_{r_\star+1}> 5 r_\star$ then by the triangle inequality, we have
$$
\left| \sum_{k=1}^{r_\star+1} \frac{\ell_k}{4n_k^2-j^2}  \right| \geq \gamma_{r,r_\star} \langle n_1 \rangle^{-\beta_{r_\star}} -  \frac{r}{4n_{r_\star+1}^2-j^2} 
$$
Therefore, if $n_{r_\star+1}  > \sqrt{25r_\star^2 +2r \gamma_{r,r_\star}^{-1} \langle n_1 \rangle^{\beta_{r_\star}} }$, we have
$$
\left| \sum_{k=1}^{r_\star+1} \frac{\ell_k}{4n_k^2-j^2}  \right| \geq \frac{\gamma_{r,r_\star}}2 \langle n_1 \rangle^{-\beta_{r_\star}}.
$$
As a consequence, now we assume that $n_{r_\star+1}  \leq  \sqrt{25r_\star^2 +2r \gamma_{r,r_\star}^{-1} \langle n_1 \rangle^{\beta_{r_\star}} }$. We note that there exists $P\in \mathbb{Z}[X]$ such that
$$
\sum_{k=1}^{r_\star+1} \frac{\ell_k}{4n_k^2-X^2} =    \frac{P(X)}{(4n_1^2-X^2) \cdots (4n_{r_\star+1}^2-X^2)  }
$$
where $P\neq 0$ (by the uniqueness of the partial fraction decomposition) and $\deg P \leq 2 (r_\star +1)$. Therefore $P$ has at most $2 (r_\star +1)$ roots. Consequently there exists $j_\star\in \llbracket 1,5r_\star \rrbracket \setminus \bigcup_{k=1}^{r_\star+1}\{ n_k,2n_k\}$ such that $P(j_\star) \neq 0$ and so $|P(j_\star)|\geq 1$. As a consequence, we have
\begin{multline*}
\left| \sum_{k=1}^{r_\star+1} \frac{\ell_k}{4n_k^2-j_\star^2}  \right| \geq ((4n_1^2-j_\star^2) \cdots (4n_{r_\star+1}^2-j_\star^2) )^{-1} \geq (4n_{r_\star+1}^2+25r_\star^2)^{-(r_\star+1)} \\
\geq (8r \gamma_{r,r_\star}^{-1} \langle n_1 \rangle^{\beta_{r_\star}}+125r_\star^2)^{-(r_\star+1)}  \approx_{r,r_\star} \langle n_1 \rangle^{ -(r_\star+1) \beta_{r_\star}}.
\end{multline*}
\end{proof}

As a corollary of Lemmata \ref{lem_est_d2} and \ref{lem_SD_frac}, we get the following lower bound on the second derivative of the small divisor.
\begin{corollary} \label{cor_der2}For all $r>0$ and $N\geq 1$, there exists $\rho,\eta>0$ such that for all $r_\star \leq r$, $\ell \in (\mathbb{Z}^*)^{r_\star}$, $n\in \mathbb{N}^{r_\star}$, all $V\in H^1(0,\pi;\mathbb{R})$ satisfying $\|V\|_{H^1(0,\pi)} \leq \rho$, $0\leq n_1 < \dots < n_r$, $|\ell|_{1} \leq r$ and $\langle n_1 \rangle \leq N$ there exists $j\in \llbracket 1,5r_\star \rrbracket$ such that
$$
\left| \partial^2_{\cos(j \,\cdot\,)} \sum_{k=1}^{r_\star} \ell_k (\lambda_{n_k} + \lambda_{-n_k} ) \right| \geq \eta.
$$
\end{corollary}
\begin{proof} Let $j\in \llbracket 1,5r_\star \rrbracket \setminus \bigcup_{k=1}^{r_\star}\{ n_k,2n_k\}$ be the index given by Lemma \ref{lem_SD_frac}. Assume that $\|V\|_{H^1} \leq 1$ and let $C$ be the constant given by Lemma \ref{lem_est_d2} associated with $\rho=1$. Therefore, as a consequence of Lemmata \ref{lem_est_d2} and \ref{lem_SD_frac}, we have
$$
\left| \partial^2_{\cos(j \,\cdot\,)} \sum_{k=1}^{r_\star} \ell_k (\lambda_{n_k} + \lambda_{-n_k} ) \right| \geq 2\left| \sum_{k=1}^{r_\star} \frac{\ell_k}{4n_k^2-j^2}  \right| - 2 r C \|V\|_{H^1} \geq 2\gamma_{r,r_\star} N^{-\beta_{r_\star}} - 2 r C \|V\|_{H^1}.
$$
Therefore, we just have to set $\eta = \gamma_{r,r_\star} N^{-\beta_{r_\star}}$ and $\rho = \min_{1\leq r_\star \leq r} \min(1, \gamma_{r,r_\star} N^{-\beta_{r_\star}} / 2rC)$.
\end{proof}
As a consequence of this corollary, we are in position to prove that the multi-indices we have excluded in Lemma \ref{lem_per_nonevenok} are actually also weakly non resonant. 
\begin{lemma}
\label{lem_per_evenok} Let $V$ be the even random potential defined in Proposition \ref{prop_proba_per_1d}. For all $r>0$ and $N\geq 1$, there exists $\rho>0$ and, almost surely, $\gamma>0$ such that for all $r_\star \leq r$, $n\in \mathbb{Z}^{r_\star}$ injective and $\ell \in (\mathbb{Z}^*)^{r_\star}$ satisfying $|\ell|_1 \leq r$, $\langle n_1 \rangle \leq \cdots \leq \langle n_{r_\star} \rangle$ and $\langle n_1 \rangle \leq N$, if $\| V\|_{H^1(0,\pi)} < \rho$ and $ k \mapsto \sum_{j=1}^{r_\star} \ell_j \mathbb{1}_{ k = n_j } $ is even then we have
$$
\forall k\in \mathbb{Z}, \forall h\in \llbracket -r,r \rrbracket, \ |k+\frac{h}\pi \int_0^\pi V(x) \, \mathrm{d}x + \sum_{j=1}^{r_\star} \ell_j \lambda_{n_j}|\geq \gamma  \langle n_{r_\star} \rangle^{-4 r_\star}.
$$
\end{lemma}
\begin{proof} Let $\rho$ and $\eta$ be given by Corollary \ref{cor_der2} (for $r\leftarrow 2r$). We denote 
$$
\Omega_{k,h,\ell,n}(V)=k+\frac{h}\pi \int_0^\pi V(x) \, \mathrm{d}x + \sum_{j=1}^{r_\star} \ell_j \lambda_{n_j}(V).
$$
 Implicitly, we always assume that $(r_\star,k,h,\ell,n)$ are such that $|\ell|_1 \leq r$, $\langle n_1 \rangle \leq \cdots \leq \langle n_{r_\star} \rangle$,  $\langle n_1 \rangle \leq N$ and $ k \mapsto \sum_{j=1}^{r_\star} \ell_j \mathbb{1}_{ k = n_j } $ is even. In order to estimate $\mathbb{P}(|\Omega_{k,h,\ell,n}(V)| < \varepsilon \quad \mathrm{and} \quad \|V\|_{H^1} \leq \rho )$ for $\varepsilon>0$, we consider $j\in \llbracket 1,5r_\star \rrbracket$ the index given by Corollary \ref{cor_der2} and we denote $V^{(-j)}=V - V_j \langle j\rangle^{-s} \cos(j \, \cdot)$. Since by assumption $V^{(-j)}$ and $V_j$ are independent, we have
\begin{equation}
\label{eq_J_reviendrai_2}
\begin{split}
\mathbb{P}\big(|\Omega_{k,h,\ell,n}(V)| < \varepsilon \ \mathrm{and} \ \| V\|_{H^1(0,\pi)} \leq \rho \big) &= \frac1{\sqrt{2\pi}}\mathbb{E} \Big[\int_{v_{j} \in \mathcal{I}} \mathbb{1}_{|\Omega_{k,h,\ell,n}(V^{(-j)}+ v_{j} \langle j \rangle^{-s} \cos(jx) )| < \varepsilon } \ e^{-\frac{v_{j}^2}2}\mathrm{d}v_{j} \Big] \\
&\leq  \mathbb{E} \Big[\int_{v_{j} \in \mathcal{I}} \mathbb{1}_{|\Omega_{k,h,\ell,n}(V^{(-j)}+ v_{j} \langle j \rangle^{-s} \cos(jx) )| < \varepsilon }  \mathrm{d}v_{j} \Big] \\
\end{split}
\end{equation}
where $\mathcal{I} := \{ v_{j} \in \mathbb{R} \ | \ \| V^{(-j)}+ v_{j} \langle j \rangle^{-s} \cos(jx) \|_{H^1(0,\pi)}^2 \equiv \frac{\pi}2 v_{j}^2 \langle j \rangle^{-2s+2} + \| V^{(-j)}\|_{H^1(0,\pi)}^2 \leq \rho^2 \}$.  By definition of $\mathcal{I}$ and Corollary \ref{cor_der2}, almost surely, for all $v_j \in \mathcal{I}$, we have
$$
|\partial_{v_j}^2 \Omega_{k,h,\ell,n}(V^{(-j)}+ v_{j} \langle j \rangle^{-s} \cos(jx) )| \geq \langle j \rangle^{-2s} \eta \geq  \langle 5r \rangle^{-2s} \eta.
$$
Therefore, since $\mathcal{I}$ is (random) interval and $\rho,\eta$ depends only on $(r,N)$, applying Lemma B.1. of \cite{Eli}, we deduce that, almost surely,
$$
\int_{v_{j} \in \mathcal{I}} \mathbb{1}_{|\Omega_{k,h,\ell,n}(V^{(-j)}+ v_{j} \langle j \rangle^{-s} \cos(jx) )| < \varepsilon }  \mathrm{d}v_{j} \lesssim_{r,N,s} \sqrt{\varepsilon}.
$$
As a consequence of \eqref{eq_J_reviendrai_2}, we deduce that
\begin{equation*}
\begin{split}
& \ \ \ \mathbb{P}\big(\exists (r_\star,k,h,\ell,n), \, |\Omega_{k,h,\ell,n}(V)| < \varepsilon \langle n_{r_\star} \rangle^{-4 r_\star} \ \mathrm{and} \ \| V\|_{H^1(0,\pi)} \leq \rho \big) \\ & \leq \sum_{ (r_\star,k,h,\ell,n)} \mathbb{P}\big(\exists (r_\star,k,h,\ell,n), \, |\Omega_{k,h,\ell,n}(V)| < \varepsilon \langle n_{r_\star} \rangle^{-4 r_\star} \ \mathrm{and} \ \| V\|_{H^1(0,\pi)} \leq \rho \big) \\
& \lesssim_{r,N,s} \sqrt{\varepsilon} \sum_{ (r_\star,k,h,\ell,n)}  \langle n_{r_\star} \rangle^{-2 r_\star} \approx_{r,N,s} \sqrt{\varepsilon} \ \mathop{\longrightarrow}_{\varepsilon \to 0} \ 0.
\end{split}
\end{equation*}
\end{proof}

\begin{proof}[\bf Proof of Proposition \ref{prop_proba_per_1d}] We recall that since the potential is even, the periodic spectrum is given by the Dirichlet spectrum and the Neuman spectrum (see Proposition \ref{prop_def_SL}). Therefore, we just have to apply Proposition \ref{prop_weak_become_strong}. Indeed Lemmata \ref{lem_per_nonevenok} and \ref{lem_per_evenok} ensure the weak non-resonant assumption while Proposition \ref{prop_it_acumulates} ensure the accumulation property.
\end{proof}

\section{Functional setting and the class of Hamiltonian functions}
\label{sec_class_Ham}

First, to avoid any possible confusion, we specify our functional setting and the associated differential calculus formalism. Indeed, since we are dealing with non-smooth solutions, it is especially important to have very precise definitions. Then, in a second subsection, we introduce our Hamiltonian formalism.

\subsection{Functional setting and its differential calculus formalism}
 We equip $\mathbb{C}$ of its natural real scalar product
$$
 \Re (\overline{z_1} z_2) = \Re z_1 \Re z_2 + \Im z_1 \Im z_2.
$$
If $f$ is a $C^1$ function on $\mathbb{C}$, we define as usual
$$
\partial_{\Re z} f(z) = \mathrm{d}f(z)(1) \quad \mathrm{and} \quad \partial_{\Im z} f(z) = \mathrm{d}f(z)(i) .
$$
As a consequence, if $f$ is real valued, its gradient writes
$$
\nabla f(z) = \partial_{\Re z} f(z) +  i  \partial_{\Im z} f(z).
$$
We extend this formula to non-real valued functions by
$$
2 \partial_{\bar{z}} f(z) := \partial_{\Re z} f(z) +  i  \partial_{\Im z} f(z) \quad \mathrm{and} \quad 2 \partial_{z} f(z) := \partial_{\Re z} f(z) -  i  \partial_{\Im z} f(z)
$$
Being given a set $\mathbf{Z}_d \subset \mathbb{Z}^d$ where $d\geq 1$, $s\in \mathbb{R}$ and $p\geq 1$, we define the discrete Sobolev  and Lebesgue spaces
\begin{equation}
\label{eq_def_hs}
h^s(\mathbf{Z}_d)=\big\{ u\in \mathbb{C}^{\mathbf{Z}_d} \ | \ \|u\|_{h^s}^2 := \sum_{ k \in \mathbf{Z}_d} \langle k \rangle^{2s} |u_k|^2 < \infty   \big\}
\end{equation}
$$
\ell^p(\mathbf{Z}_d)=\big\{ u\in \mathbb{C}^{\mathbf{Z}_d} \ | \ \|u\|_{\ell^p}^p := \sum_{ k \in \mathbf{Z}_d}  |u_k|^p < \infty   \big\}.
$$
We equip $\ell^2(\mathbf{Z}_d):= h^0(\mathbf{Z}_d)$ of its natural real scalar product
$$
(u,v)_{\ell^2}:=  \sum_{k\in \mathbf{Z}_d}  \Re (\overline{u_k} v_k) \in \mathbb{R}.
$$
As usual we extend this scalar product when $u\in h^s$ and $v\in h^{-s}$.
Being given a smooth function $H : h^s(\mathbf{Z}_d) \to \mathbb{R}$ and $u\in h^s(\mathbf{Z}_d)$, its gradient $\nabla H(u)$ is the unique element of $h^{-s}(\mathbf{Z}_d)$ satisfying
$$
\forall v\in h^{s}(\mathbf{Z}_d), \ (\nabla H(u),v)_{\ell^2} = \mathrm{d}H(u)(v).
$$
Note that it can be checked that
$$
\nabla H(u) = (2\partial_{\overline{u}_k} H(u))_{k\in \mathbf{Z}_d}.
$$
If $H,K : h^s(\mathbf{Z}_d) \to \mathbb{R}$ are two functions such that $\nabla H$ is  $h^{s}(\mathbf{Z}_d)$ valued then the Poisson bracket of $H$ and $K$ is defined by
$$
\{  H,K\}(u):= (i \nabla H(u),\nabla K(u))_{\ell^2}.
$$
Note that, as expected, we have
\begin{multline}
\label{eq_poisson_coord}
\{  H,K\}  = 4 \sum_{k\in \mathbf{Z}_d} \Re \big[i\partial_{\bar{u}_k}H(u) \overline{\partial_{\bar{u}_k} K(u)} \big]=  4 \sum_{k\in \mathbf{Z}_d} \big[\Re i\partial_{\bar{u}_k}H(u) \partial_{u_k} K(u)\big] \\
  =2i \sum_{k\in \mathbf{Z}_d} \partial_{\bar{u}_k}H(u) \partial_{u_k} K(u) - \partial_{u_k}H(u) \partial_{\bar{u}_k} K(u).
\end{multline}
Finally, the symplectic transformations are defined as follow.
\begin{definition}[symplectic map]
\label{def_symplectic} Let $s\geq 0$, $\Omega$ an open set of $h^s(\mathbf{Z}_d)$ and a $C^1$ map $\tau :  \Omega \to h^s(\mathbf{Z}_d)  $. The map $\tau$ is symplectic if it preserves the canonical symplectic form :
$$
\forall u \in \Omega, \forall v,w \in h^s(\mathbf{Z}_d), \ (iv,w)_{\ell^2} = (i\mathrm{d}\tau(u)(v),\mathrm{d}\tau(u)(w))_{\ell^2}.
$$
\end{definition}

\subsection{The class of Hamiltonian functions}
In this section, we aim at establishing the main properties of the following class of Hamiltonians.
\begin{definition} 
\label{def_class_ham} For $\mathbf{Z}_d\subset \mathbb{Z}^d$, $q,\alpha\geq 0$, $r\geq 2$ let $\mathscr{H}_{q,\alpha}^{r}(\mathbf{Z}_d)$ be the real Banach space of the $\alpha$-inhomogeneous, $q$-smoothing formal Hamiltonians of degree $r$ supported on $\mathbf{Z}_d$. More precisely, they are the formal Hamiltonians of the form
$$
H(u) = \sum_{\sigma \in \{-1,1\}^r} \sum_{n\in \mathbf{Z}_d^r } H_{n}^{\sigma} u_{n_1}^{\sigma_1} \dots u_{n_r}^{\sigma_r}
$$
with $H_{n}^{\sigma} \in \mathbb{C}$, satisfying the reality condition
\begin{equation}
\label{eq_def_real}
H_{n}^{-\sigma} = \overline{H_{n}^{\sigma}}
\end{equation}
the symmetry condition 
\begin{equation}
\label{def_sym_cond}
\forall \phi \in \mathfrak{S}_r, \ H_{n_1,\dots,n_r}^{\sigma_1,\dots,\sigma_r} =  H_{n_{\phi_1},\dots,n_{\phi_r}}^{\sigma_{\phi_1},\dots,\sigma_{\phi_r}} 
\end{equation}
and the bound
\begin{equation}
\label{eq_norm_ham}
\| H \|_{q,\alpha} = \! \! \! \sup_{\substack{\sigma \in \{-1,1\}^r\\ n\in \mathbf{Z}_d^r }}
 \left[ \mathop{\mathrm{hmean}}_{\nu \in (\{-1,1\}^d)^r} \Big( \big\langle \sum_{\ell =1}^r \nu_{\ell} \diamond  n_{\ell} \big\rangle^{\alpha} \Big) \right] \langle n_1 \rangle^q \dots  \langle n_r \rangle^q | H_{n}^{\sigma} | < \infty
\end{equation}
where $\nu \diamond  n  =(\nu_i n_i)_{j=1,\dots,d}$ and $\mathrm{hmean}$ denotes the harmonic mean (defined by \eqref{eq_def_hmean}).\\
The vector space of polynomials they generate is denoted by $\mathscr{H}_{q,\alpha}(\mathbf{Z}_d)$:
$$
\mathscr{H}_{q,\alpha}(\mathbf{Z}_d) = \bigoplus_{r \geq 3} \mathscr{H}_{q,\alpha}^{r}(\mathbf{Z}_d).
$$ 
\end{definition}

We choose the harmonic mean for convenience in \eqref{eq_norm_ham} (especially for Proposition \ref{prop_stable_Poisson}).

\begin{remark} If $r\geq 3$, $0\leq \alpha_1 \leq \alpha_2 $ and $0\leq q_1 \leq q_2 $ we have the continuous embedding
$$
\mathscr{H}_{q_2,\alpha_2}^{r}(\mathbf{Z}_d) \hookrightarrow \mathscr{H}_{q_1,\alpha_1}^{r}(\mathbf{Z}_d) \hookrightarrow \mathscr{H}_{0,0}^{r}(\mathbf{Z}_d) .
$$
\end{remark}

\begin{example} If $H\in \mathscr{H}_{q,\alpha}^{r}(\mathbf{Z}_d)$ satisfies the zero momentum condition
$$
\sigma_1 n_1+ \cdots+ \sigma_r n_r = 0 \quad \mathrm{or} \quad H_{n}^{\sigma}=0
$$
then $H\in \mathscr{H}_{q,\alpha'}^{r}(\mathbf{Z}_d)$ for all $\alpha'\geq 0$.
\end{example}
\noindent \emph{Hint.}
It is enough to note that, if $S$ is a finite set and $x\in (\mathbb{R}_+^*)^S$ then $\# S ( \mathop{\mathrm{hmean}}_{j\in S} x_j)^{-1} \geq x_{i}^{-1}$ for all $i\in S$.

\begin{example} Let $m\in \mathbb{N}$, $k\geq 2$ and $g\in C^m(\mathbb{T}^d;\mathbb{R})$. Identify each function of $L^2(\mathbb{T}^d)$ with the sequence of its Fourier coefficients, there exists $P\in \mathscr{H}_{0,m}^{2k}(\mathbb{Z}^d)$ such that
$$
\forall s>d/2,\forall u\in H^s(\mathbb{T}^d;\mathbb{C}), \quad P(u) = \int_{\mathbb{T}^d} g(x) |u(x)|^{2k} \mathrm{d}x.
$$
\end{example}
\noindent \emph{Hint.} It is enough to write this integral as a convolution of Fourier coefficient and to note that, since $g\in C^m(\mathbb{T}^d;\mathbb{R})$, the sequence $(\langle n\rangle^m \widehat{g}_n)_{n\in \mathbb{Z}^d}$ is bounded.

\begin{lemma}
\label{lem_val_ham}
 The $q$-smoothing  polynomials of $ \mathscr{H}_{q,0}(\mathbf{Z}_d)$ define naturally smooth real valued functions on $h^s(\mathbf{Z}_d)$ for $s > d/2-q$. More quantitatively, if $H\in \mathscr{H}_{q,0}^r(\mathbf{Z}_d)$ and $u^{(1)},\dots,u^{(r)} \in h^s(\mathbf{Z}_d)$, we have 
 $$
\sum_{\sigma \in \{-1,1\}^r} \sum_{n\in \mathbf{Z}_d^r } | H_{n}^{\sigma} u_{n_1}^{(1),\sigma_1} \dots u_{n_r}^{(r),\sigma_r} | \lesssim_{r,s,q,d}  \| H \|_{q,0}  \prod_{j=1}^r\|u^{(j)} \|_{h^s}.
 $$
 In other words, the multi-linear map defining $H$ is well defined and continuous on $h^s(\mathbf{Z}_d)$.
\end{lemma}
\begin{proof} 
Applying the Cauchy--Schwarz estimate, we get
\begin{multline*}
 \sum_{\sigma \in \{-1,1\}^r} \sum_{n\in \mathbf{Z}_d^r } | H_{n}^{\sigma} u_{n_1}^{(1),\sigma_1} \dots u_{n_r}^{(r),\sigma_r} | \leq \| H \|_{q,0}  \sum_{\sigma \in \{-1,1\}^r} \sum_{n\in \mathbf{Z}_d^r } \prod_{j=1}^r  \langle n_j \rangle^{-q} |u_{n_j}^{(j)} | 
 \\ = 2^r  \| H \|_{q,0}  \prod_{j=1}^r \sum_{k\in \mathbf{Z}_d } \langle k \rangle^{-q} |u_{k}^{(j)} |  \leq 2^r  \| H \|_{q,0} \big( \sum_{k\in \mathbb{Z}^d } \langle k \rangle^{-2(s+q)}  \big)^{\frac{r}2} \prod_{j=1}^r\|u^{(j)} \|_{h^s}.
\end{multline*}
 Using the reality condition \eqref{eq_def_real}, it is straightforward to check that $H$ is real valued.
\end{proof}

\begin{corollary} 
\label{cor_permute_ok}
We can permute derivatives with the sum defining $H$.
\end{corollary}
\begin{proof} It is a classical corollary of the continuity of the multi-linear maps associated with $H$.
\end{proof}

\begin{corollary}
\label{cor_uniq}
If there exists $s>d/2-q$ such that $H\in \mathscr{H}_{q,0}(\mathbf{Z}_d)$ vanishes everywhere on $h^s(\mathbf{Z}_d)$ then $H=0$ (i.e. all its coefficients vanish)
\end{corollary}
\begin{proof}
We denote $H = H^{(2)} + \dots + H^{(N)}$ the decomposition of $H$ in homogeneous polynomial. It follows from the symmetry condition \eqref{def_sym_cond} and Corollary \ref{cor_permute_ok} that we have
$$
H^{(r),\sigma}_n \approx_{r,n,\sigma} \partial_{u_{n_1}^{\sigma_1}} \cdots \partial_{u_{n_{r}}^{\sigma_{r}}} H(0) = 0.
$$
\end{proof}

The following lemma provides a multi-linear estimate which is the main technical result of this section.
\begin{lemma} \label{lem_tech_key} If $d\geq 1$, $r\geq 3$, $\alpha >\max(d-q,d/2)$, $q\geq 0$, $s\geq 0$ and $s \in (d/2-q,  \alpha-d/2+q)$ then for all $u^{(1)},\dots, u^{(r-1)} \in h^s(\mathbb{Z}^d)$ and all $u^{(r)} \in h^{-s}(\mathbb{Z}^d)$, we have
\begin{equation}
\label{eq_lem_tech_key}
\sum_{n\in (\mathbb{Z}^d)^r} \langle n_1 + \dots + n_r \rangle^{-\alpha} \prod_{j=1}^r \langle n_j \rangle^{-q} |u^{(j)}_{n_j}| \lesssim_{r,d,\alpha,q,s} \| u^{(r)} \|_{h^{-s}} \prod_{j=1}^{r-1}\| u^{(j)} \|_{h^{s}} 
\end{equation}
\end{lemma}
\begin{proof}Without loss of generality we assume that $u^{(j)}_k \geq 0$ for all $j\in \llbracket 1,r \rrbracket$ and $k\in \mathbb{Z}^d$. We denote $v_k := \langle k\rangle^{-s} u^{(r)}_k$, $k\in \mathbb{Z}^d$ and $n_{0} := -(n_1+\dots+n_r)$. Note that, as a consequence, we have $\|v\|_{\ell^2} = \| u^{(r)}\|_{h^{-s}}$ and the sum we aim at estimating writes as a convolution product.

In order to remove the factor $\langle n_r \rangle^{s-q}$, we apply Jensen's and Minkowski's inequalities to get
 $$
 \langle n_r \rangle^{s-q} =  \langle n_0 + \dots + n_{r-1} \rangle^{s-q} \leq \big( \langle n_0 \rangle+ \dots + \langle n_{r-1} \rangle \big) ^{(s-q)_+}  \leq  r^{(s-q-1)_+} \sum_{\ell =0}^{r-1} \langle n_{\ell} \rangle^{(s-q)+} 
 $$
 where $x_+ := \max(0,x)$ denotes the positive part. Therefore, denoting by $\star$ the usual convolution product on $\mathbb{Z}^d$, we have
 \begin{multline*}
 \sum_{n\in (\mathbb{Z}^d)^r} \langle n_1 + \dots + n_r \rangle^{-\alpha} \prod_{j=1}^r \langle n_j \rangle^{-q} |u^{(j)}_{n_j}| =   \big[\langle \, \cdot \, \rangle^{-\alpha} \star \big( \mathop{\bigstar}_{j=1}^{r-1} \langle \, \cdot  \, \rangle^{-q} u^{(j)} \big) \star (\langle \, \cdot \, \rangle^{s-q} v )\big]_0 \\ \lesssim_{r,s,q} [v\star \langle \, \cdot \, \rangle^{(s-q)_+-\alpha} \star \big( \mathop{\bigstar}_{j=1}^{r-1} \langle \, \cdot  \, \rangle^{-q} u^{(j)} \big) ]_0 +\big[ \sum_{\ell=1}^{r-1} v\star (\langle \, \cdot\,\rangle^{(s-q)_+-q} u^{(\ell)}) \star \langle \, \cdot \, \rangle^{-\alpha} \star \big( \mathop{\bigstar}_{\substack{j=1 \\ j\neq \ell}}^{r-1} \langle \, \cdot  \, \rangle^{-q} u^{(j)} \big) \big]_0
 \end{multline*}
Consequently, the estimate \eqref{eq_lem_tech_key} is just a consequence of the Young's convolution inequality $\ell^2 \star \ell^2 \star \ell^1  \star \dots \star \ell^1 \hookrightarrow \ell^\infty$ and the following estimates

$\bullet$ since $s > d/2-q$, applying the Cauchy-Schwarz inequality we have $$\|  \langle \, \cdot \, \rangle^{-q} u \|_{\ell^1} \leq \| \langle \, \cdot \, \rangle^{-(s+q)}   \|_{\ell^2} \| u^{(j)} \|_{h^s}\lesssim_{s,d,q}\| u^{(j)} \|_{h^s},$$

$\bullet$ since $s<\alpha -d/2+q$ and $\alpha > d/2$, we have $\alpha-(s-q)_+> d/2$ and we have 
$$\| \langle \, \cdot \, \rangle^{(s-q)_+-\alpha} \|_{\ell^2} \lesssim_{\alpha,s,q,d} 1,$$   

$\bullet$ since $q \geq 0$ and $\alpha> \max(d/2,d-q)$ we have
$$
\| (\langle \, \cdot\,\rangle^{(s-q)_+-q} u^{(\ell)}) \star \langle \, \cdot \, \rangle^{-\alpha} \|_{\ell^2} \lesssim_{\alpha,s,d,q}\| u^{(\ell)} \|_{h^s}.
$$
Let us prove this last estimate which is much less obvious than the previous ones. We have to distinguish three cases.

$\bullet$  If $q=0$ then $\alpha>d$ and it is enough to apply the Young's convolution inequality $\ell^2 \star \ell^1  \hookrightarrow \ell^2$.

$\bullet$  If $q\geq d/2$ then we control $\langle \, \cdot\,\rangle^{(s-q)_+-q} u^{(\ell)}$ in $\ell^1$ and so since $\alpha>d/2$ and it is enough to apply the Young's convolution inequality $\ell^1 \star \ell^2  \hookrightarrow \ell^2$.

$\bullet$ Else we have $0<q<d/2$ and $\alpha > d-q$. We set $p = \frac{2d}{2\rho+d}$ and $b = \frac{d}{d -\rho}$ where $\rho \in (0,q)$ is a number close enough to $q$ to have 
$$
b \alpha =  d \frac{\alpha}{d-\rho} = d \frac{\alpha}{d-q} \frac{d-q}{d-\rho}  > d.
$$
 We note that by construction $b,p\in (1,2)$ and
$$
\frac1p + \frac1b =\frac{2\rho+d}{2d} +  \frac{d -\rho}{d} =1+\frac12.
$$
Therefore, recalling that $b\alpha >d$, $s\geq 0$ and $q\geq 0$, we apply the Young's convolution inequality $\ell^p \star \ell^b  \hookrightarrow \ell^2$ to get
$$
\| (\langle \, \cdot\,\rangle^{(s-q)_+-q} u^{(\ell)}) \star \langle \, \cdot \, \rangle^{-\alpha} \|_{\ell^2} \lesssim_{d,q} \| \langle \, \cdot\,\rangle^{(s-q)_+-q} u^{(\ell)}\|_{\ell^p} \|\langle \, \cdot \, \rangle^{-\alpha} \|_{\ell^b} \lesssim_{d,q,\alpha,\rho} \| \langle \, \cdot\,\rangle^{s-q} u^{(\ell)}\|_{\ell^p}.
$$
Finally applying the H\"older inequality we get
\begin{multline*}
\| (\langle \, \cdot\,\rangle^{(s-q)_+-q} u^{(\ell)}) \star \langle \, \cdot \, \rangle^{-\alpha} \|_{\ell^2} \lesssim_{d,q,\alpha,\rho} \| \langle \, \cdot\,\rangle^{s-q} u^{(\ell)}\|_{\ell^p} \approx_{d,q,\alpha,\rho}  \| \langle \, \cdot\,\rangle^{-pq} (\langle \, \cdot\,\rangle^{s} u^{(\ell)})^p\|_{\ell^1}^{1/p} \\
\lesssim_{d,q,\alpha,\rho} \| u^{(\ell)} \|_{h^s} \|  \langle \, \cdot\,\rangle^{-q \frac{2p}{2-p}} \|_{\ell^1}^{\frac{2-p}{2p}}
\end{multline*}
which conclude the proof since
$$
q \frac{2p}{2-p} = q \frac{2\frac{2d}{2\rho+d}}{2-\frac{2d}{2\rho+d}} = d \frac{q}{\rho} > d.  
$$

\end{proof}

First we deduce that the vector field generated by $\alpha$-inhomogeneous, $q$-smoothing Hamiltonians maps $h^s$ into itself.
\begin{proposition}
\label{prop_vf}
 If $s\in (d/2-q,  \alpha-d/2+q)$, $s\geq 0$ and $\alpha >\max(d-q,d/2)$ then the gradients of $\alpha$-inhomogeneous, $q$-smoothing Hamiltonians of $ \mathscr{H}_{q,\alpha}(\mathbf{Z}_d)$ are smooth functions from $h^s(\mathbf{Z}_d)$ into $h^s(\mathbf{Z}_d)$. More quantitatively, if $H\in \mathscr{H}_{q,\alpha}^r(\mathbf{Z}_d)$, $r\geq 3$, we have
\begin{equation}
\label{eq_cont_X}
\forall u \in h^s(\mathbf{Z}_d), \quad \|\nabla H(u)\|_{h^s} \lesssim_{s,r,d,q,\alpha}  \| H \|_{q,\alpha}  \|u\|_{h^s}^{r-1}
\end{equation}
and
\begin{equation}
\label{eq_cont_dX}
\forall u \in h^s(\mathbf{Z}_d), \quad \| \mathrm{d} \nabla H(u)\|_{\mathscr{L}(h^s)} \lesssim_{s,r,d,q,\alpha}  \| H \|_{q,\alpha}  \|u\|_{h^s}^{r-2}
\end{equation}
\end{proposition}
\begin{proof} We recall that as a consequence of Lemma \ref{lem_val_ham}, $H$ defines a smooth function on $h^{s}$. As a consequence, its gradient is well defined and belongs to $h^{-s}$. Let us check that actually it is a smooth function taking values into $h^s$.

As a consequence of Corollary \ref{cor_permute_ok} (which ensures that sums and derivatives can be permuted) and the symmetry condition \eqref{def_sym_cond}, for $k\in \mathbf{Z}_d$ and $u \in h^s(\mathbf{Z}_d)$ we have
\begin{equation}
\label{eq_dev_nabla}
(\nabla H(u))_k = 2\partial_{\bar{u}_k} H(u) = 2\,r  \sum_{\sigma \in \{-1,1\}^{r-1}} \sum_{n\in \mathbf{Z}_d^{r-1} }  H_{n,k}^{\sigma,-1} u_{n_1}^{\sigma_1} \dots u_{n_{r-1}}^{\sigma_{r-1}}.
\end{equation}
Then, we recall that by definition we have
$$
|H_{n,k}^{\sigma,-1}| \leq 2^{-dr} \| H\|_{q,\alpha} \langle k \rangle^{-q} \langle n_1 \rangle^{-q} \cdots   \langle n_{r-1} \rangle^{-q}  \sum_{\nu \in (\{-1,1\}^d)^r}  \big\langle \sum_{\ell =1}^{r-1} \nu_{\ell} \diamond  n_{\ell} + \nu_{r} \diamond  k \big\rangle^{-\alpha}.
$$
Therefore, since both $u\mapsto u_{\nu_j \diamond \cdot}$ and $u\mapsto \bar u$ are isometries on the Sobolev spaces, as a corollary of Lemma \ref{lem_tech_key}, being given $u^{(1)},\dots,u^{(r-1)}\in h^s(\mathbf{Z}_d)$, the multi-linear forms
$$
\phi(u^{(1)},\dots,u^{(r-1)}) : \left\{ \begin{array}{clll} h^{-s}(\mathbf{Z}_d) &\to& \mathbb{R} \\
v &\mapsto& 2\,r  \displaystyle \sum_{k\in \mathbf{Z}_d} \sum_{\sigma \in \{-1,1\}^{r-1}} \sum_{n\in \mathbf{Z}_d^{r-1} } \Re \left[ H_{n,k}^{\sigma,-1} u_{n_1}^{(1),\sigma_1} \dots u_{n_{r-1}}^{(r-1),\sigma_{r-1}} \overline{v_k} \right]. \end{array} \right. 
$$
are well defined, smooth and we have the bound
\begin{equation}
\label{eq_mult_lin_phi}
\|\phi(u^{(1)},\dots,u^{(r-1)})\|_{(h^{(-s)})'} \lesssim_{s,r,q,\alpha,d} \| H\|_{q,\alpha}  \| u^{(1)} \|_{h^s} \dots \| u^{(r-1)} \|_{h^s}
\end{equation}
where $(h^{(-s)})'$ denotes the dual of $h^{(-s)}$. We denote $\mathcal{I}_s :(h^{(-s)})'\to h^s $ the usual isometry associated with the $\ell^2$ scalar product. 
Therefore, it follows of \eqref{eq_dev_nabla} that for all $v\in h^\infty(\mathbf{Z}_d)$, we have
$$
(\nabla H(u),v)_{\ell^2} = \phi(u,\dots,u)(v) = (\mathcal{I}_s \phi(u,\dots,u),v)_{\ell^2}.
$$
Consequently, we have $\nabla H(u)=\mathcal{I}_s \phi(u,\dots,u)$ which ensures that $\nabla H(u) \in h^s$. Finally, the continuity estimate \eqref{eq_mult_lin_phi} proves that $u\mapsto \nabla H(u)\in h^s$ is smooth and satisfies the bounds \eqref{eq_cont_X} and \eqref{eq_cont_dX}.

\end{proof}

As a corollary, we extend the differential of $\nabla H$ into some negative spaces. This technical corollary is crucial to prove that the change of variable associated with the normal form preserves the time differentiability. Its proof relies on the symmetry of $\mathrm{d}\nabla H(u)$ and duality arguments.
\begin{corollary} \label{cor_extend} If $s\in (d/2-q,  \alpha-d/2+q)$, $s\geq 0$, $\alpha >\max(d-q,d/2)$ and $ H\in \mathscr{H}_{q,\alpha}^r(\mathbf{Z}_d)$ for $r\geq 3$ then for all $u\in h^s(\mathbf{Z}_d)$, $\mathrm{d}\nabla H(u)$ admits an unique continuous extension from $h^{-s}(\mathbf{Z}_d)$ into $h^{-s}(\mathbf{Z}_d)$. Furthermore, the map $ h^s(\mathbf{Z}_d) \ni u \mapsto \mathrm{d}\nabla H(u) \in \mathscr{L}(h^{-s}(\mathbf{Z}_d))$ is smooth and we have the bound
$$
\forall u \in h^s(\mathbf{Z}_d), \ \| \mathrm{d}\nabla H(u) \|_{\mathscr{L}(h^{-s})} \lesssim_{s,r,d,q,\alpha}  \| H \|_{q,\alpha}  \|u\|_{h^s}^{r-2}.
$$
\end{corollary} 
\begin{proof}
Since the embedding of $h^s$ into $h^{-s}$ is continuous and $h^s$ is dense in $h^{-s}$, applying the continuous extension Theorem, we just have to prove that
\begin{equation}
\label{eq_chocolat_chaud}
\forall u \in h^s(\mathbf{Z}_d), \ \sup_{\substack{v\in h^s \\ \| v\|_{h^{-s}}\leq 1}} \| \mathrm{d}\nabla H(u)(v) \|_{h^{-s}} \lesssim_{s,r,d,q,\alpha}  \| H \|_{q,\alpha}  \|u\|_{h^s}^{r-2}.
\end{equation}
First, by duality between $h^s$ and $h^{-s}$, we note that
$$
\sup_{\substack{v\in h^s \\ \| v\|_{h^{-s}}\leq 1}}   \| \mathrm{d}\nabla H(u)(v) \|_{h^{-s}} =  \sup_{\substack{v\in h^s \\ \| v\|_{h^{-s}}\leq 1}} \sup_{\substack{w\in h^s \\ \| w\|_{h^{s}}\leq 1}}   (w, \mathrm{d} \nabla H(u)(v) )_{\ell^2}  .
$$
Then by applying the Schwarz Theorem we have
\begin{multline*}
(w, \mathrm{d} \nabla H(u)(v) )_{\ell^2}   = \mathrm{d}[ (w,\nabla H(u) )_{\ell^2} ](v)   = \mathrm{d}[ \mathrm{d} H(u)(w)  ](v) =  \mathrm{d}^2 H(u)(w)(v)  \\= \mathrm{d}^2 H(u)(v)(w) = \mathrm{d}[ (v, \nabla H(u) )_{\ell^2} ](w)  = (v, \mathrm{d} \nabla H(u)(w) )_{\ell^2}.
\end{multline*}
Therefore, applying once again the duality between $h^s$ and $h^{-s}$, we deduce that
\begin{multline*}
\sup_{\substack{v\in h^s \\ \| v\|_{h^{-s}}\leq 1}}   \| \mathrm{d} \nabla H(u)(v) \|_{h^{-s}} = \! \! \sup_{\substack{w\in h^s \\ \| w\|_{h^{s}}\leq 1}}   \sup_{\substack{v\in h^s \\ \| v\|_{h^{-s}}\leq 1}}   (v, \mathrm{d} \nabla H(u)(w) )_{\ell^2}  = \! \!  \sup_{\substack{w\in h^s \\ \| w\|_{h^{s}}\leq 1}}     \| \mathrm{d} \nabla H(u)(w) \|_{h^{s}} 
\\= \|\mathrm{d} \nabla H(u)\|_{\mathscr{L}(h^s)}.
\end{multline*}
As a consequence, \eqref{eq_chocolat_chaud} is just a corollary of the estimate \eqref{eq_cont_dX} of Proposition \ref{prop_vf}. Finally, we note that, the smoothness of $u\in h^s(\mathbf{Z}_d) \mapsto \mathrm{d}\nabla H(u) \in \mathscr{L}(h^{-s}(\mathbf{Z}_d))$ is just a natural corollary of the homogeneity of $H$ (i.e. $\mathrm{d}\nabla H(u)$ could be replaced by a multi-linear map as we did in the previous proof).
\end{proof}

\begin{proposition}
\label{prop_ham_flow} If $s\in (d/2-q,  \alpha-d/2+q)$, $s\geq 0$, $\alpha >\max(d-q,d/2)$ and $\chi \in \mathscr{H}_{q,\alpha}^r(\mathbf{Z}_d)$ for $r\geq 3$, then there exist $ \varepsilon_0 \gtrsim_{s,d,\alpha,q,r} \|\chi \|_{q,\alpha}^{-1/(r-2)}$ and a smooth map
$$
\Phi_\chi : \left\{ \begin{array}{cll} [-1,1] \times B_{h^s(\mathbf{Z}_d)}(0,\varepsilon_0) &\to& h^s(\mathbf{Z}_d) \\ (t,u) &\mapsto& \Phi_\chi^t(u) \end{array} \right.
$$
solving the equation
$$
-i\partial_t \Phi_\chi = ( \nabla \chi)\circ \Phi_\chi,
$$
and such that for all $t\in [-1,1]$, $\Phi_\chi^t$ is symplectic, close to the identity
\begin{equation}
\label{eq_dindon}
\forall u \in  B_{h^s(\mathbf{Z}_d)}(0,\varepsilon_0), \quad \| \Phi_\chi^t u - u \|_{h^s} \lesssim_{s,d,\alpha,q,r} \|\chi \|_{q,\alpha}  \| u \|_{h^s}^{r-1},
\end{equation}
invertible
\begin{equation}
\label{eq_invertibility}
\|\Phi_\chi^{t} (u) \|_{h^s} < \varepsilon_0 \quad \Rightarrow \quad  \Phi_\chi^{-t}\circ  \Phi_\chi^{t} (u) = u
\end{equation}
and its differential admits an unique continuous extension from $h^{-s}(\mathbf{Z}_d)$ into $h^{-s}(\mathbf{Z}_d)$. Moreover, the map $u\in  B_{h^s(\mathbf{Z}_d)}(0,\varepsilon_0) \mapsto \mathrm{d} \Phi_\chi^t(u) \in \mathscr{L}(h^{-s}(\mathbf{Z}_d))$ is continuous and we have the estimates
\begin{equation}
\label{eq_petanque}
\forall u \in  B_{h^s(\mathbf{Z}_d)}(0,\varepsilon_0),\forall \sigma\in \{-1,1\},\ \| \mathrm{d} \Phi_\chi^t (u)\|_{\mathscr{L}(h^{\sigma s})}\leq 2.
\end{equation}
\end{proposition}
\begin{proof}  Since the vector field $i \nabla \chi$ is smooth on $h^s(\mathbf{Z}_d)$ (see Proposition \ref{prop_vf}), the Cauchy-Lipschitz Theorem proves that the flow it generates, denoted $\Phi_\chi^{t} (u)$, is locally well defined and is smooth. Let denote $I_u$ the maximal interval on which $\Phi_\chi^{t} (u)$ is well defined. Let us prove that if $u$ is small enough then $[-1,1] \subset I_u$. More precisely, we are going to prove that if $t\in [-1,1]\cap I_u$ then $\| \Phi_\chi^{t} (u) \|_{h^s}\leq 2 \| u \|_{h^s}$.

By definition of $\Phi_\chi$, if $t\in I_u$, we have
$$
\Phi_\chi^{t} (u) = u + i \int_0^t ( \nabla \chi )\circ \Phi_\chi^\tau \circ u \ \mathrm{d}\tau.
$$
As a consequence, applying the estimate \eqref{eq_cont_X} of  Proposition \ref{prop_vf}, if $|t|\leq 1$, it satisfies
$$
\|\Phi_\chi^{t} (u) - u\|_{h^s} \leq \sup_{\tau \in (0,t)} \| ( \nabla \chi )\circ \Phi_\chi^\tau \circ u  \|_{h^s} \leq C_{s,d,\alpha,q,r}  \|\chi \|_{q,\alpha} \sup_{\tau \in (0,t)} \| \Phi_\chi^\tau \circ u\|_{h^s}^{r-1} 
$$
where $ C_{s,d,\alpha,q,r} >0$ denotes the maximum of the implicit constants in the estimates \eqref{eq_cont_X} and \eqref{eq_cont_dX} of Proposition \ref{prop_vf}. Let $J_u\subset I_u$ be the maximal interval (with $0\in J_u$) such that for all $t\in J_u$,  $\| \Phi_\chi^{t} (u) \|_{h^s}\leq 3 \| u \|_{h^s}$.
It follows that if $t\in J_u \cap [-1,1]$ then
$$
\|\Phi_\chi^{t} (u) - u\|_{h^s} \leq 3^{r-1} C_{s,d,\alpha,q,r}  \|\chi \|_{q,\alpha} \| u \|_{h^s}^{r-1}  \leq \| u \|_{h^s}
$$
provided that $3^{r-1} C_{s,d,\alpha,q,r}  \|\chi \|_{q,\alpha} \| u \|_{h^s}^{r-2}\leq 1$. Therefore, by a standard bootstrap argument, we deduce that provided that $\| u \|_{h_s}\leq \varepsilon_0 := (3^{r-1} C_{s,d,\alpha,q,r}  \|\chi \|_{q,\alpha})^{-1/r-2} $, $\Phi_\chi^t(u)$ is well defined for $t\in [-1,1]$, $\| \Phi_\chi^{t} (u) \|_{h^s}\leq 2 \| u \|_{h^s}$ and is close to the identity (i.e. \eqref{eq_dindon} holds). 
The invertibility property \eqref{eq_invertibility} is a classical corollary of the uniqueness provided by the Cauchy--Lipschitz Theorem.

Since $\Phi_\chi$ is smooth, $\mathrm{d}\Phi_\chi^t$ is a solution of the linear equation
\begin{equation}
\label{eq_lin}
 \partial_t \mathrm{d}\Phi_\chi^t(u) = i \mathrm{d}\nabla \chi(\Phi_\chi^t(u))\circ \mathrm{d}\Phi_\chi^t(u) \quad \mathrm{and} \quad  \mathrm{d}\Phi_\chi^0(u) = \mathrm{id}_{h^s}.
\end{equation}
Therefore it is classical to check that $\Phi_\chi^t$ is symplectic (it is a consequence of the Schwarz Theorem). If $t\in [-1,1]$, applying the estimate  \eqref{eq_cont_dX} of Proposition \ref{prop_vf}, we have
\begin{multline*}
\| \mathrm{d}\Phi_\chi^t(u) \|_{\mathscr{L}(h^s)} \leq 1+ \big| \int_0^t \| \mathrm{d}\nabla \chi(\Phi_\chi^\tau(u)) \|_{\mathscr{L}(h^s)}  \| \mathrm{d}\Phi_\chi^\tau(u) \|_{\mathscr{L}(h^s)} \mathrm{d}\tau  \big| \\ \leq 1+ C_{s,d,\alpha,q,r}  2^{r-2} \|\chi \|_{q,\alpha}  \|u\|_{h^s}^{r-2}\big| \int_0^t  \| \mathrm{d}\Phi_\chi^\tau(u) \|_{\mathscr{L}(h^s)} \mathrm{d}\tau \big| .
\end{multline*}
Consequently, by definition of $\varepsilon_0$, we have
$$
\| \mathrm{d}\Phi_\chi^t(u) \|_{\mathscr{L}(h^s)} \leq 1+\big| \frac13 \int_0^t  \| \mathrm{d}\Phi_\chi^\tau(u) \|_{\mathscr{L}(h^s)} \mathrm{d}\tau \big|.
$$
Applying the Gr\"onwall's inequality, we deduce that
$\| \mathrm{d}\Phi_\chi^t(u) \|_{\mathscr{L}(h^s)} \leq e^\frac13 \leq 2.$
For clarity, we denote by $E(u):h^{-s}\to h^{-s}$ the extension of $\mathrm{d}\nabla \chi(u)$ provided by Corollary \ref{cor_extend}. Recalling that $u\mapsto E(u)$ is continuous and bounded on bounded set of $h^s(\mathbf{Z}_d)$, applying the Cauchy--Lipschitz Theorem the following non-autonomous linear equation 
\begin{equation*}
 \partial_t A_u(t) = iE(\Phi_\chi^t(u))\circ A_u(t) \quad \mathrm{and} \quad  A_u(0) = \mathrm{id}_{h^{-s}}
\end{equation*}
admits a unique solution for $t\in [-1,1]$. Moreover, the map $u\in h^s \mapsto A_u(t) \in \mathscr{L}(h^{-s})$ is smooth. Finally, we just have to check that $A_u(t)$ is an extension of $\mathrm{d}\Phi_\chi^t(u)$. Indeed, recalling that $E$ is an extension of $\nabla \chi$, if $v\in h^s(\mathbf{Z}_d)$ both $A_u(t)(v)$ and $\mathrm{d}\Phi_\chi^t(u)(v)$ are solutions of the linear non-autonomous equation
$$
 \partial_t w(t) = i E(\Phi_\chi^t(u))(w(t)) \quad \mathrm{and} \quad w(0) = v.
$$
Therefore, as a consequence of the uniqueness in the the Cauchy--Lipschitz Theorem, $A_u(t)(v) =\mathrm{d}\Phi_\chi^t(u)(v)$. The estimate \eqref{eq_petanque} when $\sigma=-1$ can be obtained as we did when $\sigma=1$.
\end{proof}

Now, we prove that the class of Hamiltonians is stable by Poisson bracket.
\begin{proposition} 
\label{prop_stable_Poisson}
If $\alpha>\max(d-q,d/2)$ and $q\geq 0$ then for all Hamiltonians $H,K \in \mathscr{H}_{q,\alpha}(\mathbf{Z}_d)$ there exists a Hamiltonian $N \in \mathscr{H}_{q,\alpha}(\mathbf{Z}_d)$ such that for all $s\in (d/2-q,  \alpha-d/2+q)$ with $s\geq 0$, we have
$$
\forall u \in h^s(\mathbf{Z}_d), \ N(u) = \{ H,K\}(u) 
$$ 
Moreover, if $r,r'\geq 2$, for $H\in \mathscr{H}_{q,\alpha}^r$ and $K\in \mathscr{H}_{q,\alpha}^{r'}$, $N \in \mathscr{H}_{q,\alpha}^{r+r'-2}$ is of order $r+r'-2$ and satisfies the continuity estimate
 \begin{equation}
 \label{eq_cont_est_Poisson}
 \| N\|_{q,\alpha} \lesssim_{\alpha,d} r r'   \| H \|_{q,\alpha}  \| K \|_{q,\alpha}. 
 \end{equation}
\end{proposition}
\begin{proof} First we note that since $\alpha>\max(d-q,d/2)$ and $q\geq 0$, we have $d/2-q <  \alpha-d/2+q$ and $\alpha-d/2+q \geq 0$. Consequently, there exists some $s$ satisfying the assumptions of this proposition. 

The uniqueness follows from Corollary \ref{cor_uniq}. For the existence and the estimate, by linearity, it is enough to consider homogeneous Hamiltonians, $H\in \mathscr{H}_{q,\alpha}^r(\mathbf{Z}_d)$ and $K\in \mathscr{H}_{q,\alpha}^{r'}(\mathbf{Z}_d)$. Let $u\in h^s(\mathbf{Z}_d)$ for some $s\in (d/2-q,  \alpha-d/2+q)$ with $s\geq 0$. We recall that we have (see \eqref{eq_poisson_coord})
 \begin{equation}
 \label{eq_un_lapin}
 \{ H,K\}(u)  = 
  2i \sum_{k\in \mathbf{Z}_d} \partial_{\bar{u}_k}H(u) \partial_{u_k} K(u) - \partial_{u_k}H(u) \partial_{\bar{u}_k} K(u).
 \end{equation}

 Since the coefficients of $H$ and $K$ are symmetric (i.e. satisfy \eqref{def_sym_cond}), we have
\begin{equation}
\label{eq_a_tue_un_chasseur}
\partial_{\bar{u}_k} H \partial_{u_k} K = r r'  \sum_{\substack{\sigma \in \{-1,1\}^{r-1}\\ \sigma' \in \{-1,1\}^{r'-1} }} \sum_{\substack{n\in \mathbf{Z}_d^{r-1} \\ n'\in \mathbf{Z}_d^{r'-1}  }}    H_{n,k}^{\sigma,-1} u_{n_1}^{\sigma_1} \dots u_{n_{r-1}}^{\sigma_{r-1}}  K_{n',k}^{\sigma',1} u_{n'_1}^{\sigma'_1} \dots u_{n'_{r'-1}}^{\sigma'_{r'-1}}.
\end{equation}
As a consequence, naturally, we study the convergence of the series $\sum_k H_{n,k}^{\sigma,-1} K_{n',k}^{\sigma',1}$. By definition of the $\| \cdot \|_{q,\alpha}$ norm, denoting $n''=(n,n')$ and $r''=r+r'-2$, we have
\begin{multline*}
\sum_{k\in \mathbf{Z}_d}|H_{n,k}^{\sigma,-1} K_{n',k}^{\sigma',1}| \leq 2^{-d(r+r')} \| H \|_{q,\alpha} \|K \|_{q,\alpha}\prod_{j=1}^{r''}\langle n''_j \rangle^{-q} \\
 \times \Big(\sum_{k\in \mathbf{Z}_d}   \langle k \rangle^{-2q} \sum_{\substack{\nu \in (\{-1,1\}^d)^r \\ \nu' \in (\{-1,1\}^d)^{r'} } }   \big\langle \nu_{r} \diamond k + \sum_{\ell =1}^{r-1} \nu_{\ell} \diamond  n_{\ell} \big\rangle^{-\alpha}   \big\langle \nu'_{r'} \diamond k + \sum_{\ell =1}^{r'-1} \nu'_{\ell} \diamond  n'_{\ell} \big\rangle^{-\alpha} \Big)
\end{multline*}
and so
\begin{multline*}
\sum_{k\in \mathbf{Z}_d}|H_{n,k}^{\sigma,-1} K_{n',k}^{\sigma',1}| \leq 2^{-d(r+r'-2)} \| H \|_{q,\alpha} \|K \|_{q,\alpha}\prod_{j=1}^{r''} \langle n''_j \rangle^{-q} \\
 \times \Big(  \sum_{\substack{\nu \in (\{-1,1\}^d)^{r-1} \\ \nu' \in (\{-1,1\}^d)^{r'-1} } } \sum_{k\in \mathbb{Z}^d} \langle k \rangle^{-2q} \big\langle  k - \sum_{\ell =1}^{r-1} \nu_{\ell} \diamond  n_{\ell} \big\rangle^{-\alpha}   \big\langle  k + \sum_{\ell =1}^{r'-1} \nu'_{\ell} \diamond  n'_{\ell} \big\rangle^{-\alpha} \Big).
\end{multline*}
Therefore, since $\alpha+2q>d$ (because $\alpha>\max(d-q,d/2)$), applying Lemma \ref{lem_conv1} of the appendix\footnote{here we benefit of the choice of the harmonic mean in Definition \ref{def_class_ham}.}, we deduce that
\begin{equation}
\label{eq_lugaru}
\sum_{k\in \mathbf{Z}_d}|H_{n,k}^{\sigma,1} K_{n',k}^{\sigma',-1}| \lesssim_{\alpha,d} 2^{-dr''}\| H \|_{q,\alpha} \|K \|_{q,\alpha}\prod_{j=1}^{r''} \langle n''_j \rangle^{-q} 
 \Big( \!  \!   \!   \!   \!   \! \! \!   \!   \! \sum_{\nu'' \in (\{-1,1\}^d)^{r+r'-2}  }  \!   \!  \big\langle \sum_{\ell =1}^{r''} \nu''_{\ell} \diamond  n''_{\ell} \big\rangle^{-\alpha}    \Big) .
\end{equation}
Since this series converges, we define
$$
M_{n''}^{\sigma''} := 2 i r r' \sum_{k\in \mathbf{Z}_d} H_{n,k}^{\sigma,-1} K_{n',k}^{\sigma',1}  - H_{n,k}^{\sigma,1} K_{n',k}^{\sigma',-1} \quad \mathrm{and} \quad N_{n''}^{\sigma''} = \frac1{r'' !} \sum_{\rho \in \mathfrak{S}_{r''}} M_{n''\circ \rho}^{\sigma''\circ \rho}.
$$
Let us check that $N\in \mathscr{H}_{q,\alpha}^{r''}(\mathbf{Z}_d)$. By definition, the coefficients of $N$ are obviously symmetric. Furthermore, the estimates \eqref{eq_lugaru} proves that
$$
\| N\|_{q,\alpha} \lesssim_{\alpha,d} r r' \| H \|_{q,\alpha} \|K \|_{q,\alpha}.
$$ 
 and the reality condition follows of the following calculation
$$
\frac{\overline{M_{n''}^{-\sigma''} }}{2rr'} =  -i  \sum_{k\in \mathbf{Z}_d}  \overline{H_{n,k}^{-\sigma,1} K_{n',k}^{-\sigma',-1}  - H_{n,k}^{-\sigma,-1} K_{n',k}^{-\sigma',1}} 
=  -i  \sum_{k\in \mathbf{Z}_d} H_{n,k}^{\sigma,-1} K_{n',k}^{\sigma',1}  - H_{n,k}^{\sigma,1} K_{n',k}^{\sigma',-1} = \frac{M_{n''}^{\sigma''}}{2rr'} .
$$
Finally, we just have to check that $N(u) = \{ H,K\}(u) $. Note that the following formal computation are justified rigorously by the Fubini Theorem, the estimate \eqref{eq_lugaru} and the  multi-linear estimate of Lemma \ref{lem_val_ham}. Indeed, by \eqref{eq_un_lapin} and \eqref{eq_a_tue_un_chasseur} we have
\begin{equation*}
\begin{split}
 \{ H,K\}(u)  &= 2irr' \sum_{k\in \mathbf{Z}_d} \! \sum_{\substack{\sigma \in \{-1,1\}^{r-1}\\ \sigma' \in \{-1,1\}^{r'-1} }} \! \sum_{\substack{n\in \mathbf{Z}_d^{r-1} \\ n'\in \mathbf{Z}_d^{r'-1}  }}  (H_{n,k}^{\sigma,-1}  K_{n',k}^{\sigma',1} -H_{n,k}^{\sigma,1}  K_{n',k}^{\sigma',-1} ) \prod_{j=1}^{r''} u_{(n,n')_j}^{(\sigma,\sigma')_j} \\
 &= \sum_{\sigma'' \in \{-1,1\}^{r''}} \sum_{n\in \mathbf{Z}_d^{r''} }   M_{n''}^{\sigma''}\prod_{j=1}^{r''} u_{n''_j}^{\sigma''_j} = N(u).
  \end{split}
\end{equation*}

\end{proof}

The last lemma of this section concerns the computation of the Poisson bracket between a quadratic integrable Hamiltonian and a $\alpha$-inhomogeneous, $q$-smoothing Hamiltonian.
\begin{lemma} 
\label{lem_Z2_vs_H}
If $s\in (d/2-q,  \alpha-d/2+q)$, $s\geq 0$, $\alpha >\max(d-q,d/2)$ and $ H\in \mathscr{H}_{q,\alpha}^r(\mathbf{Z}_d)$ where $r\geq 3$, and
 $Z_2 : h^s(\mathbf{Z}_d) \to \mathbb{R}$ is a quadratic Hamiltonian of the form
$$
Z_2(u) = \sum_{n\in \mathbf{Z}_d} \omega_n |u_n|^2
$$
where $\omega_n \in \mathbb{R}$ is such that $( \langle n \rangle^{-2s} \omega_n)_{n\in \mathbf{Z}_d}$ is bounded, then for all $u\in h^s(\mathbf{Z}_d)$ we have
\begin{equation}
\label{formula_H_Z2}
\{ H , Z_2\}(u) = -2 \, i  \sum_{\sigma \in \{-1,1\}^{r}} \sum_{n\in \mathbf{Z}_d^r } (\sigma_1 \omega_{n_1} + \dots + \sigma_r \omega_{n_r})  H_{n}^{\sigma} u_{n_1}^{\sigma_1} \dots  u_{n_r}^{\sigma_r}.
\end{equation}
\end{lemma}
\begin{proof} First, note that since, by Proposition \ref{prop_vf}, $\nabla H(u) \in h^s(\mathbf{Z}_d)$ and by assumption $\nabla Z_2(u) \in h^{-s}(\mathbf{Z}_d)$, it makes sense to consider $\{ H , Z_2\}$. Moreover, the convergence of the series in the right hand side of \eqref{formula_H_Z2} is ensured by Lemma \ref{lem_tech_key}.
We recall that, as usual, we have (see \eqref{eq_poisson_coord})
$$
\{ H , Z_2\}(u) = 2i \sum_{k\in \mathbf{Z}_d} \partial_{\bar{u}_k}H(u) \partial_{u_k} Z_2(u) - \partial_{u_k}H(u) \partial_{\bar{u}_k} Z_2(u).
$$
Therefore, computing $\partial_{\bar{u}_k}H(u)$ thanks to Corollary \ref{cor_permute_ok}, we have
$$
\{ H , Z_2\}(u) = -2 \, i \,r \sum_{k\in \mathbf{Z}_d} \sum_{\sigma \in \{-1,1\}^{r}} \sum_{n\in \mathbf{Z}_d^{r-1} }  H_{n,k}^{\sigma} u_{n_1}^{\sigma_1} \dots u_{n_{r-1}}^{\sigma_{r-1}} \sigma_r \omega_k u_k^{\sigma_r}.
$$
Since by Lemma \ref{lem_tech_key}, this series is absolutely convergent, applying Fubini's Theorem, we deduce 
$$
\{ H , Z_2\}(u) = -2 \, i\,r  \sum_{\sigma \in \{-1,1\}^{r}} \sum_{n\in \mathbf{Z}_d^r }  H_{n}^{\sigma} u_{n_1}^{\sigma_1} \dots u_{n_{r-1}}^{\sigma_{r-1}} \sigma_r \omega_{n_r} u_{n_r}^{\sigma_r}.
$$
Therefore, since the coefficients of $H$ are symmetric, we get \eqref{formula_H_Z2}.
\end{proof}

\section{A Birkhoff normal form theorem}

\label{sec_BNF_LR}

The following theorem is the main technical result of this paper. It is a new generalization of the classical Birkhoff normal form theorem for Hamiltonian PDEs. In its statement some parameters may seem surprising. They aim at handling some critical situations ($N_{\max}$ is useful for \eqref{eq_NLS} on $\mathbb{T}$ and $\eta$ to deal with \eqref{eq_NLS_2}).  Furthermore, the statement and the proof are quite heavy because we pay attention to track most of the dependencies. This is especially useful to deal with intricate situations like \eqref{eq_NLS_2} where we have to optimize some parameters ($\eta$ and $\mathbf{Z}_d$) with respect to $\varepsilon \equiv \|u^{(0)} \|_{h^s}$. In the favorable cases (e.g. the nonlinear Klein-Gordon or Schr\"odinger equations with Dirichlet boundary conditions in dimension $1$), we will choose $\eta\approx 1$, $N_{\max}=+\infty$, $\mathbf{Z}_d = \mathbb{Z}$ or $\mathbf{Z}_d = \mathbb{N}\setminus \{0\}$.
\begin{theorem} 
\label{thm_NF}
Let $d\geq 1$, $s\geq 0$, $r>p\geq 3$, $\mathbf{Z}_d \subset \mathbb{Z}^d$, $N_{\max}\in [1,+\infty]$ and $\eta>0$.

Let $Z_2 : h^s(\mathbf{Z}_d) \to \mathbb{R}$ be a quadratic Hamiltonian of the form
$$
Z_2(u) = \frac12 \sum_{n\in \mathbf{Z}_d} \omega_n |u_n|^2
$$
where $( \langle n \rangle^{-2s} \omega_n)_{n\in \mathbf{Z}_d}$ is bounded and the family of frequencies $\omega \in \mathcal{D}_{\gamma,N_{\max}}^{\beta,r}(\mathbf{Z}_d)$ is strongly non-resonant, up to order $r$, for small divisors involving at least one mode of index smaller than $N_{\max}$ and for some positive constants $(\beta_j)_{3\leq j\leq r},(\gamma_j)_{3\leq j\leq r}$ (according to Definition \ref{def_nonres}).

Let $P : h^s(\mathbf{Z}_d) \to \mathbb{R}$ be a $\alpha$-inhomogeneous, $q$-smoothing polynomial Hamiltonian of the form
$$
P(u) = \sum_{p\leq j\leq r-1} P^{(j)}(u)
$$
 where  $P^{(j)} \in \mathscr{H}_{q,\alpha}^j(\mathbf{Z}_d) $ satisfies $\| P^{(j)}\|_{q,\alpha} \leq c_j \eta^{-(j-2)}$ and $(c_j)_{p\leq j\leq r-1}$ is a sequence of positive constants and $(\alpha,q)$ satisfy the estimates $\alpha >\max(d-q,d/2)$,  $d/2-q<s<\alpha-d/2+q$ and $q\geq 0$ .

There exists some positive constants $C$ depending on $(r,s,\gamma,\alpha,d,q,c)$ and $b$ depending only on $(\beta,r)$ such that
for all $N\in [1,N_{\max}]$, there exists $\varepsilon_0 \geq \,  \eta /(C N^{b})$ and there exist two smooth symplectic maps $\tau^{(0)}$ and $\tau^{(1)}$ 
making the following diagram to commute
\begin{equation}
\label{mon_beau_diagram_roi_des_forets}
\xymatrixcolsep{5pc} \xymatrix{  B_{h^s(\mathbf{Z}_d)}(0, \varepsilon_0 ) \ar[r]^{ \tau^{(0)} }
 \ar@/_1pc/[rr]_{\mathrm{id}_{{h}^s}} &  B_{h^s(\mathbf{Z}_d)}(0,2\, \varepsilon_0)  \ar[r]^{ \tau^{(1)} }  & 
 h^s(\mathbf{Z}_d)  } 
\end{equation}
and close to the identity
\begin{align*}
\forall \sigma\in \{0,1\}, \ \|u\|_{{h}^s} <2^{\sigma}\varepsilon_0 \;\; \Rightarrow \;\;  
\|\tau^{(\sigma)}(u)-u\|_{h^s} \leq  \left( \frac{\|u\|_{h^s}}{\varepsilon_0} \right)^{p-2} \|u\|_{h^s}
\end{align*}
such that, on $ B_{h^s(\mathbf{Z}_d)}(0,2\varepsilon_0)$, $(Z_2+P) \circ \tau^{(1)}$ admits the decomposition
\begin{equation}\label{eq_NF}
(Z_2+P) \circ \tau^{(1)} = Z_2+ Q^{\leq N}_{\mathrm{res}} + R
\end{equation}
where $Q^{\leq N}_{\mathrm{res}} \in  \mathscr{H}_{q,\alpha}(\mathbf{Z}_d)$ commutes with the low super-actions
\begin{equation}
\label{eq_ca_commte}
\forall n \in \mathbf{Z}_d, \ \langle n \rangle \leq N \ \Rightarrow \ \{ J_n , Q^{\leq N}_{\mathrm{res}} \} = 0, \quad \mathrm{where} \quad J_n = \sum_{ \omega_k = \omega_n} |u_k|^2
\end{equation}
and the remainder term $R$ is a smooth function on $ B_{h^s(\mathbf{Z}_d)}(0,2\varepsilon_0)$ satisfying
$$
\|\nabla R(u)\|_{h^s} \leq C \eta^{-(r-2)} N^{b} \|u\|_{h^s}^{r-1}.
$$

Moreover, for $\sigma\in \{0,1 \}$ and $u\in B_{h^s(\mathbf{Z}_d)}(0,2^{\sigma}\varepsilon_0)$, $\mathrm{d}\tau^{(\sigma)}(u)$ admits an unique continuous extension from $h^{-s}(\mathbf{Z}_d)$ into $h^{-s}(\mathbf{Z}_d)$ which depends continuously on $u$ and we have the bounds
$$
\|\mathrm{d} \tau^{(\sigma)}(u) \|_{\mathscr{L}(h^{s})} \leq 2^{r-p} \quad \mathrm{and} \quad \|\mathrm{d} \tau^{(\sigma)}(u) \|_{\mathscr{L}(h^{-s})} \leq 2^{r-p}.
$$
\end{theorem}
\begin{proof} We are going to prove by induction on $r_{\star} \in \llbracket p,r \rrbracket$ that there exist\footnote{note that in this proof almost everything depends on $r_\star$. Nevertheless we do not specify it explicitly.} some non-negative constants
\begin{itemize} 
\item  $b^{(1)},(b_j^{(2)})_{ p \leq j \leq r}$ depending only on $(\beta,r_\star)$ and $b^{(3)}$ depending also on $r$,
\item  $C^{(1)},(C_{j}^{(2)})_{ p \leq j \leq r }$ depending only on $(r_\star,s,\gamma,\alpha,d,q,c)$ and $C^{(3)}$ depending also on $r$,
\end{itemize}
such that for all $N\in [1,N_{\max}]$, there exists $\varepsilon_0 = \,  \eta /( C^{(1)} N^{b^{(1)}})$ and there exist two smooth symplectic maps $\tau^{(0)}$ and $\tau^{(1)}$ 
making the diagram \eqref{mon_beau_diagram_roi_des_forets} to commute and close to the identity
\begin{equation}
\label{eq_close_to_the_id_pres}
\forall \sigma\in \{0,1\}, \ \|u\|_{{h}^s} <2^{\sigma}\varepsilon_0 \;\; \Rightarrow \;\;  
\|\tau^{(\sigma)}(u)-u\|_{h^s} \leq \left( \frac{\|u\|_{h^s}}{\varepsilon_0} \right)^{p-2}  \|u\|_{h^s}.
\end{equation}
such that, on $ B_{h^s(\mathbf{Z}_d)}(0,2\varepsilon_0)$, $(Z_2+P) \circ \tau^{(1)}$ admits the decomposition
$$
(Z_2+P) \circ \tau^{(1)} = Z_2+ Q^{(p)} + \dots + Q^{(r-1)}  + R
$$
where $Q^{(j)}  \in \mathscr{H}_{q,\alpha}^j(\mathbf{Z}_d)$ satisfies $ \| Q^{(j)} \|_{q,\alpha} \leq  C_{j}^{(2)} \eta^{-(j-2)} N^{b_j^{(2)}}$ and the  firsts polynomials commute with the low super-actions
$$
|j| < r_{\star} \quad \mathrm{and} \quad \langle n \rangle\leq N  \quad \Rightarrow \quad \{ J_n , Q^{(j)} \} = 0
$$
and the remainder term $R$ is a smooth function on $B_{h^s(\mathbf{Z}_d)}(0,2\varepsilon_0)$ satisfying
$$
\|\nabla R(u)\|_{h^s} \leq C^{(3)} \eta^{-(r-2)} N^{b^{(3)}} \|u\|_{H^s}^{r-1}.
$$
Moreover, for $\sigma\in \{0,1 \}$ and $u\in B_{h^s(\mathbf{Z}_d)}(0,2^{\sigma}\varepsilon_0)$, $\mathrm{d}\tau^{(\sigma)}(u)$ admits an unique continuous extension from $h^{-s}(\mathbf{Z}_d)$ into $h^{-s}(\mathbf{Z}_d)$ which depends continuously on $u$ and we have the bounds
$\|\mathrm{d} \tau^{(\sigma)}(u) \|_{\mathscr{L}(h^{s})} \leq 2^{r_{\star}-p} $ and $\|\mathrm{d} \tau^{(\sigma)}(u) \|_{\mathscr{L}(h^{-s})} \leq 2^{r_{\star}-p}.$ 

\medskip

Note that when $r_{\star}=r$ this result is slightly stronger than the one of Theorem \ref{thm_NF}.

\medskip

\noindent $\bullet$ \emph{Initialization $r_\star = p$}. We set $\tau^{(0)}=\tau^{(1)}=\mathrm{id}_{h^s}$, $\varepsilon_0= +\infty$, $Q^{(j)}=P^{(j)}$, $C_{j}^{(2)}=c_j$, $C^{(1)}=C^{(3)}=0$, $b^{(1)}=b_{j}^{(2)}=b^{(3)}=0$ and $R =0$. Note that, since by assumption $\| P^{(j)}\|_{q,\alpha} \leq c_j \eta^{-(j+2)}$, the estimate on $\| Q^{(j)}\|_{q,\alpha}$ holds. Here the decomposition of $(Z_2+P) \circ \tau^{(1)}$ and its properties are obvious. 

\medskip 

\noindent $\bullet$ \emph{Induction step $(r_\star) \Rightarrow (r_{\star}+1)$}. In order to distinguish the constants and objects associated with the index $r_{\star}+1$ from those corresponding to the index $r_{\star}$ , they are denoted with a symbol sharp $\sharp$.

We aim at the removing the terms of $ Q^{(r_\star)}$ which do not commute with the low super actions. As a consequence, we decompose $Q^{(r_\star)}$ as 
$$
Q^{(r_\star)} = L+U 
$$
where $L,U \in  \mathscr{H}_{q,\alpha}^{r_\star}(\mathbf{Z}_d)$ are defined by
$$
 L_n^\sigma = \left\{\begin{array}{cll} Q^{(r_\star),\sigma}_n & \mathrm{if} & \kappa_\omega(\sigma,n) \leq N \\
0 & \mathrm{else}
\end{array} \right. \quad \mathrm{and} \quad  U_n^\sigma = \left\{\begin{array}{cll} 0 & \mathrm{if} & \kappa_\omega(\sigma,n) \leq N \\
Q^{(r_\star),\sigma}_n & \mathrm{else}
\end{array} \right.
$$
Note that, since these Hamiltonians are extracted from $Q^{(r_\star)}$, they satisfy the same estimate.

\medskip

\noindent \underline{$\triangleright$ \emph{$U$ commutes with the low super actions.}} Indeed, applying Lemma \ref{lem_Z2_vs_H}, if $\langle m \rangle\leq N$, for $u\in h^s$, we have
\begin{multline*}
\{ J_m , U \}(u) =2 \, i   \sum_{\sigma \in \{-1,1\}^{r_\star}} \sum_{n\in \mathbf{Z}_d^{r_\star} } (\sigma_1 \mathbb{1}_{\omega_{n_1}=\omega_m} + \dots + \sigma_{r_\star}  \mathbb{1}_{\omega_{n_{r_\star}}=\omega_m})  U_{n}^{\sigma} u_{n_1}^{\sigma_1} \dots  u_{n_{r_\star}}^{\sigma_{r_\star}} 
\\
=2 \, i  \sum_{\sigma \in \{-1,1\}^{r_\star}} \sum_{n\in \mathbf{Z}_d^{r_\star} } \big(  \sum_{\omega_{n_k}=\omega_{m}} \sigma_k\big)  U_{n}^{\sigma} u_{n_1}^{\sigma_1} \dots u_{n_{r_\star}}^{\sigma_{r_\star}} 
.
\end{multline*}
However, since $ \langle m \rangle \leq N$, by definition of $U$ and $\kappa$ (see \eqref{eq_def_kappa}), either $\sum_{\omega_{n_k}=\omega_{m}} \sigma_k$  vanishes or $U_{n}^{\sigma}$  vanishes. Consequently $U$ and $J_m$ commute : $\{ J_m , U \}(u)=0$.

\medskip

\noindent \underline{$\triangleright$ \emph{The homological equation.}} We set $\chi \in \mathscr{H}_{q,\alpha}^{r_\star}(\mathbf{Z}_d)$ the Hamiltonian defined by
$$
\chi_n^\sigma = \frac{L_n^\sigma}{i(\sigma_1 \omega_{n_1} + \dots + \sigma_{r_\star} \omega_{n_{r_\star}} )} \quad \mathrm{if} \quad \kappa_\omega(\sigma,n) \leq N \quad \mathrm{and} \quad \chi_n^\sigma = 0 \quad \mathrm{else}.
$$
Since the frequencies are strongly non-resonant (see Definition \ref{def_nonres}), if $\kappa_\omega(\sigma,n) \leq N$ then 
$$
|\sigma_1 \omega_{n_1} + \dots + \sigma_{r_\star} \omega_{n_{r_\star}} |\geq \gamma_{r_\star} N^{-\beta_{r_\star}} \neq 0.
$$
Therefore, we have 
\begin{equation}
\label{eq_bound_on_chi}
\| \chi\|_{q,\alpha} \leq \frac12 \gamma_{r_\star}^{-1} N^{\beta_{r_\star}} \|Q^{(r_\star)}  \|_{q,\alpha}\leq \frac12 \gamma_{r_\star}^{-1} C_{r_\star}^{(3)} \eta^{-(r_\star-2)} N^{\beta_{r_\star}+ b_{r_\star}^{(3)}}.
\end{equation}
Note that, as a consequence of  Lemma \ref{lem_Z2_vs_H}, $\chi$ solves the homological equation
$$
\{ \chi, Z_2\} + L =0.
$$

\medskip

\noindent \underline{$\triangleright$ \emph{The new variables.}} As usual, to define the new variables, we want to compose $\tau^{(\sigma)}$ and $\Phi_\chi^t$. Consequently, we pay attention to match their domains of definition (this is a little bit fastidious but it is simple).

 Applying Proposition \ref{prop_ham_flow}, we get a constant $K>0$ depending only on $(s,d,\alpha,q,r_\star)$ such that setting $\varepsilon_1 = (K  \|\chi \|_{q,\alpha})^{-1/(r_\star-2)}$, $\chi$ generates a smooth map 
$$
\Phi_\chi : \left\{ \begin{array}{cll} [-1,1] \times B_{h^s(\mathbf{Z}_d)}(0,\varepsilon_1) &\to& h^s(\mathbf{Z}_d) \\ (t,u) &\mapsto& \Phi_\chi^t(u) \end{array} \right.
$$
solving the equation 
$
-i\partial_t \Phi_\chi = ( \nabla \chi)\circ \Phi_\chi,$
and such that for all $t\in [-1,1]$, $\Phi_\chi^t$ is symplectic, close to the identity
\begin{equation}
\label{eq_dindon_new}
\forall u \in  B_{h^s(\mathbf{Z}_d)}(0,\varepsilon_1), \quad \| \Phi_\chi^t u - u \|_{h^s} \leq  \left( \frac{\| u \|_{h^s}}{\varepsilon_1}\right)^{r_\star-2} \| u \|_{h^s},
\end{equation}
invertible
\begin{equation}
\label{eq_invertibility_new}
\|\Phi_\chi^{-t} (u) \|_{h^s} < \varepsilon_1 \quad \Rightarrow \quad  \Phi_\chi^{t}\circ  \Phi_\chi^{-t} (u) = u
\end{equation}
and its differential admits an unique continuous extension from $h^{-s}(\mathbf{Z}_d)$ into $h^{-s}(\mathbf{Z}_d)$. Moreover, the map $u\in  B_{h^s(\mathbf{Z}_d)}(0,\varepsilon_1) \mapsto \mathrm{d} \Phi_\chi^t(u) \in \mathscr{L}(h^{-s}(\mathbf{Z}_d))$ is continuous and we have the estimates
\begin{equation}
\label{eq_petanque_new}
\forall u \in  B_{h^s(\mathbf{Z}_d)}(0,\varepsilon_1),\forall \sigma\in \{-1,1\},\ \| \mathrm{d} \Phi_\chi^t (u)\|_{\mathscr{L}(h^{\sigma s})}\leq 2.
\end{equation}

Recalling the bound \eqref{eq_bound_on_chi} on $\chi$, we have
$$
 \varepsilon_1 \geq ( \frac12 K  \gamma_{r_\star}^{-1} C_{r_\star}^{(2)} \eta^{-(r_\star-2)} N^{\beta_{r_\star}+ b_{r_\star}^{(2)}} )^{-1/(r_{\star}-2)}  \geq \,  3 \eta /( C^{(1)}_\sharp N^{b^{(1)}_\sharp})=: 3 \varepsilon_0^\sharp
$$
where we have set 
$$ C^{(1)}_\sharp = 3 \max(C^{(1)}, ( \frac12 K  \gamma_{r_\star}^{-1} C_{r_\star}^{(2)})^{1/(r_\star-2)},1 ) \quad \mathrm{and} \quad b^{(1)}_{\sharp}=\max(b^{(1)},(\beta_{r_\star}+ b_{r_\star}^{(2)})/(r_{\star}-2)).
$$ Note that with this choice, we have 
$$
3 \varepsilon_0^\sharp \leq \min(\varepsilon_0,\varepsilon_1).
$$
As a consequence, since $\tau^{(0)},\Phi_\chi^t$ are close to the identity (see \eqref{eq_close_to_the_id_pres},\eqref{eq_dindon_new}) and $p\geq 3$, we have
$$
\| u \|_{h^s} \leq \frac83 \varepsilon_0^\sharp \ \Rightarrow \ \|\tau^{(0)}(u)-u\|_{h^s} \leq \frac13 \left( \frac{\|u\|_{h^s}}{\varepsilon_0^\sharp} \right)^{p-2}  \|u\|_{h^s},
$$
$$
\| u \|_{h^s} \leq 2 \varepsilon_0^\sharp  \ \Rightarrow \ \|\Phi_\chi^t(u)-u\|_{h^s} \leq \frac13 \left( \frac{\|u\|_{h^s}}{\varepsilon_0^\sharp} \right)^{r_\star-2}  \|u\|_{h^s}.
$$
Therefore, it makes sense to define 
$$
\tau^{(1)}_\sharp := \tau^{(1)} \circ \Phi_\chi^1\quad \mathrm{on} \quad B_s(0,2 \varepsilon_0^\sharp)\quad \mathrm{and} \quad \tau^{(0)}_\sharp := \Phi_\chi^{-1}\circ \tau^{(0)}\quad \mathrm{on} \quad B_s(0, \varepsilon_0^\sharp).
$$
 Since $\Phi_\chi^{1} \circ \Phi_\chi^{-1} =\mathrm{id}_{h^s}$ (see \eqref{eq_invertibility_new}), these new transformations still make the diagram \eqref{mon_beau_diagram_roi_des_forets} to commute. Since $2/3\leq 1$ and $r_\star \geq p$, they are still close to the identity (i.e. they satisfy \eqref{eq_dindon_new}). Since $\Phi_\chi^t$ is symplectic, $\tau^{(0)}_\sharp,\tau^{(1)}_\sharp$  are symplectic. Finally the existence of the continuous extensions of $\mathrm{d}\Phi_\chi^t,\mathrm{d} \tau^{(0)}, \mathrm{d} \tau^{(1)}$ ensures the existence of such an extension for $\mathrm{d}\tau^{(0)}_\sharp,\mathrm{d}\tau^{(1)}_\sharp$ satisfying the expected bounds.

\medskip

\noindent \underline{$\triangleright$ \emph{The new Hamiltonian.}} Now we aim at studying the expansion of $(Z_2+P)\circ \tau^{(1)}_\sharp$.

Let $u\in  B_s(0,2 \varepsilon_0^\sharp)$. Since $t\mapsto \Phi_\chi^t(u)$ is smooth and is the solution of the equation $\partial_t \Phi_\chi(u) = i ( \nabla \chi)\circ \Phi_\chi(u)$, if $A$ is a smooth Hamiltonian on $h^s$, we have
$$
\partial_t A \circ \Phi_\chi^t := \{\chi , A \} \circ \Phi_\chi^t=: (\mathrm{ad}_{\chi} A) \circ \Phi_\chi^t.
$$
As a consequence, realizing a Taylor expansion in $t=0$ (and omitting the evaluation in $u$) we have
\begin{multline*}
(Z_2+P)\circ \tau^{(1)}_\sharp = (Z_2+P)\circ \tau^{(1)} \circ \Phi_\chi^1 
= Z_2 \circ  \Phi_\chi^1  + \sum_{j=p}^{r-1} Q^{(j)} \circ  \Phi_\chi^1    + R \circ \Phi_\chi^1 \\
= Z_2 + \sum_{j=p}^{r-1} Q^{(j)} + \{\chi,Z_2\} + \sum_{j=p}^{r-1} \sum_{k=1}^{m_{j}} \frac1{k !} \mathrm{ad}_{\chi}^k Q^{(j)}  + \sum_{k=1}^{m_{r_\star}} \frac1{(k+1) !} \mathrm{ad}_{\chi}^{k+1} Z_2  + R \circ \Phi_\chi^1 \\
+  \int_0^1 \frac{(1-t)^{m_{r_\star} +1 }}{(m_{r_\star} +1) !}  (\mathrm{ad}_{\chi}^{m_{r_\star} +2} Z_2)\circ \Phi_\chi^t  +   \sum_{j=p}^{r-1} \frac{(1-t)^{m_j}}{m_j !}  (\mathrm{ad}_{\chi}^{m_j+1} Q^{(j)}) \circ \Phi_\chi^t  \  \mathrm{d}t 
\end{multline*}
where $m_j$ denotes the largest integer such that $j+m_j (r_\star -2) < r$. Recalling that by construction $\{\chi,Z_2\} = -L \in  \mathscr{H}_{q,\alpha}^{r_\star}(\mathbf{Z}_d)$ is of order $r_\star$, that $\chi \in \mathscr{H}_{q,\alpha}^{r_\star}(\mathbf{Z}_d)$ is of order $r_\star$ and that by Proposition \ref{prop_stable_Poisson} the Poisson bracket of Hamiltonians of order $r_1$ and $r_2$ is of order $r_1+r_2-2$, it is natural to set
$$
Q^{(j)}_\sharp = Q^{(j)} \quad \mathrm{if} \quad j < r_\star, \quad Q^{(r_\star)}_\sharp = Q^{(r_\star)}+  \{\chi,Z_2\} =  Q^{(r_\star)} - L =U,
$$
$$
Q^{(j)}_\sharp =  \sum_{ j_\star+k(r_\star - 2)= j} \frac1{k !} \mathrm{ad}_{\chi}^k Q^{(j_\star)} - \sum_{ r_\star+k(r_\star - 2)= j} \frac1{(k+1) !} \mathrm{ad}_{\chi}^k L  \quad \mathrm{if} \quad j > r_\star
$$
$$
R_\sharp=  R \circ \Phi_\chi^1 
-  \int_0^1 \frac{(1-t)^{m_{r_\star} +1 }}{(m_{r_\star} +1) !}  (\mathrm{ad}_{\chi}^{m_{r_\star} +1} L)\circ \Phi_\chi^t  +   \sum_{j=p}^{r-1} \frac{(1-t)^{m_j}}{m_j !}  (\mathrm{ad}_{\chi}^{m_j+1} Q^{(j)}) \circ \Phi_\chi^t  \  \mathrm{d}t .
$$
where $k$ and $j_\star$ are the indices on which the sums hold in the definition of $Q^{(j)}_\sharp$.\\
 If  $j\leq r_\star$, it is clear that $Q^{(j)}_\sharp \in \mathscr{H}_{q,\alpha}^{j}(\mathbf{Z}_d)$ and we have
$$
\| Q^{(j)}_\sharp \|_{q,\alpha} \leq \| Q^{(j)} \|_{q,\alpha} \leq C_{j}^{(2)} \eta^{-(j-2)} N^{b_j^{(2)}} =:  C_{\sharp,j}^{(2)} \eta^{-(j-2)} N^{b_{\sharp,j}^{(2)}}.
$$
Note that by construction it is clear that these Hamiltonians commute with the low super-actions.\\
 Applying Proposition \ref{prop_stable_Poisson}, if $p\leq j< r$, $Q^{(j)}_\sharp \in \mathscr{H}_{q,\alpha}^{j}(\mathbf{Z}_d)$. Moreover, if $j>r_\star$ and $j_\star+k(r_\star - 2)= j$
\begin{multline*}
\! \! \! \! \! \! \| \mathrm{ad}_{\chi}^k Q^{(j_\star)} \|_{q,\alpha} \leq K_{j_\star,r_\star,j} \| \chi \|_{q,\alpha}^k \| Q^{(j_\star)} \|_{q,\alpha} \leq K_{j_\star,r_\star,j} \big(C_{r_\star}^{(2)} \eta^{-(r_\star-2)} N^{b_{r_\star}^{(2)}} \big)^k  (C_{j_\star}^{(2)} \eta^{-(j_\star-2)} N^{b_{j_\star}^{(2)}}) \\
\leq K_{j_\star,r_\star,j} \big(C_{r_\star}^{(2)}  \big)^k  C_{j_\star}^{(2)}  \eta^{-(j-2)}  N^{kb_{r_\star}^{(2)} + b_{j_\star}^{(2)} }
\end{multline*}
where $K_{j_\star,r_\star,j}$ is the constant provided by the continuity estimate \eqref{eq_cont_est_Poisson} of Proposition \ref{prop_stable_Poisson} and depending only on $(j_\star,r_\star,j,\alpha,d)$. Since $\mathrm{ad}_{\chi}^k L$ enjoys the same estimate as $\mathrm{ad}_{\chi}^k Q^{(r_\star)}$, we deduce that $\| Q^{(j)}_\sharp \|_{q,\alpha} \leq C_{\sharp,j}^{(2)} \eta^{-(j-2)} N^{b_{\sharp,j}^{(2)}}$, where we have set\footnote{in these sums we include the case $k=0$.}
$$
C_{\sharp,j}^{(2)} = 2 \! \! \! \sum_{ j_\star+k(r_\star - 2)= j} \frac1{k !} K_{j_\star,r_\star,j} \big(C_{r_\star}^{(2)}  \big)^k  C_{j_\star}^{(2)} \quad \mathrm{and} \quad b_{\sharp,j}^{(2)} = \max_{j_\star+k(r_\star - 2)= j} kb_{r_\star}^{(2)} + b_{j_\star}^{(2)}.
$$

\medskip

\noindent \underline{$\triangleright$ \emph{Control of the remainder term.}} Finally we just have to control the terms of the new remainder term $R_\sharp$. We fix $u\in B_{h^s}(0,2 \varepsilon_0^\sharp)$. First we focus on $R \circ \Phi_\chi^1(u)$. By composition, we have
$$
\nabla (R \circ \Phi_\chi^1)(u) = (\mathrm{d}\Phi_\chi^1(u))^* (\nabla R) \circ \Phi_\chi^1(u).
$$
where $ (\mathrm{d}\Phi_\chi^1(u))^* \in \mathscr{L}(h^{-s})$ denotes the adjoint of $\mathrm{d}\Phi_\chi^1(u)$. Note that since $R \circ \Phi_\chi^1$ is a smooth real valued function on a ball of $h^s$, a priori its gradient belongs to $h^{-s}$. Nevertheless since $\mathrm{d}\Phi_\chi^1(u)$ admits a continuous extension in $\mathscr{L}(h^{-s})$, $(\mathrm{d}\Phi_\chi^1(u))^*$ maps $h^s$ into $h^s$ and we have $\|(\mathrm{d}\Phi_\chi^1(u))^* \|_{\mathscr{L}(h^{s})} = \|\mathrm{d}\Phi_\chi^1(u) \|_{\mathscr{L}(h^{-s})} \leq 2$ . Therefore, $\nabla (R \circ \Phi_\chi^1)(u)$ belongs to $h^s$ and we have
$$
\| \nabla (R \circ \Phi_\chi^1)(u) \|_{h^s} \leq 2 \| (\nabla R) \circ \Phi_\chi^1(u) \|_{h^s} \leq 2  C^{(3)} \eta^{-(r-2)} N^{b^{(3)}} \|\Phi_\chi^1(u) \|_{H^s}^{r-1} 
\leq  2^r  C^{(3)} \eta^{-(r-2)} N^{b^{(3)}} \| u \|_{H^s}^{r-1}.
$$

Now, we focus on $(\mathrm{ad}_{\chi}^{m_j+1} Q^{(j)}) \circ \Phi_\chi^t(u) $ where $p\leq j \leq r-1$ and $t\in [0,1]$. Reasoning as previously, using Proposition \ref{prop_stable_Poisson} to estimate the norm of the Poisson brackets and Proposition \ref{prop_vf} to estimate the norm of the gradient, we have
\begin{equation*}
\begin{split}
&\| \nabla ((\mathrm{ad}_{\chi}^{m_j+1} Q^{(j)}) \circ \Phi_\chi^t)(u) \|_{h^s} \leq 2 \|  (\nabla (\mathrm{ad}_{\chi}^{m_j+1} Q^{(j)}))\circ \Phi_\chi^t (u)  \|_{h^s} \\
\leq& 2 K_{j,r_\star,r_j} \big(C_{r_\star}^{(2)} \eta^{-(r_\star-2)} N^{b_{r_\star}^{(2)}} \big)^{m_j+1}  (C_{j}^{(2)} \eta^{-(j-2)} N^{b_{j}^{(2)}}) M_{r_j}  \|  \Phi_\chi^t (u)  \|_{h^s}^{r_j-1} \\
\leq& \big[ 2^{r_j} K_{j,r_\star,r_j} M_{r_j} \big(C_{r_\star}^{(2)})^{m_j+1}  C_{j}^{(2)} \big] \ \eta^{-(r_j -2)} \|u\|_{h^s}^{r_j-1} N^{b_{r_\star}^{(2)} (m_j+1) +  b_{j}^{(2)} } 
\end{split}
\end{equation*}
where $r_j = j+(m_j+1)(r_\star - 2) \in \llbracket r,2r-4\rrbracket$ and $M_{r_j}$ denotes the  implicit constant in the vector field estimate \eqref{eq_cont_X} of Proposition \ref{prop_vf} (it depends only on $(s,d,r_j,q,\alpha)$).
 Recalling that $\| u \|_{h^s} \leq 2 \varepsilon_0^\sharp= 2 \eta/ (C^{(1)}_\sharp N^{b^{(1)}_\sharp})$ and $C^{(1)}_\sharp \leq 1   $ we deduce that $\| u \|_{h^s}  \eta^{-1} \leq 2$ so that
 $$
 \| \nabla ((\mathrm{ad}_{\chi}^{m_j+1} Q^{(j)}) \circ \Phi_\chi^t)(u) \|_{h^s} \leq A_j  \ \eta^{-(r -2)} \|u\|_{h^s}^{r-1} N^{b_{r_\star}^{(2)} (m_j+1) +  b_{j}^{(2)}  } .
 $$
where $A_j :=  2^{2r_j-r} K_{j,r_\star,r_j} M_{r_j} \big(C_{r_\star}^{(2)})^{m_j+1}  C_{j}^{(2)}$. 

We also recall that by construction $\nabla ((\mathrm{ad}_{\chi}^{m_{r\star}+1} L) \circ \Phi_\chi^t)(u) $ and $\nabla ((\mathrm{ad}_{\chi}^{m_{r_\star}+1} Q^{(r_\star)}) \circ \Phi_\chi^t)(u)$ enjoys the same estimate. As a consequence, we have\footnote{note that the rigorous justification of the permutation between the integral and the gradient could be done easily using the smoothness of $(\mathrm{ad}_{\chi}^{m_j+1} Q^{(j)}) \circ \Phi_\chi^t$.}
\begin{multline*}
\| \nabla R_\sharp(u) \|_{h^s} \leq  2^r  C^{(3)} \eta^{-(r-2)} N^{b^{(3)}} \| u \|_{H^s}^{r-1} + 2\sum_{j=p}^{r-1} \frac1{m_j !} A_j  \ \eta^{-(r -2)} \|u\|_{h^s}^{r-1} N^{b_{r_\star}^{(2)} (m_j+1) +  b_{j}^{(2)}  } \\
\leq C^{(3)}_\sharp \eta^{-(r-2)} N^{b^{(3)}_\sharp} \|u\|_{H^s}^{r-1}
\end{multline*}
where we have set
$$
b^{(3)}_\sharp = \max_{j = p,\dots,r-1} (b^{(3)} , b_{r_\star}^{(2)} (m_j+1) +  b_{j}^{(2)}) \quad \mathrm{and} \quad  C^{(3)}_\sharp = 2^r  C^{(3)} + 2\sum_{j=p}^{r-1} \frac1{m_j !} A_j .
$$
\end{proof}

\section{Dynamical corollary}
\label{sec_dyn_co}

The following theorem is the main abstract result of this paper. It is a corollary of the normal form Theorem \ref{thm_NF}. We recall, to facilitate the lecture of the statement, that in the most favorable case $\eta\approx 1$ and $N_{\max}= +\infty$.
\begin{theorem}
\label{thm_main_dyn} With the same assumptions and notations as in Theorem \ref{thm_NF} and two additional arbitrary constants $K>0,\varepsilon_1>0$, if $u\in C^0_b(\mathbb{R};h^s(\mathbf{Z}_d))\cap C^1(\mathbb{R};h^{-s}(\mathbf{Z}_d))$ is a global solution of
\begin{equation}
\label{eq_non_auto}
i\partial_t u(t) = \nabla Z_2(u(t)) + \nabla P(u(t)) + F(t)
\end{equation}
where $F\in C^0_b(\mathbb{R};h^{-s}(\mathbf{Z}_d))$ and $u$ satisfy
$$
\forall t\in \mathbb{R}, \quad \| F(t) \|_{h^{-s}} \leq K \eta^{-(r-2)}\varepsilon_1^{r-1} \quad \mathrm{and} \quad \| u(t) \|_{h^s} \leq \varepsilon_1
$$
and the frequencies are coercive\footnote{note that, if $\mathbf{Z}_d$ is bounded, the frequencies are coercive by convention.} (i.e. $|\omega_n| \to \infty$ as $|n| \to \infty$),
 then
\begin{equation}
\label{eq_Willy}
|t|< \left(\frac{\varepsilon_1 }{\eta} \right)^{-(r-p)}  \mathrm{and} \quad \langle n \rangle \leq N_{\max} \quad  \Longrightarrow \ |J_n(u(t))-J_n(u(0))|\leq M \eta ^{-(p-2)}  \langle n \rangle^{\mathfrak{b}} \varepsilon_1^{p}
\end{equation}
where $M$ depends only on $(K,r,s,\gamma,\alpha,d,q,c)$ and $\mathfrak{b}$ depends only on $(\beta,r)$.
\end{theorem}
\begin{proof} Let $n\in \mathbf{Z}_d$ be such that $\langle n \rangle \leq N_{\max}$. If $\varepsilon_1 \geq  \eta /( C \langle n \rangle^{b})$, where $C$ is defined in Theorem \ref{thm_NF}, then 
$$
|J_n(u(t))-J_n(u(0))| \leq \|u(t) \|_{\ell^2}^2 + \|u(0) \|_{\ell^2}^2 \leq 2 \varepsilon_1^2 
\leq 2 ( C \eta^{-1}\langle n \rangle^{b} )^{p-2} \varepsilon_1^p .
$$
Consequently, we only have focus on the case $\varepsilon_1 <  \eta /( C \langle n \rangle^{b})$. Therefore, we set 
$$
N= \langle n \rangle,
$$
 and we apply  Theorem \ref{thm_NF}. Note that we have
  $$
  \forall t \in \mathbb{R}, \ \| u(t) \|_{h^s}\leq \varepsilon_1 < \eta /( C N^{b}) <\varepsilon_0.
  $$ 
 As a consequence it makes sense to consider 
$$
v(t) := \tau^{(0)}(u(t)).
$$

\medskip

\noindent \underline{$\triangleright$ \emph{Time differentiability of $v$}}. First we have to check that $v$ is time differentiable and to compute its derivative. Since $\tau^{(0)}$ is not defined on $h^{-s}$, a priori this fact is not obvious. Nevertheless, $\mathrm{d}\tau^{(0)}$ can be extended to $\mathscr{L}(h^{-s})$ and it is sufficient. \\
Let us clarify this point. We fix $t\in \mathbb{R}$ and we consider a small parameter $h\in (-1,1)\setminus \{0\}$.
Since $\tau^{(0)}$ is smooth on $h^s$, we have
$$
v(t+h)-v(t) = \tau^{(0)}(u(t+h)) -  \tau^{(0)}(u(t)) = \int_0^1 \mathrm{d}\tau^{(0)}( u_{\nu,t,h})(u(t+h) - u(t)) \  \mathrm{d}\nu
$$
where $u_{\nu,t,h} = \nu u(t+h) +(1-\nu) u(t)$. For clarity, we denote by $L^{(0)}$ (resp $L^{(1)}$)  the continuous extension of $\mathrm{d}\tau^{(0)}$ (resp. $\mathrm{d}\tau^{(1)}$) to $\mathscr{L}(h^{-s})$ and thus we have
$$
\frac{v(t+h)-v(t)}h = \int_0^1 L^{(0)}( u_{\nu,t,h}) \  \mathrm{d}\nu \ \Big(\frac{u(t+h) - u(t)}{h}\Big)
$$
and so, since $ \| L^{(0)}( u_{\nu,t,h}) \|_{\mathscr{L}(h^{-s})}\leq 2^{r-p}$, we have
\begin{equation*}
\begin{split}
&\left\| \frac{v(t+h)-v(t)}h - L^{(0)}(u(t))(\partial_t u(t)) \right\|_{h^{-s}}  \\
\leq & \left\| \int_0^1 L^{(0)}( u_{\nu,t,h}) \  \mathrm{d}\nu \ \Big(\frac{u(t+h) - u(t)}{h} -\partial_t u(t)\Big) \right\|_{h^{-s}} 
 + \left\| \int_0^1 L^{(0)}( u_{\nu,t,h}) - L^{(0)}(u(t)) \  \mathrm{d}\nu \ \big(\partial_t u(t)\big) \right\|_{h^{-s}} \\
 \leq & 2 \Big\| \frac{u(t+h) - u(t)}{h} -\partial_t u(t)\Big \|_{h^{-s}} + \| \partial_t u(t) \|_{h^{-s}} \int_0^1 \| L^{(0)}( u_{\nu,t,h}) - L^{(0)}(u(t)) \|_{\mathscr{L}(h^{-s})}  \mathrm{d}\nu  
\end{split}
\end{equation*}
Since $u\in C^1(\mathbb{R};h^{-s})$, the first term goes to $0$ as $h$ goes to $0$. To prove the same for the second term, we apply the Lebesgue's dominated convergence theorem. Indeed, on the one hand, we have the bound $ \| L^{(0)}( u_{\nu,t,h}) - L^{(0)}(u(t)) \|_{\mathscr{L}(h^{-s})} \leq 2^{r-p+1}$ and on the other hand, $\nu$ being fixed, since $u\in C^0(\mathbb{R};h^{s})$, $u_{\nu,t,h}$ converges to $u(t)$ as $h$ goes to $0$ and so, since $L^{(0)}$ is continuous, $\| L^{(0)}( u_{\nu,t,h}) - L^{(0)}(u(t)) \|_{\mathscr{L}(h^{-s})}$ goes to $0$ as $h$ goes to $0$.

As a consequence, $v$ is time derivable and we have
$$
\partial_t v(t) =  L^{(0)}(u(t))(\partial_t u(t)).
$$

\medskip 

\noindent \underline{$\triangleright$ \emph{Computation of $\partial_t v$}}. Since, by assumption, $u$ solves the equation \eqref{eq_non_auto}, denoting $H=Z_2+P$, we have
$$
\partial_t v(t) = -  L^{(0)}(u(t))( i \nabla H(u(t)) + iF(t) ) = - L^{(0)}(u(t))( i \nabla H (u(t)) ) -i F^\sharp (t)
$$
where we have set $  F^\sharp (t):=-iL^{(0)}(u(t))(iF(t))$. First we assume the following relation (which is proven at the end of this paragraph) on $B_{h^s}(0,\varepsilon_0)$
\begin{equation}
\label{eq_because_symplectic}
 L^{(0)} \, i= i \, ((\mathrm{d}\tau^{(1)}) \circ \tau^{(0)} )^* 
\end{equation}
where  $((\mathrm{d}\tau^{(1)}) \circ \tau^{(0)} )^*  \in \mathscr{L}(h^{-s})$ denotes the adjoint of $(\mathrm{d}\tau^{(1)}) \circ \tau^{(0)} $. Therefore, we have
$$
\partial_t v(t) = -i \, (\mathrm{d}\tau^{(1)}(v(t))^*( \nabla H (u(t)) ) -i F^\sharp (t).
$$
However, since the diagram \eqref{mon_beau_diagram_roi_des_forets} commutes, we have $u(t) = \tau^{(1)} (v(t))$ and so
$$
\partial_t v(t) = -i \, (\mathrm{d}\tau^{(1)}(v(t))^*( (\nabla H)\circ \tau^{(1)}(v(t)) ) -i F^\sharp (t) = -i \nabla (H\circ \tau^{(1)})(v(t)) -i F^\sharp (t).
$$
As a consequence, recalling the decomposition \eqref{eq_NF} of $(Z_2+P)\circ \tau^{(1)}$, we have
\begin{equation}
\label{eq_PDE_for_v}
i\partial_t v(t) = \nabla Z_2(v(t)) + \nabla Q^{\leq N}_{\mathrm{res}}(v(t)) + \nabla R(v(t)) + F^\sharp (t).
\end{equation}
Now we focus on \eqref{eq_because_symplectic}. Since $\tau^{(1)}$ is symplectic (see Definition \ref{def_symplectic}), we have
$$
((\mathrm{d}\tau^{(1)}) \circ \tau^{(0)} )^* \ i \ (\mathrm{d}\tau^{(1)}) \circ \tau^{(0)} = i.
$$
However, since the diagram \eqref{mon_beau_diagram_roi_des_forets} commutes, we have $\tau^{(1)} \circ \tau^{(0)} =  \mathrm{id}_{h^s}$ and so
$$
((\mathrm{d}\tau^{(1)} ) \circ \tau^{(0)} ) \ \mathrm{d} \tau^{(0)} = \mathrm{id}_{h^s}.
$$
As a consequence we deduce that
$$
((\mathrm{d}\tau^{(1)}) \circ \tau^{(0)} )^* \ i = i \ \mathrm{d} \tau^{(0)}.
$$
Extending this relation by density from $\mathscr{L}(h^s;h^{-s})$ to  $\mathscr{L}(h^{-s};h^{-s})$, we get \eqref{eq_because_symplectic}.

\medskip

\noindent \underline{$\triangleright$ \emph{Estimation of $\partial_t J_n(v(t))$}}. Since the frequencies are coercive, the sum defining $J_n$ (see \eqref{eq_ca_commte}) is finite. Consequently it is a smooth function on $h^{-s}$. Therefore, $t\mapsto J_n(v(t)) \in C^1(\mathbb{R};\mathbb{R})$ and recalling that $\partial_t v$ satisfies \eqref{eq_PDE_for_v} we have
$$
\partial_t J_n(v(t)) = \{J_n , Z_2   \}(v(t))+ \{J_n , Q^{\leq N}_{\mathrm{res}}   \}(v(t)) + \{J_n , R   \}(v(t))  + (\nabla J_n(v(t)) ,-i F^\sharp(t))_{\ell^2} .
$$ 
First, we recall that by construction $\{J_n , Q^{\leq N}_{\mathrm{res}}   \}=0$. Similarly, we have $\{J_n , Z_2   \}=0$. Indeed, we have
$$
\{J_n , Z_2   \}(v(t)) =  (i\nabla J_n(v(t)), \nabla Z_2(v(t)))_{\ell^2}=2 \sum_{\omega_k = \omega_n} \omega_k \Re (i |v_k(t)|^2) = 0.
$$
Therefore, we have $|\partial_t J_n(v(t))| \leq | i(\nabla J_n(v(t)) ,\nabla R(v(t)))_{\ell^2} | +| (\nabla J_n(v(t)) ,-i F^\sharp(t))_{\ell^2}|$ and so
\begin{equation}
\label{eq_coleoptere}
|\partial_t J_n(v(t))| \leq \| \nabla J_n (v(t))  \|_{\ell^2} \| \nabla R (v(t))  \|_{h^s} +  \| \nabla J_n (v(t))  \|_{h^{s}} \| F^\sharp(t)  \|_{h^{-s}}.
\end{equation}
As a consequence, using that 
$$
\|v(t)\|_{h^s} \leq \|u(t) \|_{h^s} + \| \tau^{(0)}(u(t)) - u(t) \|_{h^s} \leq \|u(t) \|_{h^s} + \left( \frac{\| u(t)\|_{h^s}}{\varepsilon_0} \right)^{p-2} \|u(t)\|_{h^s} 
\leq 2 \|u(t) \|_{h^s} \leq 2\varepsilon_1,
$$
we deduce 
\begin{equation*}
\begin{split}
\| \nabla J_n (v(t))  \|_{h^s}^2 &= 4 \sum_{\omega_k = \omega_n} \langle k \rangle^{2s} |v_k(t)|^2  \leq 4 \| v(t) \|_{h^s}^2 \leq  16\,  \varepsilon_1^2, \\
\| \nabla R (v(t))  \|_{h^s} &\leq C \eta^{-(r-2)} N^{b} \|v(t)\|_{H^s}^{r-1} \leq 2^{r-1} C \eta^{-(r-2)} \langle n\rangle^{b} \varepsilon_1^{r-1}, \\
\| F^\sharp(t)  \|_{h^{-s}} &\leq \| L^{(0)}(u(t)) \|_{\mathscr{L}(h^{-s})} \| F(t) \|_{h^{-s}} \leq  2^{r-p}  K \eta^{-(r-2)} \varepsilon_1^{r-1}.
\end{split}
\end{equation*}
Plugging these estimates in \eqref{eq_coleoptere}, we get 
$$
|\partial_t J_n(v(t))| \leq M^\sharp   \eta^{-(r-2)} \langle n \rangle^{b} \varepsilon_1^{r}
$$ 
where we have set $M^\sharp = 2^{r+2}  (C+K)$.
Therefore, by the mean value inequality, we have
$$
|t|< \left(\frac{ \varepsilon_1}{\eta} \right)^{-(r-p)} \quad \Longrightarrow \quad |J_n(v(t)) - J_n(v(0))| \leq  M^\sharp   \eta^{-(p-2)} \langle n \rangle^{b} \varepsilon_1^{p}
$$
\medskip

\noindent \underline{$\triangleright$ \emph{Conclusion}}. To prove the same result for $u(t)$, we use that $u(t)$ is close to $v(t)$ (i.e. that $\tau^{(0)}$ is close to the identity) :
$$
\| u(t) - v(t) \|_{h^s}  \leq \left( \frac{\|u(t) \|_{h^s}}{\varepsilon_0} \right)^{p-2} \| u(t) \|_{h^s} \leq  C^{p-2}   \langle n \rangle^{b(p-2)} \eta^{p-2}  \varepsilon_1^{p-1}
$$
and that $J_n$ is quadratic
\begin{equation*}
\begin{split}
| J_n(v(t)) - J_n(u(t)) | &\leq (\| v(t) \|_{\ell^2} + \| u(t) \|_{\ell^2} ) \| u(t) -v(t)\|_{\ell^2} \\
&\leq  3 \, C^{p-2}   \langle n \rangle^{b(p-2)} \eta^{p-2}  \varepsilon_1^{p}.
\end{split}
\end{equation*}
As a consequence setting $\mathfrak{b} = b(p-2)$ and $M = 3 \, C^{p-2}  + M^\sharp$, we get the estimate we wanted to prove (i.e. \eqref{eq_Willy}).
\end{proof}

\section{Proofs of the applications}
\label{sec_proofs}

\subsection{Application to nonlinear Klein-Gordon equations}
\label{sec_KG}
In this section, we aim at proving the results about the Klein--Gordon equations \eqref{eq_KG} stated in the subsection \ref{sub_sec_KG}.

First, we start with the proof of  Lemma \ref{lemma_Ham_KG_coer} about the ellipticity of the Klein--Gordon's Hamiltonian.
\begin{proof}[\bf Proof of Lemma \ref{lemma_Ham_KG_coer}] We are going to establish two estimates of which this lemma is a direct corollary. On the one hand, the Poincar\'e inequality proves that the term associated with the mass can be "absorbed"
$$
\int_0^\pi (\Phi(x))^2 \mathrm{d}x \leq \int_0^\pi (\partial_x \Phi(x))^2 \mathrm{d}x.
$$ 
On the other hand, in the estimate \eqref{eq_neg_ok} below, we prove that the nonlinearity can be neglected. Indeed, there exists an universal constant $C>0$ such that
$$
\| \Phi \|_{L^\infty} \leq C \| \Phi \|_{H^1}.
$$
As a consequence,  since $G$ is of order $p$ (with respect to its second variable), there exists an universal constant $K>0$ such that we have
$$
|y|\leq C \ \Rightarrow \ \|G(\cdot,y)\|_{L^\infty} \leq C\|G(\cdot,y)\|_{H^1}  \leq K |y|^p.
$$
Therefore, provided that $\| \Phi \|_{H^1_0}\leq 1$, we have
\begin{equation}
\label{eq_neg_ok}
\int_0^\pi |G(x,\Phi(x))| \ \mathrm{d}x \leq K \int_0^\pi |\Phi(x)|^p \ \mathrm{d}x \leq \pi K C^p \| \Phi \|_{H^1}^p.
\end{equation}
\end{proof}

Now, we focus on the proof of the main result about Klein-Gordon equation: Theorem \ref{thm_main_KG}.
\begin{proof}[\bf Proof of Theorem \ref{thm_main_KG}] We fix $m>-1$ in the set of full measure given by Lemma \ref{lem_nr_KG} which ensures that the frequencies $\omega:=(\sqrt{n^2+m})_{n\geq 1}$ are strongly non-resonant (up to any order and for all modes, i.e. $N_{\max}=+\infty$). We consider the Hilbert basis of $L^2(0,\pi;\mathbb{R})$ diagonalizing $\partial_x^2$ with homogeneous Dirichlet boundary conditions :
$$
e_n(x) := \sqrt{\frac2\pi} \sin(n \, x), \quad n\geq 1.
$$
As usual, for all $s\in \mathbb{R}$, we identify sequences in $h^{s}(\mathbb{N}^*)$ with distributions of $\mathcal{D}'(0,\pi)$ through the formula
$$
(u_n)_{n\geq 1} \mapsto \sum_{n\geq 1} u_n e_n.
$$
Note that this identification induces the following isometries 
\begin{equation}
\label{eq_nice_isom}
H^1_0([0,\pi];\mathbb{R}) \equiv h^1(\mathbb{N}^*;\mathbb{R}), \ H^{-1}(0,\pi;\mathbb{R}) \equiv  h^{-1}(\mathbb{N}^*;\mathbb{R})\ \mathrm{and} \ L^{2}(0,\pi;\mathbb{R}) \equiv  \ell^2(\mathbb{N}^*;\mathbb{R}).
\end{equation}
For all $s\in \mathbb{R}$, we define the diagonal operator $\Omega :h^s(\mathbb{N}^*) \to h^{s-1/2}(\mathbb{N}^*)$ by the relation
$$
\forall u\in h^s, \ \Omega u = \sum_{n\geq 1} \sqrt[4]{n^2+m} \ u_n \ e_n.
$$
As usual, in order to diagonalize the linear part of \eqref{eq_KG}, we define the complex variables
$$
u(t) := \Omega   \Phi(t)+i \ \Omega^{-1} \partial_t \Phi(t) .
$$
Note that, thanks to the isometries \eqref{eq_nice_isom}, if $(\Phi,\partial_t \Phi) \in C^0_b(\mathbb{R}; H^1_0 \times L^2 ) \cap C^1(\mathbb{R}; L^2 \times H^{-1} )$, $u \in C^0_b(\mathbb{R}; h^{1/2}) \cap C^1(\mathbb{R}; h^{-1/2})$ and if $\Phi$ solves \eqref{eq_KG}, then $u$ solves the equation
\begin{equation}
\label{eq_ca_commence_a_ressembler_a_qqch}
i\partial_t u(t) = \Omega^2 u(t) - \Omega^{-1} g(\, \cdot\, ,\Omega^{-1} \Re u(t)).
\end{equation}
Furthermore, we have an upper bound on the norm of $u$. Indeed, as a consequence of the strong convexity of $H$ (see Lemma \ref{lemma_Ham_KG_coer}) and its preservation (Theorem \ref{thm_GWP_KG}), for all $t\in \mathbb{R}$, we have the bound
\begin{equation*}
\begin{split}
 (\| \Phi(t) \|_{H^1} + \| \partial_t \Phi(t) \|_{L^2})^2 \leq \Lambda_m H(\Phi(t),\partial_t \Phi(t)) = \Lambda_m H(\Phi^{(0)},\dot{\Phi}^{(0)}) 
&\leq \Lambda_m^2 (\|\Phi^{(0)}\|_{H^1} +\|\dot{\Phi}^{(0)}\|_{ L^2})^2 \\ &= \Lambda_m^2 \varepsilon^2
\end{split}
\end{equation*}
and so there exists $C_m>0$ such that
\begin{multline*}
\forall t\in \mathbb{R}, \ \|u(t)\|_{h^{1/2}}^2 = \sum_{k\geq 1} \langle k \rangle |u_k(t)|^2 =   \sum_{k\geq 1} \langle k \rangle \left( \sqrt{k^2+m} \Phi_k^2(t) +  \frac{(\partial_t \Phi_k(t))^2}{ \sqrt{k^2+m} } \right) \\
\leq C_m  \sum_{k\geq 1}  \langle k \rangle^2 \, \Phi_k^2(t) +  (\partial_t \Phi_k(t))^2 \leq C_m\Lambda_m^2 (\| \Phi(t) \|_{H^1} + \| \partial_t \Phi(t) \|_{L^2})^2 =  C_m\Lambda_m^2 \varepsilon^2 =:\varepsilon_1^2.
\end{multline*}
In order to use the class of Hamiltonian we introduced in section \ref{sec_class_Ham}, 
we rewrite the equation \eqref{eq_ca_commence_a_ressembler_a_qqch} using the Taylor expansion of $y\mapsto g(\cdot,y)$ in $y=0$ (we recall that by assumption $g$ is of order $p-1$ in $0$):
\begin{equation}
\label{eq_ca_ressemble_presque_a_qqch}
i\partial_t u(t) = \Omega^2 u(t) - \Omega^{-1} \sum_{j=p-1}^{r+p-2} \frac{( \Omega^{-1} \Re u(t) )^j}{j !} g_j + F(t). 
\end{equation}
where $g_j := \partial_{y}^{j} g(\cdot,0) \in H^1([0,\pi];\mathbb{R})$ and 
$$
F(t) =  -\Omega^{-1} \left[ (\Phi(t))^{r+p-1} \int_0^1 \frac{(1-\lambda)^{r+p-1}}{(r+p-1)!}  \partial_{y}^{r+p} g(\cdot,\lambda  \Phi(t) ) \, \mathrm{d}\lambda \right].
$$

\noindent \underline{$\triangleright$ \emph{Identification of the Hamiltonian structure.} }
First, recalling that $\omega_k := \sqrt{k^2+m}$, we note that obviously, we have
$$
\Omega^2 u = \nabla Z_2(u) \quad \mathrm{where} \quad Z_2(u) = \frac12 \sum_{k\geq 1} \omega_k |u_k|^2.
$$
So we only focus on the structure of the nonlinear part of \eqref{eq_ca_ressemble_presque_a_qqch}. Using Sobolev embeddings it is clear that  
$$
P^{(j)} : v\mapsto - \frac1{j !}\int_0^\pi g_{j-1}(x) (\Omega^{-1} \Re v)^j(x) \ \mathrm{d}x
$$
is a smooth function on $h^{1/2}$ and that
$$
\nabla P^{(j+1)}(v) = - \frac1{j !} \Omega^{-1}\left[   ( \Omega^{-1} \Re v )^j g_j\right].
$$
Therefore, as a consequence of the Taylor expansion \eqref{eq_ca_ressemble_presque_a_qqch}, $u$ solves the equation
\begin{equation}
\label{eq_ca_ressemble_a_ca}
i\partial_t u(t) = \nabla (Z_2 + \sum_{j=p}^{r+p-1} P^{(j)}) u(t) + F(t).
\end{equation}
Hence, we just have to identity $P^{(j)}$ with a formal Hamiltonian in $\mathscr{H}_{\frac12,1}^j(\mathbb{N}^*)$.
Indeed, we have
\begin{align*}
 P^{(j)}(v) &= -\frac1{j !}\int_0^\pi g_{j-1}(x) \big(\sum_{k\geq 1}   \frac{\Re v_k }{(k^2+m)^{1/4}} e_k(x) \big)^j \ \mathrm{d}x \\
 &= -\frac{(-i)^j}{j !} (8\pi)^{-j/2} \int_0^\pi g_{j-1}(x) \big(\sum_{k\geq 1}   \frac{v_k + \overline{v_k}}{(k^2+m)^{1/4}} (e^{ikx} - e^{-ikx})\big)^j \ \mathrm{d}x \\
 &= \sum_{\sigma\in \{-1,1\}^j} \sum_{k\in (\mathbb{N}^*)^j}  P^{(j),\sigma}_k v_{k_1}^{\sigma_1} \cdots v_{k_j}^{\sigma_j}
\end{align*}
where we recall that $v_k^{-1}:= \overline{v_k}$ and we have denoted
$$
P^{(j),\sigma}_k := -\frac{(-i)^j}{j !} (8\pi)^{-j/2}   \sum_{\mu \in \{-1,1\}^j}   \int_0^\pi g_{j+1}(x) \prod_{\ell=1}^j \frac{\mu_{\ell}\, e^{i\mu_{\ell}k_{\ell}x}}{(k_\ell^2+m)^{1/4}}  \ \mathrm{d}x. 
$$
It remains to estimate the coefficients $P^{(j),\sigma}_k$. Indeed, we have
$$
|P^{(j),\sigma}_k| \lesssim_{j,m} \sum_{\mu \in \{-1,1\}^j}   \left| \int_0^\pi g_{j-1}(x) e^{i(\mu \cdot k) x}  \mathrm{d}x \right| \sqrt{ \langle k_1 \rangle \cdots \langle k_\ell \rangle }^{-1}
$$
and if $\mu \cdot k \neq 0$, integrating by part, we have
$$
 \int_0^\pi g_{j-1}(x) e^{i(\mu \cdot k) x}  \mathrm{d}x =-i (\mu \cdot k)^{-1} \left( e^{i\pi(\mu \cdot k) }  g_{j-1}(\pi) - g_{j-1}(0) -  \int_0^\pi \partial_x g_{j-1}(x) e^{i(\mu \cdot k) x}  \mathrm{d}x \right).
$$
Therefore, since by assumption $g_{j-1} \in H^1$, using the Sobolev inequalities, we deduce that
$$
|P^{(j),\sigma}_k| \lesssim_{j,m}  \sum_{\mu \in \{-1,1\}^j}   \langle \mu \cdot k \rangle^{-1}  \sqrt{ \langle k_1 \rangle \cdots \langle k_\ell \rangle }^{-1}
$$
which proves that $P^{(j)} \in \mathscr{H}_{\frac12,1}^j(\mathbb{N}^*)$ and $\| P^{(j)} \|_{\frac12,1} \lesssim_{m,j} 1$ (here we choose\footnote{it is just a convenient way to remove a constant in the statement of the Theorem. } $\eta = \sqrt{C_m} \Lambda_m= \varepsilon_1/\varepsilon\approx_m 1$ to apply Theorem \ref{thm_main_dyn}).

\noindent \underline{$\triangleright$ \emph{Estimate of the remainder term.} } Recalling that we have the a priori bound $\| \Phi(t) \|_{H^1} \leq \Lambda_m \varepsilon \leq  \Lambda_m \varepsilon_m $ and using the Sobolev embedding $H^1 \to C^0$, we get a constant $B_m$ (depending only on $m$) such that
$$
\forall t\in \mathbb{R}, \forall x\in [0,\pi], \ | \Phi(t,x) | \leq B_m.
$$
Moreover, since $y\mapsto  \partial_{y}^{r+p} g(\cdot,y)$ is continuous, using the Sobolev embedding $H^1 \hookrightarrow C^0$, we get a constant $A_m$ (depending only on $m$) such that
$$
\sup_{|y|\leq B_m} \sup_{x\in [0,\pi]} | \partial_{y}^{r+p} g(x,y)| \leq A_m.
$$
Therefore, we have
$$
\sup_{t\in \mathbb{R}} \sup_{0 \leq x \leq \pi} \sup_{0 \leq \lambda \leq 1} |\partial_{y}^{r+p} g(x,\lambda  \Phi(t,x) )| \leq A_m.
$$
As a consequence, using once again the constant associated with the Sobolev embedding $H^1 \hookrightarrow C^0$, we deduce that for all $t\in \mathbb{R}$,
\begin{multline*}
\| F(t) \|_{h^{-1/2}} \leq \| F(t) \|_{\ell^2} \lesssim_m \big\| \left[ \Phi^{r+p-1}(t) \int_0^1 \frac{(1-\lambda)^{r+p-1}}{(r+p-1)!}  \partial_{y}^{r+p} g(\cdot,\lambda  \Phi(t) ) \, \mathrm{d}\lambda \right] \big\|_{L^2} \\
 \lesssim_{m,r}  \| \Phi(t) \|_{L^2} \| \Phi(t) \|_{L^\infty}^{r+p-2} \lesssim_{m,r}  \| \Phi(t) \|_{H^1}^{r+p-1} \lesssim_{m,r} \varepsilon_1^{r+p-1}. 
\end{multline*}

\noindent \underline{$\triangleright$ \emph{Conclusion.} } Finally,  observing that $0=d/2-q \leq s=1/2 \leq \alpha - d/2 +q =1 $ and $1=\alpha > \max(d-q,d/2) =\max(1-1/2,1/2)=1/2$, recalling that $u$ solves the equation \eqref{eq_ca_ressemble_a_ca} and applying Theorem \ref{thm_main_dyn} (of which we have checked all the assumptions) with $N_{\max}=+\infty$ and the change of notation $r\leftarrow r+p$, we get the almost global preservation of the low harmonic energies which are the super actions of the nonlinear Klein Gordon equation \eqref{eq_KG}.
\end{proof}

\subsection{Application to nonlinear Schr\"odinger equations in dimension one}
\label{sec_NLS}
In this section, we aim at proving the results about the nonlinear Schr\"odinger equations \eqref{eq_NLS} in dimension $d=1$ stated in the subsection \ref{sub_sec_NLS}. We recall that the probabilistic results have been proven in Section \ref{sec_NR}.

\subsubsection{Some properties of the Sturm--Liouville eigenfunctions.}
First, we need to collect some useful properties on the Sturm--Liouville eigenfunctions $f_n$ (defined in Proposition \ref{prop_def_SL}). As in section \ref{sec_NR}, most of the time, we do not specify the dependency of $f_n$ and $\lambda_n$ with respect to $V$.
\begin{lemma} \label{lem_well_loc} For all $V\in L^2(0,\pi;\mathbb{R})$ we have
\begin{equation}
\label{eq_well_loc_Dir}
\forall k,n \in \mathbb{N}^*, \  \big| \int_0^\pi f_n(x) \sin(k x) \mathrm{d}x \big| \lesssim_{\|V\|_{L^2}}  \mathbb{1}_{n=k}  + (\langle n - k  \rangle \langle n + k  \rangle)^{-1}
\end{equation}
\begin{equation}
\label{eq_well_loc_Neu}
\forall k,n \in \mathbb{N}, \  \big| \int_0^\pi f_{-n}(x) \cos(k x) \mathrm{d}x \big|  \lesssim_{\|V\|_{L^2}} \mathbb{1}_{n=k}  + (\langle n - k  \rangle \langle n + k  \rangle)^{-1}.
\end{equation}
\end{lemma}
\begin{proof} We only focus on\eqref{eq_well_loc_Dir}. The proof of \eqref{eq_well_loc_Neu} is similar. Since $-\partial_x^2 + V$ is self adjoint on $L^2$, we have
\begin{equation}
\label{eq_on_sen_sort_pas}
\begin{split}
\lambda_n ( f_n(x) ,\sin(k x))_{L^2} = (  (-\partial_x^2 + V) f_n(x) ,\sin(k x))_{L^2} &= (   f_n(x) ,(-\partial_x^2 + V) \sin(k x))_{L^2} \\
&= k^2 ( f_n(x) ,\sin(k x))_{L^2} +  ( f_n(x) ,V \sin(k x))_{L^2}.
\end{split}
\end{equation}
Moreover, by Theorem 4 page 35 of \cite{PT}, we know there exists $b \in \ell^{\infty}(\mathbb{N}^*)$ such that 
$\lambda_n = n^2 + b_n.$
Therefore, there exists $C>0$ (depending on $\|V\|_{L^2}$), such that if $n+k>C$ and $n\neq k$ then 
$$|\lambda_n - k^2| \geq |n^2 - k^2|/2 \gtrsim   \langle n - k  \rangle \langle n + k  \rangle .$$
Since it is clear that $|( f_n(x) ,\sin(k x))_{L^2} |\lesssim \sqrt{\pi/2}$ and since there are only finitely many indices such that $n+k\leq C$, by \eqref{eq_on_sen_sort_pas} we get \eqref{eq_well_loc_Dir}.
\end{proof}

\begin{proposition}
\label{prop_isom_Hilb}
For all $V\in L^2(0,\pi;\mathbb{R})$, the following maps are isomorphisms of Banach spaces
$$
\Psi^{\mathrm{Dir}}:\left\{ \begin{array}{ccc} H^1_0(0,\pi;\mathbb{R}) & \to & h^1(\mathbb{N}^*;\mathbb{R}) \\
u & \mapsto & (\int_0^\pi u(x) f_n(x) \mathrm{d}x)_n
\end{array}  \right. \quad \quad \Psi^{\mathrm{Neu}}:\left\{ \begin{array}{ccc} H^1(0,\pi;\mathbb{R}) & \to & h^1(\mathbb{N};\mathbb{R}) \\
u & \mapsto & (\int_0^\pi u(x) f_{-n}(x) \mathrm{d}x)_n
\end{array}  \right. .
$$ 
\end{proposition}
\begin{proof} As usual we only focus on $\Psi^{\mathrm{Dir}}$. Note that assuming it is well defined, since $(f_n)_n$ is an Hilbertian basis of $L^2$ (see Proposition \ref{prop_def_SL}), its injectivity is obvious. Furthermore, by the Banach isomorphism Theorem, it is enough to prove that it is continuous and surjective (i.e. actually the continuity of the inverse follows from the proof of the surjectivity).

\noindent  \underline{$\triangleright$ \emph{Continuity.}} First, we check that $\Psi^{\mathrm{Dir}}$ is well defined and is continuous.
We define
\begin{equation}
\label{eq_plein_de_monde}
c_{n,k} =\sqrt{\frac{2}\pi} \int_0^\pi f_n(x) \sin(k x) \mathrm{d}x, \! \quad v_k =\sqrt{\frac{2}\pi} \int_0^\pi u(x) \sin(k x) \mathrm{d}x, \! \quad w_n = \int_0^\pi u(x) f_n(x) \mathrm{d}x, \! \quad z_n = c_{n,n} v_n.
\end{equation}
We aim at proving that $\| w\|_{h^1} \lesssim \| u \|_{H^1}$. We recall that it is well known that $\| v \|_{h^1} \lesssim \| u\|_{H^1}$.  By the triangular inequality  we have
$$
\| w\|_{h^1} \leq \| z \|_{h^1} + \|w - z \|_{h^1} \lesssim \| v \|_{h^1} +  \|w - z \|_{h^1} \lesssim \| u \|_{H^1} +  \|w - z \|_{h^1}.
$$
Therefore, we only have to focus on  $\|w - z \|_{h^1}$. Since $(2/\pi)^{1/2}(\sin(kx))_k$ is an Hilbertian basis of $L^2(0,\pi;\mathbb{R})$, we have $
w_n - z_n = \sum_{k\neq n} v_k c_{n,k}.$
Consequently, applying Lemma \ref{lem_well_loc}, we deduce that\footnote{we do not pay attention to the dependency with respect to $\|V\|_{L^2}$.}
$$
\|w - z \|_{h^1}^2 = \sum_{n\geq 1} \langle n \rangle^2  \big(  \sum_{k\neq n} v_k c_{n,k} \big)^2 \lesssim \sum_{n\geq 1}   \big(  \sum_{k\neq n} \frac{|v_k| \langle n \rangle}{\langle n-k \rangle \langle n+k \rangle }  \big)^2 \lesssim \sum_{n\geq 1}   \big(  \sum_{k\neq n} \frac{|v_k| \langle k \rangle}{\langle n-k \rangle \langle k \rangle }  \big)^2.
$$
However, by a straightforward generalization of Lemma \ref{lem_conv1} of the appendix, we have
$$
 \sum_{k\neq n} \frac{|v_k| \langle k \rangle}{\langle n-k \rangle \langle k \rangle } \lesssim \frac{\| v\|_{h^1}}{\langle n\rangle} \quad \mathrm{and  } \ \mathrm{so}  \quad \|w - z \|_{h^1}^2\lesssim \| v\|_{h^1}^2 \lesssim \|u\|_{H^1}^2.
$$

\noindent  \underline{\noindent $\triangleright$ \emph{Surjectivity.}} Let $w\in h^1$ and let us set $u = \sum_{n\geq 1} w_n f_n$ (a priori in $L^2$). We just have to prove that $u \in H^1_0$. Naturally it is enough to prove that $\|v\|_{h^1} \leq \|w\|_{h^1}$ (where $v$ is defined by \eqref{eq_plein_de_monde}). Denoting $y_n = w_n c_{n,n}$, it is clear that $\| y\|_{h^1} \leq \|w\|_{h^1}$. Therefore, by the triangle inequality, it is enough to prove that  $\|v-y\|_{h^1} \leq \|w\|_{h^1}$. Observing that, by definition, we have
$$
\|v-y\|_{h^1}^2 = \sum_{k \geq 1} \langle k\rangle^2 \big(  \sum_{k\neq n} w_n c_{n,k} \big)^2 \lesssim \sum_{k \geq 1}  \big(  \sum_{k\neq n} \frac{|w_n| \langle k \rangle}{\langle n-k \rangle \langle n+k \rangle }   \big)^2 ,
$$
arguing as we did for the continuity estimate, we deduce that $\|v-y\|_{h^1}^2\lesssim \| w \|_{h^1}^2$.
\end{proof}
As a straightforward corollary we get the following result on $\mathbb{T}$.
\begin{corollary}
\label{cor_id_per}
For all $V\in L^2(\mathbb{T};\mathbb{R})$ even, the following map is an isomorphism of Banach spaces
$$
\Psi^{\mathrm{per}}:\left\{ \begin{array}{ccc} H^1(\mathbb{T};\mathbb{R}) & \to & h^1(\mathbb{Z};\mathbb{R}) \\
u & \mapsto & \frac1{\sqrt2} (\int_\mathbb{T} u(x) f_n(x) \mathrm{d}x)_n
\end{array}  \right.
$$
\end{corollary}

Finally, the following lemma deals with the restriction of $\Psi^{\mathrm{Dir}}$ to $H^{s}\cap H^1_0$ for $s\in[1,3/2)$.  We only use it to prove Corollary \ref{cor_whaou}.

\begin{lemma}
\label{lem_inter_fn}
Let $s \in [1,3/2)$ and $V\in L^2(0,\pi;\mathbb{R})$. There exists $C>0$, such that for all $u\in H^{s}(0,\pi;\mathbb{R}) \cap H^1_0(0,\pi;\mathbb{R})$, we have
$$
\sum_{n\geq 1} \langle n \rangle^{2s} \big( \int_0^\pi u(x) f_n(x) \mathrm{d}x \big)^2 \leq C^2 \|u\|_{H^{s}}^2.
$$
\end{lemma}
\begin{proof} We use the notations \eqref{eq_plein_de_monde} that we have introduced to prove Proposition \ref{prop_isom_Hilb}.
We aim at proving that, being given $u\in H^{s} \cap H^1_0$, we have $\| w\|_{h^{s}} \lesssim \| u \|_{H^s}$. It is proven, in Lemma \ref{lem_interpol} of the appendix (since we did not find a proof in the literature), that
 $\| v \|_{h^s} \lesssim \| u\|_{H^s}$.  Therefore, by the triangular inequality  we have
$$
\| w\|_{h^{s}} \leq \| z \|_{h^{s}} + \|w - z \|_{h^{s}} \lesssim \| v \|_{h^{s}}+  \|w - z \|_{h^{s}} \lesssim \| u \|_{H^s} +  \|w - z \|_{h^{s}}.
$$
Finally, proceeding as in the proof of Proposition \ref{prop_isom_Hilb}, since $2(s-1)-2<-1$ (i.e. $s<3/2$), we deduce that
$$
\|w - z \|_{h^{s}}^2 \lesssim \sum_{n\geq 1} \langle n \rangle^{2(s-1)} \| v\|_{h^1}^2   \langle n \rangle^{-2} \lesssim \| v\|_{h^1}^2 \lesssim \| u\|_{H^1}^2 \lesssim \| u\|_{H^s}^2.
$$

\end{proof}

\subsubsection{Proofs of the results of the subsection \ref{sub_sec_NLS}.}

\begin{proof}[\bf Proof of Lemma \ref{lemma_Lag_NLS_coer}] The proof is similar to the proof of Lemma \eqref{lemma_Ham_KG_coer}, so we omit it. 
\end{proof}

\begin{proof}[\bf Proof of Theorem \ref{thm_Dir_1d}] The proof is almost the same as the proof of Theorem \ref{thm_per_1d} below (excepted that, thanks to the stronger non-resonance condition, most of the estimates are uniform with respect to $(r,N)$).
\end{proof}

\begin{proof}[\bf Proof of Theorem \ref{thm_per_1d}] We fix $r\geq 2$ (an even number), $N\geq 1$ and $V\in L^\infty(\mathbb{T};\mathbb{R})$ an even potential such that $(\lambda_n(V_{|(0,\pi)}))_{n\in \mathbb{Z}}$ is strongly non-resonant, up to order $r$, for small divisors involving at least one mode smaller than $N$, according to Definition \ref{def_SNR_lim}. We set $\rho = \|V\|_{L^\infty}$ and $\epsilon_0 = \varepsilon_{\rho}$ (which is defined by Lemma \ref{lemma_Lag_NLS_coer}). Assuming that $u^{(0)}\in H^1(\mathbb{T})$ satisfies $\varepsilon = \| u^{(0)} \|_{H^1} \leq \epsilon_0$, the solution of \eqref{eq_NLS} given by Theorem \ref{thm_GWP_NLS_1d} satisfies
\begin{equation}
\label{eq_glob_cont_norm}
\forall t\in \mathbb{R}, \ \Lambda_\rho^{-1} \| u(t) \|_{H^1}^2 \leq \mathcal{H}(u(t))+ (\rho+1) \mathcal{M}(u(t)) = \mathcal{H}(u(0))+ (\rho+1) \mathcal{M}(u(0)) \leq \Lambda_\rho \varepsilon^2,
\end{equation}
i.e. $\| u(t) \|_{H^1}\leq \Lambda_{\rho} \varepsilon$. Thanks to Corollary \ref{cor_id_per}, we identify functions of $H^1(\mathbb{T}) / L^2(\mathbb{T})/H^{-1}(\mathbb{T})$ with sequences of $h^1(\mathbb{Z})/\ell^2(\mathbb{Z})/h^{-1}(\mathbb{Z})$ through their decompositions in the Hilbertian basis $(f_n/\sqrt{2})_{n\in \mathbb{Z}}$. For example, we denote $u_n(t) = \frac1{\sqrt2} \int_\mathbb{T} u(t,x) f_n(x) \mathrm{d}x$ and we have $\| u(t) \|_{h^1} \approx \| u(t) \|_{H^1}$. Defining $\omega_n = \lambda_n$, \eqref{eq_NLS} rewrites
$$
i\partial_t u_n(t) = \omega_n u_n(t) + (g(\cdot,|u(t,\cdot)|^2) u(t,\cdot))_{n}.
$$
We introduce the Taylor expansion of $y\mapsto g(\cdot,y)$ in $y=0$ at the order $r/2-1$
$$
g(\cdot ,y) = \sum_{j=p/2-1}^{r/2-2} g_j(\cdot) \frac{y^j}{j !} + y^{r/2-1}\int_0^1 \frac{(1-a)^{r/2-2}}{(r/2-2) ! } \partial_y^{r/2-1} g(\cdot, a y) \, \mathrm{d}a
$$
where $g_j := \partial_y^j g(\cdot,0) \in H^2(\mathbb{T};\mathbb{R})$. Therefore, \eqref{eq_NLS} rewrites
$$
i\partial_t u_n(t) = \omega_n u_n(t) + \sum_{j=p/2-1}^{r/2-2} \frac1{j!} (g_j(\cdot) |u(t,\cdot)|^{2j}u(t,\cdot)  )_n  + F_n(t),
$$
the remainder term being defined by
$$
F(t,x)=|u(t,x)|^{r-2}u(t,x)  \int_0^1 \frac{(1-a)^{r/2-2}}{(r/2-2) ! } \partial_y^{r/2-1} g(x, a |u(t,x)|^{2} ) \, \mathrm{d}a.
$$
Reasoning as in the proof of Theorem \ref{thm_main_KG}, it can be easily proven that $\| F\|_{h^{-1}} \leq \| F\|_{\ell^2} = \| F\|_{L^2} \lesssim \varepsilon^{r-1}$.

\noindent \underline{$\triangleright$ \emph{Identification of the Hamiltonian structure.} }
First,concerning the linear part, we note that for all $v\in h^1$
$$
\omega_n v_n = (\nabla Z_2(v))_n \quad \mathrm{where} \quad Z_2(v) = \frac12 \sum_{k\in \mathbb{Z}} \omega_k |v_k|^2.
$$
For the nonlinear part, we define
$$
P^{(2j)}(v)= \int_{\mathbb{T}} \frac1{2j!} g_{j-1}(x) |v(x)|^{2j}   \mathrm{d}x.
$$
It is clearly a smooth function on $h^1$ and we have
$$
\nabla P^{(2j)}(v) =  \frac1{(j-1)!} g_{j-1}(x) |v(x)|^{2j-2} v(x).
$$
Therefore, \eqref{eq_NLS} rewrites
\begin{equation}
\label{raclette_a_la_chevre}
i\partial_t u(t) = \nabla(Z_2 + \sum_{j=p/2}^{r/2-1} P^{(2j)} )(u(t)) + F(t).
\end{equation}
Hence, we just have to identity $P^{(2j)}$ with a formal Hamiltonian in $\mathscr{H}_{0,2}^{2j}(\mathbb{Z})$. Indeed, noting that the formal permutation below are justified by convergence in $H^1$ (thanks to Corollary \ref{cor_id_per}), for $v\in h^1$, we have
\begin{multline*}
\!\!\!\!\!\! P^{(2j)}(v) = \int_{\mathbb{T}} \frac{g_{j-1}(x)}{(2j)!} \Big| \sum_{n \in \mathbb{Z}} v_n f_n(x) \Big|^{2j}   \mathrm{d}x =\sum_{n \in \mathbb{Z}^{2j}}  v_{n_1}\cdots v_{n_{j}} \overline{v_{n_{j+1}} \cdots v_{n_{2j}}}  \int_{\mathbb{T}} \frac{g_{j-1}(x)}{(2j)!}  f_{n_1}(x) \cdots  f_{n_{2j}}(x)    \mathrm{d}x  \\
= \sum_{n \in \mathbb{Z}^{2j}} \sum_{\sigma \in \{-1,1\}^{2j}}   v_{n_1}^{\sigma_1}\cdots v_{n_{2j}}^{\sigma_{2j}}  (P^{(2j)})_n^\sigma 
\end{multline*}
where  $(P^{(2j)})_n^\sigma:= 0$ if $\sigma_1+\dots+ \sigma_{2j}\neq 0$ and, 
$$
(2j)! \binom{2j}{j}  (P^{(2j)})_n^\sigma := \int_{\mathbb{T}} g_{j-1}(x) f_{n_1}(x) \cdots  f_{n_{2j}}(x)  \  \mathrm{d}x  \quad \mathrm{if} \quad \sigma_1+\dots+ \sigma_{2j} = 0.
$$
To estimate these coefficients, we introduce the Fourier basis of $L^2(\mathbb{T};\mathbb{R})$ defined by
$$
\forall n>0, \ e_n(x)=(\pi)^{-\frac12}\sin(n x) \quad \mathrm{and}\quad  e_{-n}(x)=(\pi)^{-\frac12}\cos(n x) \quad \mathrm{and}\quad  e_0(x)=(2\pi)^{-\frac12}.
$$
We note that as a consequence\footnote{if $kn<0$, by parity, we have $(f_n,e_k)_{L^2}=0$} of Lemma \ref{lem_well_loc}, we have 
$$
\forall n,k\in \mathbb{Z}, \ |(f_n,e_k)_{L^2}| \lesssim \langle n-k \rangle^{-2} + \langle n+k \rangle^{-2}.
$$
Therefore, since $g_{j-1}\in H^2$, applying Lemma \ref{lemma_convol} of the appendix, we have
\begin{equation*}
\begin{split}
|(P^{(2j)})_n^\sigma|&\lesssim_j  \sum_{k\in \mathbb{Z}^{2j+1}} | (g_{j-1},e_{k_{2j+1}})_{L^2}  (f_{n_1},e_{k_1})_{L^2} \cdots (f_{n_{2j}},e_{k_{2j}})_{L^2}|  \Big|  \int_{\mathbb{T}} e_{k_1}(x) \cdots e_{k_{2j+1}}(x)  \  \mathrm{d}x  \Big| \\
&\lesssim_j  \sum_{\nu \in \{-1,1\}^{2j+1}}  \sum_{\substack{k\in \mathbb{Z}^{2j+1}\\ \nu \cdot k =0}} \frac{\| g_{j-1} \|_{H^2}}{\langle k_{2j+1} \rangle^2} \prod_{\ell=1}^{2j} \frac1{\langle n_{\ell}-k_{\ell} \rangle^2} + \frac1{\langle n_{\ell}+k_{\ell} \rangle^2} \\
&\lesssim_j \sum_{\nu \in \{-1,1\}^{2j}}  \sum_{k_1+\cdots+k_{2j+1}=0} \langle k_{2j+1} \rangle^{-2} \langle \nu_\ell n_{\ell}-k_{\ell} \rangle^{-2}  \lesssim_j \sum_{\nu \in \{-1,1\}} \langle \nu_1 n_1 + \cdots+ \nu_{2j} n_{2j} \rangle^{-2}.
\end{split}
\end{equation*}
As a consequence, $P^{(2j)} \in \mathscr{H}_{0,2}^{2j}(\mathbb{Z})$ and we have $\| P^{(2j)} \|_{0,2} \lesssim_j 1$.

\noindent \underline{$\triangleright$ \emph{Conclusion.} } Since $u$ is solution of \eqref{raclette_a_la_chevre}, the frequencies are strongly non resonant, the remainder term is of order $r-1$ and the leading polynomial part are controlled in $\mathscr{H}_{0,2}^{2j}(\mathbb{Z})$, to conclude, we just have to apply Theorem \ref{thm_main_dyn} with $s=1$, $\alpha =2$, $q=0$, $\varepsilon_1 = \Lambda_\rho \varepsilon$, $\eta= \Lambda_{\rho}$, $d=1$, $\mathbf{Z}_d=\mathbb{Z}$, $N_{\max}=N$. Indeed, we have $1/2=d/2-q \leq s=1 \leq \alpha - d/2 +q =3/2 $ and $2=\alpha > \max(d-q,d/2) =\max(1-0,1/2)=1$.
\end{proof}

\begin{proof}[\bf Proof of Corollary \ref{cor_whaou}] Without loss of generality, we assume that $s\in (1,3/2)$. We fix $r$ and we apply Theorem \ref{thm_Dir_1d}. Therefore assuming that  $\varepsilon=\| u^{(0)}\|_{H^s}\leq \epsilon_0$, we know that
$$
 \forall n\geq 1, \quad \big|\, |u_n(t)|^2 - |u_n(0)|^2 \big| \leq C_r \langle n \rangle^{\beta_r} \ \varepsilon^p
$$
where $t\in \mathbb{R}$ satisfy $t<|\varepsilon|^{-r}$ is fixed in all this proof.  We choose $\theta_n(t)$  in order to have $e^{i\theta_n(t)}u_n(0) \in \mathbb{R}_+ u_n(t) $. As a consequence, we have
$$
|u_n(t) - e^{i\theta_n(t)} u_n(0)| = ||u_n(t)| -| e^{i\theta_n(t)} u_n(0)| |.
$$
Therefore, as a consequence of Proposition \ref{prop_isom_Hilb}, we have
\begin{multline}
\label{eq_sioux}
\big\| u(t) - \sum_{n\geq 1} e^{i\theta_n(t)} u_n(0) f_n   \big\|_{H^1}^2 \approx \sum_{n\geq 1} \langle n \rangle^2 |u_n(t) - e^{i\theta_n(t)} u_n(0)|^2\approx \sum_{n\geq 1} \langle n \rangle^2 (|u_n(t)| - |u_n(0)|)^2 \\ \lesssim
\sum_{n\geq 1} \langle n \rangle^2 ||u_n(t)|^2 - |u_n(0)|^2| \lesssim C_r N^{\beta_r+3} \varepsilon^p + \sum_{\langle n \rangle \geq N} \langle n \rangle^2  |u_n(0)|^2  + \sum_{\langle n \rangle \geq N} \langle n \rangle^2 |u_n(t)|^2 
\end{multline}
where $N\geq 1$ is a parameter we will optimize with respect to $\varepsilon$ (you can think $N=\varepsilon^{0^-}$). Before, we have to estimate the two last sums of \eqref{eq_sioux} with respect to $N$ and $\varepsilon$. Since $u(0) \in H^s\cap H^{1}_0$ satisfies $\varepsilon= \|u(0)\|_{H^s}$, as a consequence of Lemma \ref{lem_inter_fn}, we have
\begin{equation}
\label{est_rem1}
 \sum_{\langle n \rangle \geq N} \langle n \rangle^2  |u_n(0)|^2 \leq N^{-2(s-1)} \varepsilon^2.
\end{equation}
Now we focus on estimating $\sum_{\langle n \rangle \geq N} \langle n \rangle^2 |u_n(t)|^2$. By Theorem 4 page 35 of \cite{PT}, we know that there exists $n_0>0$ such that 
$$
\forall n\geq n_0, \ \lambda_n \geq n^2/2. 
$$
Therefore, assuming that $N\geq \langle n_0\rangle $, we have
$$
 \sum_{\langle n \rangle \geq N} \langle n \rangle^2  |u_n(t)|^2 \leq 2  \sum_{\langle n \rangle \geq N} \lambda_n  |u_n(t)|^2  
$$
Since the Hamiltonian of \ref{eq_NLS} is a constant of the motion (see Theorem \ref{thm_GWP_NLS_1d}), we know that
$$
 \sum_{n=1}^\infty \lambda_n |u_n(t)|^2 +  \int_0^\pi G(x,|u(t,x)|^2) \ \mathrm{d}x = \sum_{n=1}^\infty \lambda_n |u_n^{(0)}|^2 +  \int_0^\pi G(x,|u^{(0)}(x)|^2) \ \mathrm{d}x
$$
Therefore, recalling that $\lambda_n \lesssim \langle n \rangle^2$, we have
\begin{multline}
\label{eq_tempopo}
 \sum_{\langle n \rangle \geq N} \langle n \rangle^2  |u_n(t)|^2 \lesssim \sum_{\langle n \rangle \geq N} \langle n \rangle^2 |u_n(0)|^2 + \sum_{\langle n \rangle < N } \langle n \rangle^2 ||u_n(0)|^2  - |u_n(t)|^2 | +  | \int_0^\pi G(x,|u(t,x)|^2) \ \mathrm{d}x |\\ +  | \int_0^\pi G(x,|u(0,x)|^2) \ \mathrm{d}x |
\end{multline}
Reasoning as in the proof of Lemma \ref{lemma_Ham_KG_coer}, we prove that
$$
 | \int_0^\pi G(x,|u(t,x)|^2) \ \mathrm{d}x | \lesssim \| u(t) \|_{H^1}^p.
$$
Moreover, using the convexity estimate of Lemma \ref{lemma_Lag_NLS_coer}, we prove, as in \eqref{eq_glob_cont_norm} in the proof of Theorem \ref{thm_per_1d}, that $\|u(t)\|_{H^1}\lesssim \varepsilon$ (we do not track the dependency with respect to $\|V\|_{L^\infty}$). As a consequence, estimating the two first terms of the right hand side of \eqref{eq_tempopo} as before, we deduce that
$$
 \sum_{\langle n \rangle \geq N} \langle n \rangle^2  |u_n(t)|^2 \lesssim_r N^{-2(s-1)} \varepsilon^2  + N^{\beta_r+3} \varepsilon^p.
$$
Consequently, plugging this estimate (and \eqref{est_rem1}) in \eqref{eq_sioux}, yield to
$$
\big\| u(t) - \sum_{n\geq 1} e^{i\theta_n(t)} u_n(0) f_n   \big\|_{H^1}^2  \lesssim_r N^{-2(s-1)} \varepsilon^2  + N^{\beta_r+3} \varepsilon^p.
$$
Finally, to conclude, we just have to optimize this estimate setting $N =\langle n_0\rangle (\epsilon_0/\varepsilon)^{\frac{p-2}{2s+2+\beta_r}}$. Note that, as a consequence, we have $\delta = \frac{(s-1)(p-2)}{2s+2+\beta_r}$.

\end{proof}

\subsection{Application to nonlinear Schr\"odinger equations in dimension two}
\label{sec_NLS_2d}

In this section, we aim at proving the deterministic results of the subsection \ref{sub_sec_NLS2} about \eqref{eq_NLS_2}.

\begin{proof}[\bf Proof of Lemma \ref{lemma_Lag_NLS_coer_2d}] It is a direct corollary of the Sobolev embedding $H^1 \hookrightarrow L^4$.
\end{proof}

To prove our main result about \eqref{eq_NLS_2} (i.e. Theorem \ref{thm_per_2d}) we have to overcome a new issue: $H^1(\mathbb{T}^2)$ is not an algebra. Fortunately it is almost an algebra in the sense that for all $s>1$, $H^s(\mathbb{T}^2)$ is an algebra. Since we are only interested in the long time behavior of the low modes, we trade some extra smoothness against arbitrarily small negative powers of $\varepsilon$ (which correspond to the factor $\varepsilon^{-\delta}$ in Theorem \ref{thm_per_2d}). To achieve this, we optimize the constant $\eta$ and the set $\mathbf{Z}_2$ of Theorem \ref{thm_main_dyn} with respect to $\varepsilon$. Such an optimization is possible because we have paid a lot of attention in Theorem \ref{thm_main_dyn} to have estimates uniform with respect to these constants.

\begin{proof}[\bf Proof of Theorem \ref{thm_per_2d}] We fix a potential $V\in H^1(\mathbb{T}^2)$ such that the frequencies $\omega_n = |n|^2 + \widehat{V}_n$ are strongly non-resonant. We set $\rho = \|V\|_{L^2}$ and $\epsilon_0 = \min(1,\varepsilon_{\rho})$ ($\varepsilon_{\rho}$ is defined by Lemma \ref{lemma_Lag_NLS_coer_2d}). Assuming that $u^{(0)}\in H^1(\mathbb{T}^2)$ satisfies $\varepsilon = \| u^{(0)} \|_{H^1} \leq \epsilon_0$, the solution of \eqref{eq_NLS_2} given by Theorem \ref{thm_GWP_NLS_2d} satisfies, by conservation of the Hamiltonian and of the mass, $\| u(t) \|_{H^1}\leq \Lambda_\rho \varepsilon$ for all $t\in \mathbb{R}$ (see \eqref{eq_glob_cont_norm} for the proof). As usual, thanks to the isometries provided by the Fourier transform, we identify $h^s(\mathbb{Z}^2)$ with $H^s(\mathbb{T}^2)$ for all $s\in \mathbb{R}$ (we omit the symbol $\widehat{\cdot}$ to denote the Fourier transform). We fix $r>0$ and, without loss of generality, we can assume that it is larger than a given constant. Then we set  
$$
u^{\leq N} := (u_n)_{n\in \mathbf{Z}_2}  \ \mathrm{where} \  \mathbf{Z}_2 := \{ n \in\mathbb{Z}^2, \ \langle n \rangle \leq N\}  \quad \mathrm{and} \quad N := \varepsilon^{-3r}.
$$
Then we fix $\delta\in (0,1/2)$ and we are going to deduce from Theorem \ref{thm_main_dyn} that
\begin{equation}
\label{eq_petite_restriction}
|t| \leq \varepsilon^{-r} \quad \Rightarrow \quad \forall n \in \mathbf{Z}_2, \ ||u_n(t)|^2 - |u_n(0)|^2 | \leq  C_{r,\delta} \langle n \rangle^{\beta_r} \ \varepsilon^{4-\delta}.
\end{equation}
It turns out that it is enough to conclude the proof of Theorem \ref{thm_per_2d} since, as $\|u(t)\|_{H^1} \leq \Lambda_\rho \varepsilon$ and assuming that $\beta_r\geq 1$, \eqref{eq_petite_restriction} trivially holds true for $\langle n \rangle > N$. So we only focus on $u^{\leq N}$.

\medskip

\noindent \underline{$\triangleright$ \emph{Time regularity of $u^{\leq N} $}}.  We recall that $u$ is only a weak solution of \eqref{eq_NLS_2} in the sense that it belongs to $L^\infty(\mathbb{R}; H^1  ) \cap W^{1,\infty}(\mathbb{R};H^{-1} )$. Therefore, since $\mathbf{Z}_2$ is bounded, it is clear that $ u^{\leq N} \in C^0_b(\mathbb{R}; h^{1}(\mathbf{Z}_2))= C^0_b(\mathbb{R}; h^{-1}(\mathbf{Z}_2)) $. Moreover, we note that since $\| u(t) \|_{H^1}\leq \Lambda_\rho \varepsilon$, we have $\|u^{\leq N} (t) \|_{h^1} \leq \Lambda_\rho \varepsilon$. We aim at proving that  $u^{\leq N} \in C^1(\mathbb{R}; h^{-1}(\mathbf{Z}_2))$.

A priori, we only know that $u^{\leq N}$ is a Lipschitz function such that for almost all $t\in\mathbb{R}$, we have
\begin{equation}
\label{eq_glaceauchocolat}
i\partial_t u^{\leq N}(t) = \nabla Z_2(u^{\leq N}(t)) + \Pi^{\leq N} (|u(t)|^2 u(t))
\end{equation}
where $\Pi^{\leq N} : \ell^2(\mathbb{Z}^2) \to \ell^2(\mathbf{Z}_2) $ is the natural restriction and $Z_2(u^{\leq N}) := \frac12 \sum_{n\in \mathbf{Z}_2} \omega_n |u_n^{\leq N}|^2 $. We have to prove that $i\partial_t u^{\leq N}$ is continuous (we do not care about which topology since $ h^{-1}(\mathbf{Z}_2)$ is finite dimensional). Actually by \eqref{eq_glaceauchocolat}, it is enough to prove that both $t \mapsto \nabla Z_2(u^{\leq N}(t))$ and $t\mapsto  \Pi^{\leq N} (|u(t)|^2 u(t))$ are continuous.

On the one hand, we note that since $\nabla Z_2$ is linear, it is continuous. Therefore, since $u^{\leq N} \in W^{1,\infty}(\mathbb{R};h^{-1}) \subset C^0(\mathbb{R};h^{-1})$, $t \mapsto \nabla Z_2(u^{\leq N}(t))$ is continuous. 

On the other hand, we have the two following homogeneous estimates (namely H\"older and Gagliardo-Nirenberg)
$$
 \| \cdot\|_{L^2}\lesssim \sqrt{ \| \cdot\|_{H^1}  \| \cdot\|_{H^{-1}}} \quad \mathrm{and} \quad \| \cdot \|_{L^6} \lesssim \| \cdot\|_{H^1}^{2/3}  \| \cdot\|_{L^2}^{1/3} 
$$ 
which imply, since $u\in L^\infty(\mathbb{R}; H^1  ) \cap W^{1,\infty}(\mathbb{R};H^{-1} )$, that $u\in C^0(\mathbb{R};L^6)$ (because it is $1/6$-H\"olderian). Therefore, we have $|u|^2 u \in C^0(\mathbb{R};L^2)$ and thus $t\mapsto \Pi^{\leq N} (|u(t)|^2 u(t))$ is continuous.

\medskip

\noindent \underline{$\triangleright$ \emph{Structure of the nonlinear part}}. The evolution equation of $u^{\leq N}$ is actually non-autonomous (because it depends on $u_n$, $\langle n \rangle >N$). Therefore, we split it between its autonomous part and its non-autonomous part. More precisely, if $\langle n \rangle \leq N$, we set
$$
(|u|^2 u)_n = (2\pi)^{-2} \sum_{k_1+k_2=\ell_1+n} \overline{u}_{\ell} u_{k_1} u_{k_2} = \nabla P^{(4)}(u)+F_n(t)
$$
where 
$$
F_n(t) = (2\pi)^{-2} \! \! \! \! \! \sum_{\substack{
k_1+k_2=\ell_1+n\\ 
\max(\langle k_1 \rangle, \langle k_2 \rangle,\langle \ell \rangle )>N}}  \! \! \! \! \! \overline{u}_{\ell}(t) u_{k_1}(t) u_{k_2}(t) \quad \mathrm{and} \quad P^{(4)}(u)  = (4\pi)^{-2}  \! \! \! \! \! \! \! \! \! \sum_{ \substack{k_1+k_2=\ell_1+\ell_2\\ \max(\langle k_1 \rangle, \langle k_2 \rangle,\langle \ell_1 \rangle,\langle \ell_2 \rangle )\leq N } }   \! \! \! \! \! \overline{u}_{\ell_1} \overline{u}_{\ell_2}  u_{k_1} u_{k_2}.
$$
Up to a straightforward symmetrization,  $P^{(4)}$ can be easily identified with a formal Hamiltonian in $\bigcap_{\alpha\geq 0}\mathscr{H}_{0,\alpha}^{4}(\mathbb{Z}) \subset \mathscr{H}_{0,10}^{4}(\mathbb{Z}) $ (we choose $\alpha=10$ in order to fix it) whose norm satisfy $\|P^{(4)} \|_{0,10} \lesssim 1$. We note that, since $\mathbf{Z}_2$ is bounded, for all $q\geq 0$, we also have $\|P^{(4)} \|_{q,10} \lesssim N^{4q}$. Therefore, defining 
$$\eta = \varepsilon^{\delta/2}/\Lambda_\rho \quad \mathrm{and} \quad q=\delta/12r,
$$
 we have $\|P^{(4)} \|_{q,10} \lesssim (\varepsilon^{-3r})^{\delta/3r}\approx \eta^{-2}$.

\noindent \underline{$\triangleright$ \emph{Estimate of the remainder term}}. We aim at estimating $\|F_n(t)\|_{h^{-1}}$ by $\eta^{-(2r-2)} \varepsilon^{ 2r-1}$. Applying the Young's convolution inequality $\ell^{3/2} \star \ell^{12/11} \star \ell^{12/11} \hookrightarrow \ell^2$, we obtain
$$
\| F(t)\|_{\ell^2}\lesssim \| (\mathbb{1}_{\langle n \rangle > N} u_n(t))_n \|_{\ell^{3/2}} \| u(t) \|_{\ell^{12/11}}^2. 
$$
Then by H\"older, (since $1/6+1/2=2/3$ and $5/12 + 1/2= 11/12$), we have
$$
\| u(t) \|_{\ell^{12/11}} \leq \|u\|_{h^1} \|  (\langle n \rangle^{-1})_n \|_{\ell^{12/5}}  \lesssim \varepsilon
$$
\begin{equation*}
\begin{split}
\| (\mathbb{1}_{\langle n \rangle > N} u_n(t))_n \|_{\ell^{3/2}}  &\leq \|u\|_{h^1} \| (\mathbb{1}_{\langle n \rangle > N} \langle n \rangle^{-1}) \|_{\ell^6} = \|u\|_{h^1} \big( \sum_{ \langle n \rangle > N } \langle n \rangle^{-6} \big)^{\frac16} \\ &\leq \|u\|_{h^1}  N^{- \frac{4r-1}{6r}}  \big( \sum_{ n\in \mathbb{Z}^2 } \langle n \rangle^{-2-\frac{1}r} \big)^{\frac16}  \lesssim_r \varepsilon N^{- \frac{4r-1}{6r}}.
\end{split}
\end{equation*}
Therefore, since $N=\varepsilon^{-3r}$, we have $\| F(t)\|_{\ell^2} \lesssim_r \varepsilon^{3+2r - 1/2} \lesssim_r \varepsilon^{ 2r-1} \lesssim_r \eta^{2r-2} \varepsilon^{ 2r-1}$. In other words, $F(t)$ is a remainder term of order $2r-1$ (in the sense of  Theorem \ref{thm_main_dyn}).

\noindent \underline{$\triangleright$ \emph{Conclusion}}. Since we have $1 - q\leq s=1 \leq \alpha - d/2 +q =9+q $ and $10=\alpha > \max(2-q,1)$, applying  Theorem \ref{thm_main_dyn} (with $r\leftarrow 2r$), we get $C_{r,\delta}$ and $\beta_r$ such that
$$
|t| \leq \varepsilon^{-(2r-p)(1-\delta/2)} \quad \Rightarrow \quad \forall n \in \mathbf{Z}_2, \ ||u_n(t)|^2 - |u_n(0)|^2 | \leq  C_{r,\delta} \langle n \rangle^{\beta_r} \ \varepsilon^{4-\delta}
$$
Therefore it is enough to assume that $r$ is large enough to ensure\footnote{which is possible since $\delta \in (0,1/2)$.} that $(2r-p)(1-\delta/2) \geq r $ to have \eqref{eq_petite_restriction} and so to conclude this proof. 
\end{proof}

\section{Appendix}
We collect some useful estimates on convolution products.
\begin{lemma}
\label{lem_conv1}
For all $d\geq 1$, $a,b\geq 0$ such that $a+b>d$ there exists $C>0$ such that for all $x,y\in \mathbb{R}^d$ we have
$$
\sum_{k\in \mathbb{Z}^d} \langle k \rangle^{-a}  \langle k-x \rangle^{-b} \langle k+ y\rangle^{-b} \leq C \langle x+y\rangle^{-b}
$$
\end{lemma}
\begin{proof} Applying the triangle inequality in the Euclidean space $\mathbb{R}^{d+1}$, we get
$$
 \langle x+y \rangle \leq \langle k-x \rangle + \langle k+ y\rangle.
$$
Therefore, for all $k\in \mathbb{Z}^d$, we have either $\langle k-x \rangle\geq \frac12 \langle x+y \rangle$ or $\langle k+y \rangle\geq \frac12 \langle x+y \rangle$. In any case, we deduce that
$$
\sum_{k\in \mathbb{Z}^d} \langle k \rangle^{-a}  \langle k-x \rangle^{-b} \langle k+ y\rangle^{-b} \leq  2^{b}  \langle x+y\rangle^{-b} \left(\sum_{k\in \mathbb{Z}^d} \langle k \rangle^{-a}  \langle k-x \rangle^{-b}  + \sum_{k\in \mathbb{Z}^d} \langle k \rangle^{-a}  \langle k+y \rangle^{-b} \right).
$$
Finally, since $a+b>d$, we control these sums (independently of $x$ and $y$) applying the H\"older inequality with $p = \frac{a+b}a$ (and so $p'= \frac{a+b}b$). 
\end{proof}

\begin{lemma}
\label{lemma_convol}
For all $\alpha>1$, there exists $C>0$ such that for all $n\in \mathbb{Z}^r$ with $r\geq 2$, we have
$$
\sum_{k_1+\dots+k_r = 0} \prod_{j=1}^r \langle n_j - k_j  \rangle^{-\alpha} \leq C^{r-1}  \langle n_1 + \dots + n_r  \rangle^{-\alpha}
$$
\end{lemma}
\begin{proof} We proceed by induction on $r$.

\medskip

\noindent $\bullet$ \emph{Initialization :} $r =2$. It is a consequence of Lemma \ref{lem_conv1} with $x= n_1$, $y=n_2$, $a=0$, $b=\alpha$.

\medskip

\noindent $\bullet$ \emph{Induction step.} We assume that the estimate hold for an index $r\geq 2$. Applying the induction hypothesis and the change of variable $k_r \leftarrow k_r-m$ we deduce that
$$,
\forall m \in \mathbb{Z}, \ f_n(m):=\sum_{k_1+\dots+k_r = m} \prod_{j=1}^r \langle n_j - k_j  \rangle^{-\alpha} \leq C^{r-1} \langle n_1 + \dots + n_r -m \rangle^{-\alpha}.
$$
As a consequence, since the convolution product is associative, we deduce that
\begin{equation*}
\begin{split}
\sum_{k_1+\dots+k_{r+1} = 0} \prod_{j=1}^{r+1} \langle n_j - k_j  \rangle^{-\alpha} &= \sum_{m+k_{r+1} = 0} f_n(m)  \langle n_{r+1} - k_{r+1}  \rangle^{-\alpha} \\
&\leq C^{r-1} \sum_{m+k_{r+1} = 0} \langle n_1 + \dots + n_r -m \rangle^{-\alpha} \langle n_{r+1} - k_{r+1}  \rangle^{-\alpha}.
\end{split}
\end{equation*}
Finally applying the estimate we proved for $r=2$, we conclude this induction.
\end{proof}

Finally, we provide a lemma about the representation of low order fractional Sobolev spaces (it is probably well known but we did not find it in the literature).
\begin{lemma}
\label{lem_interpol}
There exists $C>0$ such that for all $s\in [1,2]$ and all $u\in H^s(0,\pi;\mathbb{C}) \cap H^1_0(0,\pi;\mathbb{C})$ we have
$$
\big\| \Psi^{\mathrm{Dir}}(u) \big\|_{h^s} \leq C \| u\|_{H^s}
$$
where $\Psi^{\mathrm{Dir}} : L^2(0,\pi;\mathbb{C}) \to \ell^2(\mathbb{N}^*;\mathbb{C})$ is defined by $\Psi^{\mathrm{Dir}}_n(u) =  \int_0^\pi \sin(nx) u(x) \mathrm{d}x.$
\end{lemma}
\begin{proof} We proceed by complex interpolation (whose definition and main properties are recalled, for example, in \cite{Tri78,Agra15}).

 $\triangleright$ If $u\in H^1_0(0,\pi;\mathbb{C})$, by integration by part, since $u(0)=u(\pi)=0$, we have 
$$
\Psi^{\mathrm{Dir}}_n(u) = \big[ -\cos(nx) u(x) \big]_0^\pi + n^{-1} \int_0^\pi \cos(nx) \partial_x u(x) \mathrm{d}x  =  n^{-1} \int_0^\pi \cos(nx) \partial_x u(x) \mathrm{d}x.
$$
 The functions $\cos(nx)$ being orthogonal in $L^2$ and of norm $\sqrt{\pi/2}$, we deduce that $\| \Psi^{\mathrm{Dir}}_n(u) \|_{h^1} \lesssim \|u\|_{H^1}$.
 
 $\triangleright$ If $u\in H^2(0,\pi;\mathbb{C}) \cap H^1_0(0,\pi;\mathbb{C})$, realizing a second by integration by part, we have 
$$
\Psi^{\mathrm{Dir}}_n(u) = \big[ n^{-1} \sin(nx) \partial_x u(x) \big]_0^\pi - n^{-2} \int_0^\pi \sin(nx) \partial_x^2 u(x) \mathrm{d}x  =   - n^{-2} \int_0^\pi \sin(nx) \partial_x^2 u(x) \mathrm{d}x.
$$
Therefore, as previously, we deduce that $\| \Psi^{\mathrm{Dir}}_n(u) \|_{h^2} \lesssim \|u\|_{H^2}$.

 $\triangleright$ If $s\in (1,2)$, we have $h^s = [h^1,h^2]_{s-1}$ (Theorem page 130 of \cite{Tri78}) and $H^s \cap H^1_0 =  [H^1_0,H^2\cap H^1_0]_{s-1}$ (Theorem 13.2.2 page 198 of \cite{Agra15}). Therefore, the natural property of the interpolation (see e.g. Thm 13.2.1 page 197 of \cite{Agra15}) provides the continuity estimate we wanted to prove.
\end{proof}

\end{document}